\pdfoutput=1
\documentclass[a4paper,10pt]{amsart}
\usepackage{
 amsmath,
 amscd,
 amssymb,
 mathrsfs,
 mathtools,
 fullpage,
 enumerate,
 thmtools,
 tikz-cd
}
\usepackage[l2tabu, orthodox]{nag}
\usepackage[pagebackref]{hyperref}

\newenvironment{smatrix}{\left( \begin{smallmatrix} } {\end{smallmatrix} \right) }
\newcommand{\stbt}[4]{\begin{smatrix}#1 & #2 \\ #3 & #4\end{smatrix}}
\newcommand{\dfour}[4]{\begin{smatrix}#1\\ &#2 \\ &&#3 \\ &&& #4\end{smatrix}}
\newcommand{\dthree}[3]{\begin{smatrix}#1\\ &#2 \\ &&#3\end{smatrix}}

\theoremstyle{plain}
 \newtheorem{theorem}{Theorem}[section]
 \newtheorem{lemma}[theorem]{Lemma}
 \newtheorem{proposition}[theorem]{Proposition}
 \newtheorem{corollary}[theorem]{Corollary}
 \newtheorem{definition}[theorem]{Definition}
 \newtheorem{notation}[theorem]{Notation}
 \newtheorem{ltheorem}{Theorem} 
 
 \newtheorem{hypothesis}[theorem]{Hypothesis}

\theoremstyle{remark}
 \declaretheorem[name=Remark,sibling=theorem,qed={\lower-0.3ex\hbox{$\diamond$}}]{remark}

\DeclareMathOperator{\GSp}{GSp}
\DeclareMathOperator{\Kl}{Kl}

\DeclareMathOperator{\Sym}{Sym}
\DeclareMathOperator{\Hom}{Hom}
\DeclareMathOperator{\Spec}{Spec}
\DeclareMathOperator{\Fil}{Fil}
\DeclareMathOperator{\GL}{GL}

\DeclareMathOperator{\Gr}{Gr}

\DeclareMathOperator{\Gal}{Gal}
\DeclareMathOperator{\Ind}{Ind}

\DeclareMathOperator{\Sieg}{S}
\DeclareMathOperator{\Spf}{Spf}
\DeclareMathOperator{\diag}{diag}
\DeclareMathOperator{\sgn}{sgn}

\renewcommand*{\det}{\qopname\relax o{det}}

\newcommand{\ord}{\mathrm{ord}}
\newcommand{\can}{\mathrm{can}}

\newcommand{\cX}{\mathcal{X}}

\newcommand{\cE}{\mathcal{E}}
\newcommand{\cF}{\mathcal{F}}
\newcommand{\cG}{\mathcal{G}}
\newcommand{\cH}{\mathcal{H}}
\newcommand{\cO}{\mathcal{O}}

\newcommand{\cS}{\mathcal{S}}
\newcommand{\cU}{\mathcal{U}}
\newcommand{\cM}{\mathcal{M}}

\newcommand{\cW}{\mathcal{W}}

\newcommand{\cT}{\mathcal{T}}

\newcommand{\fX}{\mathfrak{X}}
\newcommand{\fF}{\mathfrak{F}}

\newcommand{\fIG}{\mathfrak{IG}}

\newcommand{\Ch}{\operatorname{ch}}

\newcommand{\CC}{\mathbf{C}}
\newcommand{\QQ}{\mathbf{Q}}
\newcommand{\RR}{\mathbf{R}}
\newcommand{\ZZ}{\mathbf{Z}}
\renewcommand{\AA}{\mathbf{A}}
\newcommand{\FF}{\mathbf{F}}

\newcommand{\Qp}{\QQ_p}
\newcommand{\Zp}{\ZZ_p}

\newcommand{\Af}{\AA_{\mathrm{f}}}
\newcommand{\Pif}{\Pi_{\mathrm{f}}}

\newcommand{\dR}{\mathrm{dR}}

\newcommand{\f}{\mathrm{f}}

\newcommand{\into}{\hookrightarrow}
\newcommand{\onto}{\twoheadrightarrow}

\newcommand{\htimes}{\mathop{\hat\otimes}}

\renewcommand{\le}{\leqslant}
\renewcommand{\leq}{\leqslant}
\renewcommand{\ge}{\geqslant}
\renewcommand{\geq}{\geqslant}

\newcommand{\half}{\frac{1}{2}}
\renewcommand{\Re}{\operatorname{Re}}

\numberwithin{equation}{section}

\author{David Loeffler}
\address{David Loeffler, Mathematics Institute\\
 University of Warwick\\
 Coventry CV4 7AL, UK.}
\email{d.a.loeffler@warwick.ac.uk}
\urladdr{\href{http://orcid.org/0000-0001-9069-1877}{0000-0001-9069-1877}}

\author{Vincent Pilloni}
\address{Vincent Pilloni, CNRS\\
 Ecole Normale Sup\'erieure de Lyon\\
Lyon, France.}
\email{vincent.pilloni@ens-lyon.fr}

\author{Christopher Skinner}
\address{Christopher Skinner\\
 Department of Mathematics\\
 Princeton University\\
 Princeton, NJ 08544-1000\\
 USA.}
\email{cmcls@princeton.edu}

\author{Sarah Livia Zerbes}
\address{Sarah Livia Zerbes, Department of Mathematics\\
 University College London\\
London WC1E 6BT, UK.}
\email{s.zerbes@ucl.ac.uk}
\urladdr{\href{http://orcid.org/0000-0001-8650-9622}{0000-0001-8650-9622}}

\thanks{Supported by the following grants: Royal Society University Research Fellowship ``$L$-functions and Iwasawa theory'' (Loeffler); Simons Investigator Grant \#376203 from the Simons Foundation (Skinner); ERC Consolidator Grant ``Euler systems and the Birch--Swinnerton-Dyer conjecture'' (Zerbes).}

\title[Higher Hida theory]{Higher Hida theory and $p$-adic $L$-functions for $\GSp_4$}
\date{\today}

\begin{document}

 \begin{abstract}
  We use the ``higher Hida theory'' recently introduced by the second author to $p$-adically interpolate periods of non-holomorphic automorphic forms for $\GSp_4$, contributing to coherent cohomology of Siegel threefolds in positive degrees. We apply this new method to construct $p$-adic $L$-functions associated to the degree 4 (spin) $L$-function of automorphic representations of $\GSp_4$, and the degree 8 $L$-function of $\GSp_4 \times \GL_2$.
 \end{abstract}

 \maketitle

 \setcounter{tocdepth}{1}
 \tableofcontents

 \section{Introduction}

  \subsection{Background}

   Several of the most important open problems in mathematics involve the arithmetic significance of special values of $L$-functions; and a major role in work on these problems is played by $p$-adic $L$-functions. There are, essentially, two main approaches to constructing these objects: ``topological'' constructions via Betti cohomology of symmetric spaces (such as the theory of modular symbols for $\GL_2$); or constructions of a more ``algebro-geometric'' nature, using Shimura varieties and sections of coherent sheaves over them, as in Hida and Panchishkin's construction of $p$-adic $L$-functions for $\GL_2 \times \GL_2$.

   The most powerful applications of $p$-adic $L$-functions are in cases where one can relate the $p$-adic $L$-function to a family of classes in Galois cohomology (an Euler system). Results of this kind are known as \emph{explicit reciprocity laws}, and have hugely important consequences, leading to cases of the Birch--Swinnerton-Dyer conjecture and the Bloch--Kato conjecture, as in the work of Kato \cite{kato04} and more recently \cite{darmonrotger16}, \cite{KLZ17} and others. However, a prerequisite for such an explicit reciprocity law is to have a construction of the $p$-adic $L$-function by algebro-geometric techniques, which can be related to Galois representations in \'etale cohomology.

   At this point, one hits a serious obstacle. There are many integral formulae known which relate values of $L$-functions to cohomology of automorphic sheaves on Shimura varieties. However, these automorphic sheaves can have cohomology in a range of degrees, and the $L$-value formulae that are relevant in Euler system settings always involve cohomology classes near the middle of the range of possible degrees. On the other hand, the established techniques for studying $p$-adic variation of these objects are only applicable to sections, i.e.~to cohomology in degree 0 (or, via Serre duality, to cohomology in the highest degree). This incompatibility is a fundamental limitation in the theory as it presently stands: because of this, all the reciprocity laws known so far relate to Shimura varieties which have small dimension, or which factorise as a product of two simpler subvarieties of approximately equal dimension. In particular, the previously-known techniques are not sufficient to prove an explicit reciprocity law for the Euler system constructed in \cite{LSZ17} for automorphic representations of $\GSp_4$; this is the major obstacle that must be solved in order to use this new Euler system to prove the Bloch--Kato conjecture for automorphic motives attached to this group. The same difficulty arises for several other recently-discovered Euler systems, such as those of \cite{LLZ18, CLR19, LSZ-unitary}. 

  \subsection{Our results} In this paper, we develop a new algebro-geometric approach to constructing $p$-adic $L$-functions which resolves this difficulty, and apply it to construct $p$-adic spin $L$-functions attached to automorphic representations of the group $\GSp_4$. This is the basis for the sequel to this paper \cite{LZ20} by the first and last authors, in which we shall prove an explicit reciprocity law relating the $p$-adic $L$-function of the present paper to the Euler system of \cite{LSZ17}.  We expect that the techniques of this paper should also be applicable to $p$-adic interpolation of many other automorphic $L$-functions beyond the examples we study here. For instance, the cases of the Asai $L$-function of a quadratic Hilbert modular form, and the degree 6 $L$-function of an automorphic representation of $\operatorname{GU}(2, 1)$, will be treated in forthcoming work.

  Our new approach to constructing $p$-adic $L$-functions relies crucially on the \emph{higher Hida theory} introduced by the second author in \cite{pilloni17}. This gives a theory which $p$-adically interpolates coherent cohomology of Siegel threefolds in positive degrees, while conventional Hida theory only sees $H^0$. Since Harris has shown in \cite{harris04} that the critical values of the spin $L$-function can be interpreted as cup-products involving coherent cohomology in degrees 1 and 2, this gives a path by which to approach the $p$-adic interpolation of spin $L$-values.

  Let us now state our results a little more precisely. Let $\Pi$ be a cuspidal automorphic representation of $\GSp_4(\AA_\QQ)$ which is non-CAP, globally generic, and cohomological with coefficients in the algebraic representation of highest weight $(r_1, r_2)$, for some integers $r_1 \ge r_2 \ge 0$. For technical reasons we need to suppose that $r_2 \ge 1$. Let $p$ be a prime such that $\Pi_p$ is unramified and Klingen-ordinary (with respect to some choice of embedding $\overline{\QQ} \into \overline{\QQ}_p$). Let $E \subset \CC$ be the number field generated by the Hecke eigenvalues of $\Pi$.

  \begin{ltheorem}
   \label{lthm:A}
   Suppose that either $d = r_1 - r_2 \ge 1$ or that $d = 0$ and Hypothesis \ref{hyp:non-vanish} holds. Then there exist two constants $\Omega^+_\Pi$, $\Omega^-_\Pi \in \CC^\times$, uniquely determined modulo $E^\times$, and a $p$-adic measure $\mu_{\Pi}$ on $\Zp^\times$, such that for all Dirichlet characters $\chi$ of $p$-power conductor and all integers $0 \le a \le d$, we have
   \[ \int_{\Zp^\times} x^a \chi(x)\, \mathrm{d}\mu_\Pi(x) =  (-1)^a R_p(\Pi, a, \chi) \cdot \frac{\Lambda(\Pi \otimes \chi^{-1}, \tfrac{1-d}{2} + a)}{\Omega_\Pi^{\pm}}, \]
   where $R_p(\Pi, a, \chi)$ is an explicit non-zero factor, and the sign $\pm$ denotes $(-1)^a \chi(-1)$.
  \end{ltheorem}

  \begin{ltheorem}
   \label{lthm:B}
   Let $\sigma$ be a cuspidal automorphic representation of $\GL_2(\AA_\QQ)$ generated by a holomorphic modular form of weight $\ell$, with $1 \le \ell \le r_1 - r_2 + 1$, and level coprime to $p$ and to the primes of ramification of $\Pi$. Let $d' = r_1 - r_2 - \ell + 1 \ge 0$.

   Then there is a constant $\Omega_\Pi^W$, uniquely determined modulo $E^\times$, and a $p$-adic measure $\mu_{\Pi \times \sigma}$ on $\Zp^\times$, such that for all Dirichlet characters $\chi$ of $p$-power conductor and all integers $0 \le a \le d'$, we have
   \[ \int_{\Zp^\times} x^a \chi(x)\, \mathrm{d}\mu_{\Pi \times \sigma}(x) =  R_p(\Pi \otimes \sigma, a, \chi) \cdot \frac{\Lambda(\Pi \otimes \sigma \otimes \chi^{-1}, \tfrac{1-d'}{2} + a)}{\Omega_\Pi^{W}}, \]
   where $R(\Pi \otimes \sigma, a, \chi)$ is an explicit non-zero factor. If the hypotheses of Theorem A are satisfied then we have $\Omega_\Pi^{W} = \Omega_\Pi^+ \cdot \Omega_\Pi^-$.
  \end{ltheorem}

  In fact the exact statements are a little stronger, although more complicated to state -- we refer to \S \ref{sect:final} below for the details. In both theorems, $\Lambda(-, s)$ denotes the completed $L$-function (including its Archimedean $\Gamma$-factors); the factors $R_p(-)$ are local Euler factors at $p$, consistent with those predicted by general conjectures of Panchishkin and Coates--Perrin-Riou; and the range of values for $a$ corresponds to the whole interval of critical values of the $L$-functions concerned. (In the present work we have not attempted to define a canonical normalisation of the periods $\Omega_\Pi^{W}$ and $\Omega_\Pi^{\pm}$ up to $p$-adic units, although this would be desirable for Iwasawa-theoretic applications. See Remark \ref{rmk:intperiods} for some comments on this issue.)

  \subsection{Outline of the construction}
  The first step in our strategy is to extend the higher Hida theory of \cite{pilloni17}, which was developed there for applications to ``non-regular'' weights (limits of holomorphic discrete series), to also cover regular weights. This is carried out in \S\ref{sect:vincent}. The chief new result here is that for regular weights lying in the relevant Weyl chamber, the perfect complex representing the ordinary part of $p$-adic cohomology is concentrated purely in degree 1.

  In \S \ref{sect:p-adic-pushfwd} we prove our second main technical result: a ``functoriality'' property for higher Hida theory relating the group $G = \GSp_4$ and its subgroup $H = \GL_2 \times_{\GL_1} \GL_2$. This shows that $p$-adic families of modular forms for the subgroup $H$ can be pushed forward to $p$-adic families in $H^1$ of $G$. This allows us to $p$-adically interpolate period integrals of the form
  \begin{equation}
   \label{eq:period}
   \int_{Y_H(\CC)} F(\iota(h)) f_1(h_1) f_2(h_2)\, \mathrm{d}h,
  \end{equation}
  where $Y_H$ is a Shimura variety for $H$ (of some suitable level); $\iota$ denotes the embedding $H \into G$; $h_1$ and $h_2$ are the projections of $h \in H(\AA)$ to the two $\GL_2$ factors; $F$ is an automorphic form for $G$ contributing to $H^2$ of the Siegel threefold; and $f_1$ and $f_2$ are holomorphic modular forms varying in $p$-adic families.

  However, for our desired applications, we also need to consider integrals of the above shape in which one of the $f_i$ is not holomorphic but only \emph{nearly-holomorphic} (the image of a holomorphic form under a power of the Maass--Shimura differential operator). Hence the next step in our strategy, carried out in \S \ref{sect:nearly}, is to develop a theory of ``nearly'' coherent cohomology for $G$, and an analogue of the $p$-adic unit-root splitting to relate these spaces to higher Hida theory. These are analogues for higher Hida theory of results recently proved by Zheng Liu in the setting of classical Hida theory for symplectic groups \cite{liu-nearly}.

  (Note that this unit-root splitting is not overconvergent -- it does not extend to any strict neighbourhood of the $p$-rank $\ge 1$ locus. Thus, although a ``higher Coleman theory'' for the Klingen-finite-slope, but not necessarily Klingen-ordinary, part of $H^1$ is developed in \cite{pilloni17}, our results do not extend straightforwardly to the finite-slope case; it would be necessary to develop a theory of nearly-overconvergent cohomology in this setting, analogous to the theory recently developed in \cite{andreattaiovita21} for $\GL_2$. This is surely possible, but lies beyond the scope of the present work.)

  Finally, it remains to show that values of $L$-functions can be described by period integrals of the form \eqref{eq:period}. For this, we use two integral formulae of Rankin--Selberg type: a 1-parameter integral formula due to Novodvorsky for the degree 8 $L$-function of $\GSp_4 \times \GL_2$, and a 2-parameter integral formula giving the product of two copies of the $\GSp_4$ $L$-function, which is an extension of work of Piatetski-Shapiro. In \S\S \ref{sect:localintegrals} and \ref{sect:globalintegrals} we define the local and global versions of these integrals, and evaluate the local integrals at $p$ and at $\infty$ for the specific choices of test data that arise in our construction. Finally, in \S\ref{sect:final} we put together all of the above pieces to prove our main theorems.

  (These formulae should both be seen as ``degenerate cases'' of a third, presently conjectural, integral formula: the Gan--Gross--Prasad conjecture predicts that if $f_1$ and $f_2$ are both cuspidal, then the integral \eqref{eq:period} should be related to the square root of the central value of the degree 16 $L$-function for $\GSp_4 \times \GL_2 \times \GL_2$. So, assuming the GGP conjecture, our methods give $p$-adic $L$-functions interpolating these square roots as $f_1$ and $f_2$ vary in cuspidal Hida families. However, we shall not treat this case in detail here, for reasons of space.)

  \subsection{Relation to other work} We note that Theorem A can be seen as a consequence of a theorem of Dimitrov--Janusewski--Raghuram \cite{DJR20} (applied to the lifting of $\Pi$ to $\GL_4$). However, our proof of the theorem is very different: their construction is of a topological nature, using Betti cohomology of a symmetric space associated to $\GL_4$, while ours is algebro-geometric, using coherent cohomology of the Shimura variety associated to $G$. 
  This allows us to prove Theorem A in parallel with Theorem B, which does not seem to be accessible using the methods of \cite{DJR20}. More importantly, as mentioned above, working on the $\GSp_4$ Shimura variety, and using algebraic rather than topological methods, are vital in order to relate the $p$-adic $L$-functions we construct to the Euler system of \cite{LSZ17}. This is pursued further in the sequel paper \cite{LZ20} by the first and last authors, in which we relate the values of the $p$-adic $L$-function of Theorem A (at points outside the range of interpolation) to the syntomic regulators of Euler system classes, and apply this to deduce instances of the Bloch--Kato conjecture and Iwasawa main conjecture for $\GSp_4$.
%


  One can also use this technique to construct multi-variable $p$-adic $L$-functions. For instance, in the setting of Theorem B, one can allow $f_1$ to vary through a Hida family of cusp forms, giving a 2-variable $p$-adic $L$-function in which both the weight of $f_1$ and the cyclotomic twist are varying. As a special case of this, taking $f_1$ to be a CM-type family, one obtains a measure interpolating the $L$-values of twists of $\Pi$ by Gr\"ossencharacters of an imaginary quadratic field. In an alternative direction, one can take both $f_1$ and $f_2$ to be cusp forms varying in Hida families, giving a 2-variable measure interpolating automorphic periods of Gan--Gross--Prasad type for the pair $(\operatorname{SO}_4, \operatorname{SO}_5$); this is a higher-rank analogue of the triple-product $L$-function studied by Harris--Tilouine and Darmon--Rotger \cite{harristilouine01, darmonrotger14}, which interpolates Gan--Gross--Prasad periods for the pair $(\operatorname{SO}_3, \operatorname{SO}_4)$. However, for reasons of space we shall pursue these generalisations in a subsequent paper.

  \subsection*{Acknowledgements} This project developed out of conversations at the workshop ``Motives, Galois Representations and Cohomology'' hosted by the Institute of Advanced Study, Princeton, in November 2017. The three of us who attended this workshop (Loeffler, Pilloni and Zerbes) are very grateful to the Institute and the organisers of the workshop for inviting them to participate. We are also grateful for the support of the Centre Bernoulli in Lausanne, and the Morningside Centre in Beijing, who hosted various subsets of the authors for extended visits during the preparation of this paper.

  We are also grateful to Michael Harris, Mark Kisin, and Kai-Wen Lan for helpful remarks on the topic of toroidal compactifications; and to Laurent Moret-Bailly for pointing out a valuable reference regarding properties of multiplicative group schemes (in answer to a question of ours on the website ``MathOverflow''). Finally, we thank the several anonymous referees for their valuable comments on previous drafts of this paper.


\section{Preliminaries: Groups and Shimura varieties}


\subsection{Groups}
\label{sect:groups}

 We denote by $G = \GSp_4$ the group scheme (over $\ZZ$) associated to the skew-symmetric matrix $J = \begin{smatrix} &&& 1\\ && 1 &\\ & -1 &&\\ -1 &&&\end{smatrix}$. The standard Siegel and Klingen parabolics are then given by
 \begin{align*}
 P_{\Sieg}&=
 \begin{smatrix}
 \star & \star & \star & \star \\
 \star & \star & \star & \star \\
 &       & \star & \star \\
 &       & \star & \star
 \end{smatrix}, &
 P_{\Kl}&=
 \begin{smatrix}
 \star & \star & \star & \star \\
 & \star & \star & \star \\
 & \star & \star & \star \\
 &       &       & \star
 \end{smatrix}.
 \end{align*}
 We write $M_{\Sieg}$, $M_{\Kl}$ for the standard (block-diagonal) Levi subgroups of $P_{\Sieg}$ and $P_{\Kl}$, and $T$ for the diagonal maximal torus.

 Let $H=\GL_2\times_{\GL_1}\GL_2$. We define an embedding $\iota: H\hookrightarrow G$ by
 \[
 \left(\begin{pmatrix} a & b \\ c & d\end{pmatrix}, \begin{pmatrix} a' & b'\\ c' & d'\end{pmatrix}\right) \mapsto \begin{smatrix} a &&& b\\ & a' & b' & \\ & c' & d' & \\c &&&d \end{smatrix}.
 \]

\subsection{Dirichlet characters}
 \label{sect:dirichlet}

 If $\chi: (\ZZ / N\ZZ)^\times \to R^\times$ is a finite-order character, for some ring $R$ and integer $N$, then we write $\widehat{\chi}$ for the character of $\QQ^\times \backslash \AA^\times / \RR^{\times}_{>0}$ satisfying $\widehat{\chi}(\varpi_\ell) = \chi(\ell)$ for primes $\ell \nmid N$, where $\varpi_\ell$ is a uniformizer at $\ell$. (Note that that the restriction of $\widehat{\chi}$ to $\widehat{\ZZ}^\times \subset \Af^\times$ is the \emph{inverse} of $\chi$.) If $\chi$ is $\CC$-valued and $\Pi$ is a representation of $G(\AA)$, we write $\Pi \otimes \chi$ for the twist of $\Pi$ by the composite of $\widehat{\chi}$ and the symplectic multiplier $G \to \mathbf{G}_m$, as in the statement of Theorems \ref{lthm:A} and \ref{lthm:B} above.

 \subsubsection*{Gauss sums and epsilon-factors} If $\chi$ is a $\CC$-valued Dirichlet character of conductor $N$, we let $G(\chi) \coloneqq \sum_{a \in (\ZZ / N\ZZ)^\times} \chi(a) \exp(2\pi i a/N)$ be the classical Gauss sum, understood as 1 if $\chi$ is trivial.

 If $\psi$ is the additive character of $\AA / \QQ$ whose restriction to $\RR$ is $x \mapsto \exp(-2\pi i x)$, and $\chi$ is a Dirichlet character of $p$-power conductor for some prime $p$, then the local epsilon-factor $\varepsilon(\widehat{\chi}_p, \psi_p)$ is $G(\chi)$.

\subsection{Shimura varieties}

 \subsubsection{Open varieties over $\QQ$}

 Let $K$ be a neat open compact subgroup of $\GSp_4(\Af)$. Denote by $Y_{G, \QQ}$ the canonical model over $\QQ$ of the level $K$ Shimura variety. This is a smooth quasiprojective threefold, endowed with an isomorphism of complex manifolds
 \begin{equation}
  \label{eq:cplxunif}
  Y_{G, \QQ}(\CC) \cong G(\QQ)_+ \backslash \left[ \cH_2 \times G(\Af) / K\right],
 \end{equation}
 where $\cH_2$ is the genus 2 Siegel upper half-space. It can be identified with the moduli space of abelian surfaces endowed with principal polarisations and level $K$ structures.

 We write $Y_{H, \QQ}$ for the canonical $\QQ$-model of the Shimura variety for $H$ of level $K_H = K \cap H(\Af)$, which is a moduli space for (ordered) pairs of elliptic curves with level structure. (Note that if $K = K_G(N)$ is the principal congruence subgroup, then $Y_{H, \QQ}$ is the fibre product of two copies of the modular curve $Y(N)$ over $\mu_N$.) There is a morphism of algebraic varieties
 \begin{equation}
 \label{eq:iota1}
 \iota: Y_{H, \QQ} \to Y_{G, \QQ},
 \end{equation}
 with image a closed codimension 1 subvariety of $Y_{G, \QQ}$ (a \emph{Humbert surface}). If $K$ is contained in the principal congruence subgroup $K_G(N)$ for some $N \ge 3$, then $\iota$ is a closed immersion, by \cite[Proposition 5.3.1]{LSZ17}.

 \subsubsection{Integral models and levels at $p$}

 Let $p$ be prime, and suppose $K = K^p K_p$ with $K^p \subset G(\Af^p)$ and $K_p = G(\Zp)$. Then $Y_{H, \QQ}$ and $Y_{G, \QQ}$ have canonical smooth models over $\ZZ_{(p)}$ (parametrising abelian surfaces over $\ZZ_{(p)}$-algebras with appropriate additional structures) which we denote simply by $Y_{H}$ and $Y_{G}$. The morphism $\iota$ extends to a morphism $Y_H \to Y_G$; if $K^p$ is contained in $K_G(N)$ for some $N \ge 3$ coprime to $p$, which we shall assume, then by the same arguments used for the generic fibre above, one can check that this is a closed immersion of $\ZZ_{(p)}$-schemes.

 For $n \ge 1$, we let $\Kl(p^n) = \{ g \in G(\Zp): g \pmod{p^n} \in P_{\Kl}\}$, and we denote by $Y_{G, \Kl}(p^n)_{\QQ}$ the canonical $\QQ$-model of the Shimura variety of level $K^p \Kl(p^n)$. These do not have natural smooth $\ZZ_{(p)}$-models.

 \subsubsection{Toroidal compactifications}
 \label{sect:toroidal}

 As in \cite{faltingschai}, we may define arithmetic toroidal compactifications of $Y_G$, depending on a choice of combinatorial datum $\Sigma$ (a rational polyhedral cone decomposition, or ``rpcd'' for short). We shall restrict attention to rpcd's which are ``good'' in the sense of \cite[\S 5.3.2]{pilloni17}. A choice of good rpcd $\Sigma$ gives rise to an open embedding of $\Zp$-schemes $Y_G \into X_G^\Sigma$ with the following properties:
 \begin{itemize}
  \item $X_G^\Sigma$ is smooth and projective over $\Zp$.
  \item The boundary $D_G^\Sigma = X_G^\Sigma - Y_G$ is a relative Cartier divisor.
 \end{itemize}
 Any such compactification $X_G^\Sigma$ maps naturally to the minimal compactification $X_G^{\mathrm{min}}$. If $\Sigma'$ is a refinement of $\Sigma$, then there is a projective morphism $\pi_{(\Sigma', \Sigma)}: X_G^{\Sigma'} \to X_G^\Sigma$ (compatible with the maps from $Y_G$ and to $X_G^{\mathrm{min}}$); and we have $\pi_{(\Sigma', \Sigma)}^*(I_G^\Sigma) = I_G^{\Sigma'}$, where $I_G^\Sigma$ is the ideal sheaf of the boundary $D_G^\Sigma$ in $X_G^\Sigma$.

 Over $X_G^\Sigma$ there is a canonical semiabelian variety $A_G^\Sigma$, extending the universal abelian variety over $Y_G$; and a $P_S$-torsor $\mathcal{T}_G^\Sigma$, parametrising trivialisations, as a filtered vector bundle, of the canonical extension to $X_G^\Sigma$ of the first relative de Rham cohomology of $A_G / Y_G$ (see \cite[\S 4.3]{madapusipera}). These are all compatible with refining the cone decomposition $\Sigma$.

 The same statements as above hold \emph{mutatis mutandis} for the Shimura varieties of level $K^p \Kl(p^n)$, although the resulting $\Zp$-models are no longer smooth. Similarly, we can define toroidal compactifications for $H$ in place of $G$. In this case, there is an ``optimal'' choice of $\Sigma$, for which the map $X_H^\Sigma \to X_H^{\mathrm{min}}$ is an isomorphism, but this is not the only possible choice.

 \begin{remark}
  If $K_H = K_1 \times K_2$ is the fibre product of subgroups of the $\GL_2$ factors, then $Y_H$ is a subset of the components of $Y_1 \times Y_2$, where $Y_i$ is the modular curve of level $K_i$; and $X_H^{\min}$ is the product of their compactifications $X_i$. Any other toroidal compactification $X_H^{\Sigma}$ is obtained from $X_H^{\mathrm{min}}$ by blowing up at some sheaf of ideals supported at points of the form $\{\mathrm{cusp}\} \times \{\mathrm{cusp}\}$.
 \end{remark}

 \subsubsection{Functoriality of the compactifications}
 \label{sect:toroidalfunct}

  As explained in \cite[Proposition 3]{harris89}, any rpcd $\Sigma$ for $G$ uniquely determines an rpcd $\iota^*(\Sigma)$ for $H$.
  It has recently been shown by Lan \cite{lan19} that if $K^p$ is sufficiently small, one may choose an rpcd $\Sigma$ for $G$ such that:
  \begin{itemize}
   \item both $\Sigma$ and $\iota^*(\Sigma)$ are good (for $G$ and $H$ respectively),
   \item the map $\iota$ extends to a closed embedding of $\ZZ_{(p)}$-schemes $\iota_{\Sigma}: X_H^{\iota^*(\Sigma)} \into X_G^{\Sigma}$.
  \end{itemize}
 We shall fix a choice of $\Sigma$ satisfying this condition. It follows from the construction of the torsors $\cT_G^{\Sigma}$ and $\cT_H^{\Sigma}$ that we have an isomorphism
 \begin{equation}
 \label{eq:redstructure}
 \iota_{\Sigma}^*\left( \cT_G^{\Sigma} \right) = P_{\Sieg} \times^{B_H} \cT_H^{\Sigma},
 \end{equation}
 so that $\cT_H^{\Sigma}$ is a reduction of structure of $\iota_{\Sigma}^*\left( \cT_G^{\Sigma} \right)$ from a $P_{\Sieg}$-torsor to a $B_H$-torsor. We also note the inclusion of ideal sheaves
 \begin{equation}
 \label{eq:idealsheaf}
 \iota_{\Sigma}^*\left(I_G^\Sigma\right) \subseteq I_H^{\Sigma}.
 \end{equation}

 We shall frequently omit the superscript $\Sigma$ from the notation $X_G^\Sigma$, $D_G^\Sigma$ etc when there is no risk of ambiguity (and sometimes the subscripts $G, H$ as well).

 \begin{remark}
  It seems likely that one can choose $\Sigma$ in such a way that \eqref{eq:idealsheaf} is an equality, but we have not verified this.
 \end{remark}

\subsection{Representations and coefficient sheaves}

\subsubsection{Weights and representations}

As in \cite[\S 5.1.1]{pilloni17}, the character group $X^\bullet(T)$ can be identified with the group of triples $(r_1, r_2; c) \in \ZZ^3$ such that $c = r_1 + r_2 \bmod 2$, by defining $\lambda(r_1, r_2; c)$ as the unique character of $T$ such that
\[ \dfour{st_1}{st_2}{s t_2^{-1}}{st_1^{-1}} \mapsto t_1^{r_1} t_2^{r_2} s^c. \]

The weights $\lambda(r_1, r_2; c)$ which are dominant for $M_{\Sieg}$ are those with $r_1 \ge r_2$; if $(r_1, r_2; c)$ satisfies this, we write $W_G(r_1, r_2; c)$ for the irreducible representation of $M_{\Sieg}$ with this highest weight. Those weights which also satisfy $r_2 \ge 0$ are dominant for $G$, and in this case we write $V_G(r_1, r_2; c)$ for the corresponding $G$-representation.

The torus $T$ is also a maximal torus of $H$, and we write $W_H(r_1, r_2; c)$ for the 1-dimensional representation of $M_{B_H} = T$ on which $T$ acts via $\lambda(r_1, r_2; c)$. If $r_1, r_2 \ge 0$, then we write $V_H(r_1, r_2; c)$ for the representation of $H$ of highest weight $\lambda(r_1, r_2; c)$; concretely, we have
\[
V_H(r_1, r_2; c) \coloneqq \left(\Sym^{r_1} \boxtimes \Sym^{r_2}\right) \otimes \det^{(c - r_1 - r_2)/2}.
\]

\begin{remark}
 We have changed notations for $G$-representations relative to \cite{LSZ17}; the $G$-representation denoted $V_{a,b}$ in \emph{op.cit.} is $V(a+b, a; 2a + b)$ in the new notation. The new notation has the advantage of greatly simplifying the formulae for the action of the Weyl group: the Weyl group conjugates of $\lambda(r_1, r_2; c)$ are the weights $\{ \lambda(\pm r_1, \pm r_2, c), \lambda(\pm r_2, \pm r_1, c)\}$.
\end{remark}

\subsubsection{Automorphic vector bundles}

 If $V$ is an algebraic representation of $P_{\Sieg}$ over $\ZZ_{(p)}$, then we have a vector bundle $[V]$ on $X_G$ defined by
\[ [V] \coloneqq V \times^{P_{\Sieg}} \cT_G, \]
where $\cT_G$ is the canonical $P_{\Sieg}$-torsor over $X_G$ defined above (see e.g.~\cite[\S 2.1]{liu-nearly}). The same applies, of course, to $G$ in place of $H$. This construction is obviously compatible with direct sums, tensor products, and duals.

 Over $\CC$ the automorphic bundles $[V]$ have a convenient interpretation in terms of the complex uniformisation \eqref{eq:cplxunif}. We can interpret $\cH_2$ as a $G(\RR)_+$-invariant open subset of $\cF_{\Sieg}(\CC)$, where $\cF_{\Sieg}$ is the Siegel flag variety $G / P_{\Sieg}$. Any $P_{\Sieg}$-representation $V$ gives rise to a $G(\CC)$-equivariant holomorphic vector bundle over $\cF_{\Sieg}$, and $[V]_{\CC}$ is the pullback of this to $X_{G, \CC}$. Note that the real-analytic vector bundle $[V]_{C^\infty}$ obtained from $[V]_{\CC}$ by forgetting the holomorphic structure depends only on the restriction of $V$ to $M_{\Sieg}$ (but the holomorphic structure genuinely does depend on $V$ as a $P_{\Sieg}$-representation).

\begin{remark}
 By construction, if we take $V = V(1, 0; 1)$ to be the defining 4-dimensional representation of $G$, then $[V]$ is the relative logarithmic de Rham cohomology sheaf $\mathcal{H}^1_{\dR}(A_G)$ of the universal semi-abelian surface $A_G / X_G$, and the evident two-step filtration of $V$ as a $P_{\Sieg}$-representation corresponds to the Hodge filtration of $\mathcal{H}^1_{\dR}(A_G)$. Note that this is slightly non-standard (it is more usual to send $V$ to the dual $\cH_1^{\dR}$), but is consistent with the conventions used for \'etale and motivic sheaves in \cite{LSZ17}. As a representation of $M_{\Sieg}$, $V$ splits as a direct sum of 2-dimensional subspaces; this corresponds to the \emph{Hodge splitting} of the bundle $[V]$ in the real-analytic category.
\end{remark}

\begin{notation}
 \label{not:sheaves}
 We write $\omega_G(r_1, r_2; c)$ for the vector bundle $[W_G(r_1, r_2; c)]$ on $X_G$, and similarly for $H$. We write $\omega_G(r_1, r_2) = \omega_G(r_1, r_2; r_1 + r_2 - 6)$, and similarly $\omega_H(r_1, r_2) = \omega_H(r_1, r_2; r_1 + r_2 - 4)$.
\end{notation}

\begin{remark}
 The sheaf $\omega_G(r_1, r_2)$ is isomorphic to $\Sym^{r_1 - r_2}(\omega_A) \otimes \det(\omega_A)^{r_2}$, where $\omega_A$ is the conormal bundle at the identity section of the semi-abelian scheme $A_G$ over $X_{G}$; in other words, $\omega_G(r_1, r_2)$ is the sheaf that was denoted by $\Omega^{(k, r)}$ in \cite{pilloni17}, for $(k, r) = (r_1 - r_2, r_2)$. Our choice of ``default'' normalisation for the central character coincides with that chosen in Remark 5.3.1 of \emph{op.cit.}.
\end{remark}

The Kodaira--Spencer construction identifies the sheaves of logarithmic differentials $\Omega^i_{X_G}(\log D_G)$ and $\Omega^i_{X_H}(\log D_H)$ with automorphic vector bundles. These are given by
\[
\begin{array}{r|cccc}
i & 0 & 1 & 2 & 3\\
\hline
\rule{0pt}{2.5ex}  \Omega^i_{X_G}(\log D_G) &
\omega_G(0, 0; 0) &
\omega_G(2, 0; 0) &
\omega_G(3, 1; 0) &
\omega_G(3, 3; 0) \\
\rule{0pt}{2.5ex} \Omega^i_{X_H}(\log D_H) &
\omega_H(0, 0; 0) &
\omega_H(2, 0; 0) \oplus \omega_H(0, 2; 0)&
\omega_H(2, 2; 0)
\end{array}
\]

  \subsubsection{Pullback and pushforward}
  \label{sect:pfwdsheaf}

   We shall need to consider pullbacks of automorphic vector bundles from $G$ to $H$. From the compatibility of torsors \eqref{eq:redstructure} and the inclusion of ideal sheaves \eqref{eq:idealsheaf} we obtain the relations
   \begin{equation}
   \label{eq:pullback1}
   \iota^*\big( [V] \big) = [V |_{B_H}], \qquad \iota^*\big( [V](-D_G) \big) \subseteq [V |_{B_H}](-D_H).
   \end{equation}
   for any $P_{\Sieg}$-representation $V$. (See also \cite[(2.5.3)]{harris90b} for the generic fibres; it is claimed in \emph{op.cit.} that the latter inclusion is also an inequality, but no proof is given, and we have not been able to locate a proof.)

   We also consider the ``exceptional inverse image'' functor $\iota^!$. Since $X_G$ and $X_H$ are smooth and projective over $\Spec \ZZ_{(p)}$, of relative dimensions 3 and 2 respectively, their relative dualising complexes are isomorphic to $\Omega^3_{X_G}[3]$ and $\Omega^2_{X_H}[2]$. The functoriality of the dualising complex gives a canonical isomorphism in the derived category of coherent sheaves on $X_H$,
   \[ \iota^!\left(\Omega^3_{X_G}\right) \cong \Omega^2_{X_H}[-1].\]
   Tensoring this with the pullback isomorphism \eqref{eq:pullback1} (with $V$ replaced by $V \otimes W_G(3, 3; 0)^*$), we obtain isomorphisms
   \[
      [V|_{B_H} \otimes \alpha_{G/H}^{-1}](-D_H) \cong \iota^!\big([V](-D_G)\big)[1]
   \]
   for any $P_{\Sieg}$-representation $V$, where $\alpha_{G/H}$ denotes the character $\lambda(1, 1; 0)$ of $B_H$. Together with the inclusion of ideal sheaves \eqref{eq:idealsheaf} we also obtain maps (which might not be isomorphisms)
   \[  [V|_{B_H} \otimes \alpha_{G/H}^{-1}] \to \iota^!\big([V]\big)[1]. \]

   \begin{remark}
    The restriction of $[\alpha_{G/H}]$ to the open variety $Y_{H, \QQ}$ is the conormal bundle of the closed embedding $\iota$, which explains its appearance in the pushforward formulae.
   \end{remark}

\subsection{Cohomology}

The (Zariski) cohomology groups $H^i(X_G, [V])$ and $H^i(X_G, [V](-D_G))$ are independent, up to canonical isomorphism, of the choice of cone decomposition $\Sigma$, and have actions of prime-to-$p$ Hecke operators $[K g K]$, for $g \in G(\Af^{p})$.

\begin{remark}
 Note that if $V$ has central character $c$, and $K$ has level $N$, then for all primes $\ell \ne p$ congruent to 1 modulo $N$, the double coset of $\diag(\varpi_\ell, \dots, \varpi_\ell)$ acts as $\ell^c$, where $\varpi_\ell$ is a uniformizer at $\ell$.
\end{remark}

The same is true for $H$ in place of $G$, and the morphisms of sheaves in the previous section gives us maps
\begin{subequations}
 \begin{align}
 \label{eq:pwfd1a}
 H^i\left(X_G, [V]\right) &\xrightarrow{\iota^{\star}} H^{i}\left(X_H, [V|_{B_H}]\right), \\
 \label{eq:pwfd1b}
 H^i\left(X_G, [V](-D_G)\right) &\xrightarrow{\iota^{\star}} H^{i}\left(X_H, [V|_{B_H}](-D_H)\right),\\
 \label{eq:pwfd1c}
 H^i\left(X_H, [V|_{B_H} \otimes \alpha_{G/H}^{-1}]\right) &\xrightarrow{\iota_{\star}} H^{i+1}\left(X_G, [V]\right)\\
 \label{eq:pwfd1d}
 H^i\left(X_H, [V|_{B_H} \otimes \alpha_{G/H}^{-1}](-D_H)\right) &\xrightarrow{\iota_{\star}} H^{i+1}\left(X_G, [V](-D_G)\right).
 \end{align}
\end{subequations}
for $0 \le i \le 2$ and any $P_{\Sieg}$-representation $V$. Serre duality gives canonical trace maps $H^3(X_G, \omega_G(3, 3;0)(-D_G)) \to \Zp$ and $H^2(X_H, \omega_H(2, 2; 0)(-D_H)) \to \Zp$, and with respect to the cup product pairings, the morphism \eqref{eq:pwfd1d} is the dual of the morphism \eqref{eq:pwfd1a}, and similarly \eqref{eq:pwfd1c} is dual to \eqref{eq:pwfd1b}.

If we base-extend to $\QQ$, we may drop the assuption that $K_p$ be hyperspecial and allow arbitrary level groups $K$. The direct limit
\[ H^*\left(X_{G, \QQ}(\infty), [V]\right)\coloneqq \varinjlim_K H^*\left(X_{G, \QQ}(K), [V]\right)\]
is then an admissible smooth $\QQ$-linear representation of $G(\Af)$. The pullback maps \eqref{eq:pwfd1a} assemble into morphisms of $H(\Af)$-representations
\[ \iota^*:  H^*\left(X_{G, \QQ}(\infty), [V]\right) \to H^*\left(X_{H, \QQ}(\infty), [V|_{B_H}]\right). \]
The same statements also hold with $[V]$ replaced by the subcanonical extension $[V](-D_G)$, using the maps \eqref{eq:pwfd1b}.


\section{The ordinary part of coherent cohomology}
 \label{sect:vincent}


 In this section, we'll explain how to embed the cohomology of automorphic vector bundles on $X_G$ inside larger spaces which vary in $p$-adic families, focussing on the case of $H^1$. This is an analogue for regular weights of the theory developed for singular (non-regular) weights in \cite{pilloni17}. In this section we shall work solely with objects attached to the group $G$, with the subgroup $H$ playing no role, so we shall drop the subscripts $G$ from the notation; they will reappear in the next section (where we compare the theories for $G$ and $H$).

 \begin{remark}
  We have aimed to recall enough of the definitions and notation from \cite{pilloni17} to give a reasonably self-contained \emph{statement} of the main result of this section (Theorem \ref{thm:vincent}). However, in the proof of this theorem we shall assume familiarity with \emph{op.cit.}.
 \end{remark}

 \subsection{The ordinary part of classical cohomology}

 Suppose $(k_1, k_2; c)$ is a weight with $k_2 \le 1$ and $k_1 + k_2 \ge 4$. For $n \ge 1$, we shall consider the Hecke operator on $H^1(X_{\Kl}(p^n)_{\Qp}, \omega(k_1, k_2; c)(-D))$ defined by:
 \[
  \cU_{p, \Kl} \coloneqq
  p^{k_1-3-c} [\Kl(p^n) \diag(p^2, p, p, 1) \Kl(p^n)]
 \]
 where we consider $ \diag(p^2, p, p, 1)$ as an element of $G(\Qp) \subset G(\Af)$. The normalisation factor implies that all eigenvalues of $\cU_{p, \Kl}$ are $p$-adically integral, and we denote by $e_{\Kl}$ the projection onto the ordinary part for this operator (the direct sum of the generalised eigenspaces whose eigenvalues are $p$-adic units).

 It is clear that this operator is independent of $c$, if we identify the cohomology of $\omega(k_1, k_2; c)$ for different values of $c$ in the obvious fashion. Hence for the remainder of this section, we shall fix $c$ to be equal to $k_1 + k_2 - 6$, so that the normalising factor for $\cU_{p, \Kl}$ is $p^{3-k_2}$.

 \begin{proposition}
  \label{prop:newtonpoly}
  Let $\Pi_p$ be an irreducible subquotient of the $G(\Qp)$-representation
  \[ \varinjlim_{K_p} H^i\left(X(K^p K_p)_{\bar{\QQ}_p}, \omega(k_1, k_2)(-D)\right), \]
  for any $i$, such that $e_{\Kl} \cdot \Pi_p^{\Kl(p)} \ne 0$. Then $e_{\Kl} \cdot \Pi_p^{\Kl(p)}$ has dimension $1$; and if $(\alpha, \beta, \gamma, \delta)$ are the Hecke parameters of $\Pi_p$ (in the sense of \cite[\S 5.1.5]{pilloni17}), ordered such that $v_p(\alpha) \le v_p(\beta)$ and $p^{2-k_2} \alpha \beta$ is the eigenvalue of $\cU_{p, \Kl}$ on $e_{\Kl} \cdot \Pi_p^{\Kl(p)}$, then we have
  \[
    k_2 - 2 \le v_p(\alpha) \le v_p(\beta) \le 0, \qquad  k_1 + k_2 - 3 \le v_p(\gamma) \le v_p(\delta) \le k_1-1.
  \]
 \end{proposition}

 \begin{proof}
  By assumption, there is a line in $\Pi_p^{\Kl(p)}$ on which $\cU_{p, \Kl}$ acts as $p^{2-k_2}\alpha \beta$ and the Hecke operator $p^{2-k_2}[\Kl(p) \diag(p, p, 1, 1) \Kl(p)]$ acts as $p^{2-k_2}(\alpha + \beta)$. Since both of these Hecke operators preserve an integral lattice, we deduce that $v_p(\alpha + \beta) \ge k_2 - 2$ and $v_p(\alpha \beta) = k_2 - 2$. It follows that $v_p(\alpha)$ and $v_p(\beta)$ lie in the interval $[k_2-2, 0]$, and we can order them so that $v_p(\alpha) \le v_p(\beta)$. To see that both operators preserve an integral lattice, we first observe that the cohomology we consider is a subquotient of the cohomology of a local system by \cite[chap. VI, Theorem 5.5]{faltingschai}. We can produce stable lattices for the action of the Hecke operators on the cohomology of local systems as in \cite{laff-est}. The inequalities for $v_p(\gamma)$ and $v_p(\delta)$ follow using the fact that $\alpha\delta = \beta\gamma$ is $p^{k_1 + k_2 - 3}$ times a root of unity.

  From the classification of Iwahori-spherical representations of $\GSp_4$ (see \cite{MR2114732} for instance, or Table A.15 of \cite{robertsschmidt07}), we know that $\Pi_p^{\Kl(p)}$ has dimension $\le 4$, with equality if and only if $\Pi_p$ is an unramified principal series; and the characteristic polynomial of $\cU_{p, \Kl}$ on $\Pi_p^{\Kl(p)}$ divides $(X - p^{2-k_2}\alpha\beta)\dots(X - p^{2-k_2}\gamma\delta)$. Since $\{ p^{2-k_2} \alpha\gamma, p^{2-k_2}\beta\delta, p^{2-k_2} \gamma\delta\}$ all have strictly positive valuation, this polynomial has exactly one root which is a $p$-adic unit, namely $p^{2-k_2} \alpha\beta$, and this root appears with multiplicity 1. Thus $e_{\Kl} \cdot \Pi_p^{\Kl(p)}$ is 1-dimensional.
 \end{proof}

 \begin{corollary}
  \label{cor:sphericalordprojection}
  Suppose $\Pi_p$ is as in Proposition \ref{prop:newtonpoly}, and $k_1 + k_2 \ge 5$. Then:
  \begin{enumerate}[(i)]
   \item $\Pi_p$ is either a representation of Sally--Tadic type I (i.e.~an irreducible principal series), or it is of type IIIb, so $\Pi_p \cong \chi \rtimes \sigma 1_{\GSp_2}$, with $\chi$ and $\sigma$ unramified characters. In particular, $\Pi_p$ is spherical.
   \item The composite map
   \[  \Pi_p^{G(\Zp)} \into \Pi_p^{\Kl(p)} \onto e_{\Kl} \Pi_p^{\Kl(p)} \]
   is an isomorphism of one-dimensional $\overline{\QQ}_p$-vector spaces.
  \end{enumerate}
 \end{corollary}

 \begin{proof}
  We know that $\Pi_p$ is a subquotient of the induction of an unramified character of $B_G$, determined (modulo the action of the Weyl group) by the parameters $(\alpha, \beta, \gamma, \delta)$. From the inequalities for the valuations established above, we see that either:
  \begin{itemize}
   \item none of the ratios of the Hecke parameters is $p^{\pm 1}$; or
   \item we have $\beta = p\alpha$ and $\delta = p\gamma$, and $\tfrac{\gamma}{\alpha} = \tfrac{\beta}{\delta}$ has valuation $\ge 3$.
  \end{itemize}
  In the first case, the induced representation from $B_G$ is irreducible, so $\Pi_p$ is of type I. In the second case, the induced representation has exactly two composition factors, so $\Pi_p$ is either $\chi \rtimes \sigma \operatorname{St}_{\GSp_2}$ (type IIIa) or $\chi \rtimes \sigma 1_{\GSp_2}$ (type IIIb) in the notation of Sally--Tadic, where $\chi, \sigma$ are unramified characters with $\chi(p) = \gamma/\alpha$ and $\sigma(p) = \alpha / p$. However, the type IIIa case cannot occur, because the $\Kl(p)$-invariants of $\chi \rtimes \sigma \operatorname{St}_{\GSp_2}$ are 1-dimensional with $[\Kl(p) \diag(p^2, p, p, 1) \Kl(p)]$ acting as $p \chi(p) \sigma(p)^2$, so the $\cU_{p, \Kl}$-eigenvalue is $p^{2-k_2} \alpha \gamma$, which has valuation $\ge 2$. Thus $\Pi_p$ must be of type IIIb, and in particular is spherical.

  If $\Pi_p$ is an irreducible unramified principal series, then the bijectivity of the map (ii) is proved in \cite[Corollary 3.2.4]{genestiertilouine05}. For the type IIIb case, one can argue similarly: one checks that if $\Pi_p = \chi \rtimes \sigma 1_{\GSp_2}$, then there is a basis of the 3-dimensional space $\Pi_p^{\Kl(p)}$ in which the spherical vector is $(1, 1, 1)^T$ and the matrix of $[\Kl(p)\diag(p^2, p, p, 1) \Kl(p)]$ has the form
  \[
  \sigma(p)^2 \cdot
  \begin{pmatrix}
  p^2 \lambda^2 & 0 & 0 \\
  (p^2 - p) \lambda^2 & p^3 \lambda & 0 \\
  (p^2 - p) \lambda^2 + (p-1)\lambda & (p^3 - p) \lambda & p^2 \\
  \end{pmatrix}
  \]
  where $\lambda = \chi(p)$. From this one verifies by explicit computation that the projection of the spherical vector to the ordinary eigenspace, which corresponds to the eigenvalue $p^2 \sigma(p)^2$, is always non-zero.
 \end{proof}

 \subsection{P-adic coefficient sheaves}
 \label{sect:Gsetup}
 We recall some notations and constructions from \cite[\S 9]{pilloni17}. Recall that $X \to \Spec \Zp$ is a toroidal compactification of the Shimura variety of prime-to-$p$ level $K^p G(\Zp)$.

 \subsubsection*{The Klingen and Igusa towers}

 Let $\fX$ be the $p$-adic completion of $X$, and $\fX^{\ge 1}$ the open subscheme where the $\mu_p$-rank of the universal semiabelian scheme $A$ is $\ge 1$.

 Over $\fX^{\ge 1}$ we have the tower $\fX^{\ge 1}_{\Kl}(p^\infty) = \varprojlim_m\fX^{\ge 1}_{\Kl}(p^m)$, where $\fX^{\ge 1}_{\Kl}(p^m)$ parametrises choices of subgroup $C_m \subset A[p^m]$ \'etale-locally isomorphic to $\mu_{p^m}$. Above $\fX^{\ge 1}_{\Kl}(p^m)$ there is a $(\ZZ / p^m)^\times$-torsor $\fIG(p^m)$, parametrising choices of isomorphism $C_m \cong \mu_{p^m}$.

 We can identify the generic fibre $\cX^{\ge 1}_{\Kl}(p^m)$ of $\fX^{\ge 1}_{\Kl}(p^m)$ with an open rigid-analytic subvariety of the analytification of $X_{\Kl}(p^m)_{\Qp}$, the compactified Shimura variety of level $K^p \Kl(p^m)$, where $\Kl(p^m) \subset G(\Zp)$ is the preimage of the Klingen parabolic modulo $p^m$. Similarly, the generic fibre of $\fIG(p^m)$ is an open in the rigid-analytic Shimura variety of level
 \[ K^p \times \left\{ \begin{smatrix}
 1 & * & * & * \\ &*&*&* \\ &*&*&* \\ &&&*\end{smatrix} \pmod{p^m}.\right\}\]

 \begin{remark}
  This is \textbf{different} from the normalisations of \cite{pilloni17} where the Igusa tower corresponds to the subgroups $\begin{smatrix}
  * & * & * & * \\ &*&*&* \\ &*&*&* \\ &&&1\end{smatrix} \pmod{p^m}$. This difference arises because the choices of Shimura cocharacter in use in \cite{pilloni17} and \cite{LSZ17} differ by a sign (cf.~\cite[Remark 5.1.2]{LSZ17}), and we have chosen to maintain compatibility with the latter.
 \end{remark}

 \subsubsection*{A ``big'' coefficient sheaf} Let $\pi$ be the natural map $\fIG(p^\infty) \to \fX^{\ge 1}_{\Kl}(p)$. For $R$ a $p$-adic ring with a continuous character $\kappa: \Gamma \to R^\times$, and $k_2 \in \ZZ$, we define a sheaf of $R$-modules on $\fX^{\ge 1}_{\Kl}(p)$ by
 \[
 \mathfrak{F}_{G, R}(\kappa, k_2)
 \coloneqq \left( \pi_\star \mathscr{O}_{\fIG_{G}(p^\infty)} \htimes R\right)[\Gamma = \kappa - k_2] \otimes \omega_G(k_2, k_2),
 \]
 where $\Gamma = \Zp^\times$ acts on $\pi_\star \mathscr{O}_{\fIG_{G}(p^\infty)}$ via the $\Gamma$-torsor structure of $\fIG_{G}(p^\infty)$. Here $\omega_G(k_2, k_2) = \omega_G(k_2, k_2; 2k_2 - 6)$ is the line bundle defined in Notation \ref{not:sheaves} above; this twist allows us to construct a $p$-adic family of sheaves interpolating the $\omega_G(k_1, k_2)$ for varying $k_1$ and fixed but arbitrary $k_2$, rather than just for $k_2 = 0$. (This corresponds to the twist $\omega^r$ appearing in \cite[\S 10.6]{pilloni17} for example.)

 Henceforth until the end of the section we shall omit $G$ from the subscript and write $\mathfrak{F}_R(\kappa, k_2)$. We shall also omit $R$ if $R = \Zp$ and $\kappa$ is an integer $k_1 \in \ZZ$.

 \begin{remark} In \cite{pilloni17} the sheaf is defined as above when $(R, \kappa)$ is the universal object $\Lambda = \Zp[[\Gamma]]$ (with its canonical character), and then extended to general $(R, \kappa)$ by base-change; but it is easily checked that this agrees with the definition above.
 \end{remark}

 \subsubsection*{Comparison maps}

 For any integers $k_1 \ge k_2$ there is a canonical morphism of sheaves on $\fX^{\ge 1}_{\Kl}(p)$ (\cite[Section 9.4]{pilloni17}):
 \begin{equation}
 \label{eq:comparison}
 \operatorname{comp}: \omega(k_1, k_2) \to \mathfrak{F}(k_1, k_2).
 \end{equation}

 \subsection{Statement of Theorem \ref{thm:vincent}}

 Let $\Lambda = \Zp[[\Gamma]]$ and $\kappa: \Gamma \to \Lambda^\times$ the canonical character. We consider the complex $\mathrm{R}\Gamma \left( \mathfrak{X}^{\geq 1}_{\Kl}(p), \fF_\Lambda(\kappa, k_2)(-D)\right)$ in the derived category of $\Lambda$-modules, for some $k_2 \in \ZZ$. This complex has an an action of Hecke operators away from $p$. By \cite[\S 10.6]{pilloni17}, it also has an action of $\cU_{p, \Kl}$. As we shall see in \S \ref{sect-U-loc-finite} below, this operator $\cU_{p, \Kl}$ is locally finite (in the sense of \cite[Definition 2.3.2]{pilloni17}), and hence has an associated ordinary idempotent $e_{\Kl}$. We set
 \[
 M^\bullet_{\kappa, k_2} = e_{\Kl} \mathrm{R}\Gamma \left( \mathfrak{X}^{\geq 1}_{\Kl}(p), \fF_\Lambda(\kappa, k_2)(-D)\right).
 \]

 The main result of this section is the following:

 \begin{theorem}
  \label{thm:vincent}
  If $k_2 \leq 0$,  the complex $M^\bullet_{\kappa, k_2}$ is quasi-isomorphic to a finite projective $\Lambda$-module placed in degree $1$, and for all $k_1 \in \ZZ$ such that $k_1 + k_2 \geq 4$, we have a canonical quasi-isomorphism:
  \[ M^\bullet_{\kappa, k_2} \otimes^{\mathbf{L}}_{\Lambda, k_1} \Qp = e_{\Kl}  \mathrm{R}\Gamma \left(X_{\Kl}(p)_{\Qp}, \omega(k_1, k_2)(-D)\right).\]
  This isomorphism is compatible with the action of the Hecke algebra away from $p$, and the operator $\cU_{p, \Kl}$ at $p$.
 \end{theorem}

 \begin{remark}
  We expect that this should be true for $k_2 = 1$ as well, but we have not been able to prove this. This contrasts with the case $k_2 = 2$ studied in \cite{pilloni17}, where there exist automorphic representations contributing to both $H^0$ and $H^1$.
 \end{remark}

 \subsection{Proof of Theorem \ref{thm:vincent}}

 For simplicity, we simply write $U$ for $\cU_{p, \Kl}$ in this section, and $e$ for the corresponding idempotent $e_{\Kl}$.

 \subsubsection{Existence of the projector and classicity ``along the sheaf''}\label{sect-U-loc-finite} Let $(k_1, k_2) \in \ZZ, k_1 \ge k_2$.

 \begin{lemma}\label{lem-control-sheaf}  The $U$-operator acts locally finitely on the complexes $\mathrm{R}\Gamma ( \mathfrak{X}^{\geq 1}_{\Kl}(p), \omega(k_1, k_2)(-D))$ and $\mathrm{R}\Gamma ( \mathfrak{X}^{\geq 1}_{\Kl}(p), \fF(k_1, k_2)(-D))$. Moreover, the map
  \[  e\mathrm{R}\Gamma ( \mathfrak{X}^{\geq 1}_{\Kl}(p), \omega(k_1, k_2)(-D)) \rightarrow e\mathrm{R}\Gamma( \mathfrak{X}^{\geq 1}_{\Kl}(p), \fF(k_1, k_2)(-D))\]
  is a quasi-isomorphism.
 \end{lemma}

 \begin{proof} We first show the local finiteness for  $\mathrm{R}\Gamma ( \mathfrak{X}^{\geq 1}_{\Kl}(p), \omega(k_1, k_2)(-D))$. By \cite[proposition 2.3.1]{pilloni17}, it suffices to check that the cohomology $\mathrm{R}\Gamma ( \mathfrak{X}^{\geq 1}_{\Kl}(p), \omega(k_1, k_2)(-D))$ can be represented by a complex of flat, $p$-adically separated and complete  $\Zp$-modules, that the $U$-operator can be represented by an endomorphism of this complex, and that $U$ is locally finite on the cohomology groups $H^i ({X}^{\geq 1}_{\Kl}(p)_1, \omega{(k_1,k_2)}(-D) )$ (for ${X}^{\geq 1}_{\Kl}(p)_1$ the special fiber of $\mathfrak{X}^{\geq1}_{\Kl}(p)$). The sheaf $\omega(k_1, k_2)(-D)$ is acyclic relative to the minimal compactification.
  Let $j$ be the open immersion $\mathfrak{X}^{\geq 2}_{\Kl}(p) \hookrightarrow \mathfrak{X}^{\geq 1}_{\Kl}(p)$. We deduce that
  \[
   j_\star j^\star \omega(k_1, k_2)(-D) \longrightarrow
   \frac{j_\star j^\star \omega(k_1, k_2)(-D)}{\omega(k_1, k_2)(-D)}
  \]
  is an acyclic resolution of  $\omega(k_1, k_2)(-D))$ (because the images of $\mathfrak{X}^{\geq 2}_{\Kl}(p)$ and $\mathfrak{X}^{= 1}_{\Kl}(p)$ are affine in the minimal compactification). Therefore, the cohomology $\mathrm{R}\Gamma ( \mathfrak{X}^{\geq 1}_{\Kl}(p), \omega(k_1, k_2)(-D))$ can be represented by a complex of amplitude $[0,1]$ and the $U$-operator induces an endomorphism of this complex (because the Hecke correspondence respects the $p$-rank stratification).
  We now check that $U$ is locally finite  on $H^i ({X}^{\geq 1}_{\Kl}(p)_1, \omega(k_1, k_2)(-D))$. We will prove this by decreasing induction on $k_2$. We know by \cite[Theorem 11.2.1, Theorem 11.3.1 ]{pilloni17}  that this holds true for $k_2 \geq 2$.
  We consider the following exact sequence, where $\mathrm{Ha}$ is the Hasse invariant:
  \[ 0 \rightarrow \omega(k_1,k_2)(-D) \xrightarrow{\mathrm{Ha}} \omega(k_1 +p-1, k_2 + p-1)(-D) \rightarrow \omega(k_1 +p-1, k_2 + p-1)(-D)/ \mathrm{Ha} \rightarrow 0. \]
  This yields a $U$-equivariant long exact sequence on cohomology by \cite[lemma 10.5.2.1]{pilloni17}. By our inductive hypothesis, $U$ is locally finite on $H^i ({X}^{\geq 1}_{\Kl}(p)_1, \omega(k_1 +p-1, k_2 + p-1)(-D))$.  On $H^i({X}^{\geq 1}_{\Kl}(p)_1, \omega(k_1 +p-1, k_2 + p-1)(-D)/\mathrm{Ha}) = H^i({X}^{= 1}_{\Kl}(p)_1, \omega(k_1 +p-1, k_2 + p-1)(-D))$, multiplication by the second Hasse invariant (see \cite[Section 6.3.2]{pilloni17}) induces $U$-equivariant isomorphisms  $H^i({X}^{= 1}_{\Kl}(p)_1, \omega(k_1 +p-1, k_2 + p-1)(-D)) = H^i({X}^{= 1}_{\Kl}(p)_1, \omega(k_1 + p-1 + p^2 -1, k_2 + p-1 + p^2-1)(-D))$ (\cite[Lemma 10.5.3.1]{pilloni17}). By the inductive hypothesis, $U$ is locally finite on  $H^i({X}^{= 1}_{\Kl}(p)_1, \omega(k_1 + p-1 + p^2 -1, k_2 + p-1 + p^2-1)(-D))$.  We can therefore conclude that $U$ is locally finite on $H^i ({X}^{\geq 1}_{\Kl}(p)_1, \omega(k_1, k_2)(-D))$ by \cite[lemma 2.1.1]{pilloni17}.

  We now prove the local finiteness on $\mathrm{R}\Gamma( \mathfrak{X}^{\geq 1}_{\Kl}(p), \fF(k_1, k_2)(-D))$ and the quasi-isomorphism. This cohomology is computed by a complex of amplitude $[0,1]$ and $U$ lifts to an operator on this complex by arguments similar to the ones used for the sheaf $\omega(k_1,k_2)(-D)$. We next show that $U$ is locally finite on  $H^i( \mathfrak{X}^{\geq 1}_{\Kl}(p), \fF(k_1, k_2)(-D) \otimes \FF_p)$.  As in the proof of \cite[ Theorem 11.3.1 ]{pilloni17} we establish at the same time local finiteness and the isomorphism: $ e H^i ({X}^{\geq 1}_{\Kl}(p)_1, \omega(k_1, k_2)(-D))  \rightarrow e H^i( \mathfrak{X}^{\geq 1}_{\Kl}(p), \fF(k_1, k_2)(-D) \otimes \FF_p)$. Details are left to the reader.
 \end{proof}

 \begin{lemma} The $U$-operator is locally finite on $\mathrm{R}\Gamma ( \mathfrak{X}^{\geq 1}_{\Kl}(p), \fF_\Lambda(\kappa, k_2)(-D))$.
 \end{lemma}
 \begin{proof} This follows from the previous lemma and \cite[Proposition 2.3.1]{pilloni17}.
 \end{proof}

 \subsubsection{A vanishing theorem} Recall that $X^{\geq1}_{\Kl}(p)_1$ denotes the mod $p$ reduction of $\fX^{\geq 1}_{\Kl}(p)$.

 \begin{proposition}\label{prop-vanish} If $k_2 \leq  0$, then $e H^0(X^{\geq1}_{\Kl}(p)_1, \omega(k_1, k_2)(-D)) = 0$.
 \end{proposition}

 \begin{proof}
  Over $X^{\geq 1}_{\Kl}(p)_1$ we have a universal multiplicative subgroup $C \subset A[p]$, where $A$ is the universal semiabelian scheme. The associated conormal sheaf, which is an invertible sheaf, is denoted by $\omega_C$. We have a surjective map $\omega_A \rightarrow \omega_C$, and therefore there is a surjective map
  \[
   \omega(k_1,k_2) = \Sym^{k_1-k_2} \omega_A \otimes \det^{k_2} \omega_A \longrightarrow  \omega_C^{k_1-k_2} \otimes \det^{k_2} \omega_A.\]
  It follows from \cite[Lemma 10.7.1]{pilloni17}, that the map  $eH^0(X^{\geq1}_{\Kl}(p)_1, \omega(k_1, k_2)) \rightarrow eH^0({X^{\geq1}_{\Kl}(p)_1}, \omega_C^{k_1-k_2} \otimes \det^{k_2} \omega_A(-D)) $ is injective.

  The idea of the proof is to evaluate sections of $H^0({X^{\geq1}_{\Kl}(p)_1}, \omega_C^{k_1-k_2} \otimes \det^{k_2} \omega_A(-D)) $ on various modular curves that map to ${X^{\geq1}_{\Kl}(p)_1}$. We will see that the sections of $H^0({X^{\geq1}_{\Kl}(p)_1}, \omega_C^{k_1-k_2} \otimes \det^{k_2} \omega_A(-D))$ vanish along these modular curves, and that the union of the images of these modular curves is Zariski dense.

  The tame level  of $X^{\geq 1}_{\Kl}(p)_1$ is  the compact open subgroup $K^p \subset G(\Af^p)$. Let ${K'}^p \subset {G}(\Af^p) $ be a compact open subgroup and let $g  \in G(\Af^p)$ be such $g^{-1} {K'}^p g \subset K^p$. We have a map $g : X^{\geq 1}_{{K'}^p\Kl}(p)_1 \rightarrow X^{\geq1}_{\Kl}(p)_1$ for suitable choices of polyhedral cone decompositions.

  Let $E_0$ be an ordinary elliptic curve defined  over a  field $k$ of characteristic $p$. For a suitable level structure ${K''}^p \subset \GL_2(\Af^p)$ and the choice of a suitable level structure on $E_0$, we have a  map  from the modular curve of level ${K''}^p$ away from $p$ and spherical level at $p$ to  $X^{\geq1}_{{K'}^p\Kl}(p)_1$:
  \[ j : Y_{\GL_2, K''^p,1} \mathop{\times}_{\FF_p}k \rightarrow X^{\geq1}_{{K'}^p\Kl}(p)_1\]
  defined by sending the universal  elliptic curve $E$ to the abelian surface $E_0 \times E$, equipped with the product polarization, the multiplicative subgroup $C \subset E_0$ and the apropriate level structure away from $p$.  This map extends clearly  to a map $j : X_{\GL_2, K''^p,1} \times_{\Spec \FF_p} \Spec k \rightarrow X^{\geq1}_{{K'}^p\Kl}(p)_1$ of the compactified modular curve.

  Let $f \in H^0({X^{\geq1}_{\Kl}(p)_1}, \omega_C^{k_1-k_2} \otimes \det^{k_2} \omega_A(-D)) $. The pull back $j^\star g^\star(f)$ is a cuspidal modular form of weight $k_2 \leq 0$. We therefore find that $j^\star g^\star(f) = 0$.

  It follows that $f$ vanishes at all prime to $p$ Hecke translates of points of the form $E_0 \times E$, where $E_0$ is an ordinary elliptic curve, $E$ is any elliptic curve and $C \subset E_0$ is the multiplicative subgroup. We claim that this set is Zariski dense in $X^{\geq1}_{\Kl}(p)_1$ and therefore that $f=0$.

  It suffices to prove that this set is Zariski dense in $X^{\geq 2}_{\Kl}(p)_1$, and using the irreducibility of the \'etale cover :
  $X^{\geq 2}_{\Kl}(p)_1 \rightarrow X^{\geq 2}_1$, we are left to prove the Zariski density of prime to $p$ Hecke translates of points in $X^{\geq 2}_1$ (the ordinary part of the modulo $p$ Shimura variety of prime-to-$p$ level) of the form $E_0 \times E$ for a product of two ordinary elliptic curves. This set is the union of the Hecke translates of the codimension one subscheme $\iota(X_{H,1}^{\geq 2})$ (the image of the ordinary part of the modulo $p$ Shimura variety for $H$). It is sufficient to prove that these Hecke translates form an infinite union of codimension one subschemes (indeed, one can easily reduce to the case that $X^{\geq 2}_1$ is connected). We prove this last claim as follows. Given a geometric point $x$ on some prime to $p$ Hecke translate of $X_{H,1}^{\geq 2}$, we let $d(x)$ be the minimal degree of a prime to $p$ isogeny between $x$ and a product of two elliptic curves. It will suffice to prove that there exists a sequence of points $x_n$ with $d(x_n) \rightarrow \infty$. We can produce such points as follows. We consider two non-isogenous ordinary elliptic curves $E$ and $F$ over $\bar{\FF}_p$. Let $\ell \neq p$ be a prime, and let $e_1,e_2, f_1, f_2$ be a basis of the $\ell$-adic Tate modules of $E$ and $F$ (with $\langle e_1, e_2 \rangle = \langle f_1, f_2 \rangle$), and let $C_\ell$ be the subgroup generated by the images of $e_1 + f_1$ and $e_2 - f_2$ in $E[\ell] \times F[\ell]$. This is a totally isotropic subgroup. We let $A_\ell = (E \times F) /C_\ell$. This is a principally polarized abelian surface and it is easy to see that the minimal degree of an isogeny between $A_\ell$ and a product of elliptic curves is $\ell^2$.
 \end{proof}

 \begin{corollary}
  If $k_2 \leq 0$, then $eH^i(\mathfrak{X}^{\geq1}_{\Kl}(p), \omega(k_1, k_2)(-D))$ is $0$ if $i \neq 1$, and for $i = 1$ it is a flat, $p$-adically complete and separated $\Zp$-module.
 \end{corollary}

 \begin{proof} The sheaf $\omega(k_1, k_2)(-D)$ is acyclic relative to the minimal compactification. Since the image of $\mathfrak{X}^{\geq1}_{\Kl}(p)$ in the minimal compactification is covered by two affines, we deduce that the cohomology $\mathrm{R}\Gamma(\mathfrak{X}^{\geq1}_{\Kl}(p), \omega(k_1, k_2)(-D)) $ is represented by a complex of amplitude $[0,1]$ of complete, separated, flat $\Zp$-modules. The same holds for $e\mathrm{R}\Gamma(\mathfrak{X}^{\geq1}_{\Kl}(p), \omega(k_1, k_2)(-D))$. Let us write this complex as $M^0 \stackrel{d}\rightarrow M^1$. We claim that $d$ is injective and that $\operatorname{coker} d$ is flat, $p$-adically complete and separated.

 The complex $ M^0 \otimes \FF_p \stackrel{d \otimes 1} \rightarrow M^1 \otimes \FF_p$ computes $e\mathrm{R}\Gamma({X}^{\geq1}_{\Kl}(p), \omega(k_1, k_2)(-D))$. By proposition \ref{prop-vanish}, $\mathrm{ker}( d \otimes 1) = 0$ and it follows that $\ker d = p \ker d$. Since $\ker d$ is a complete and separated $\Zp$-module, we deduce that $\ker d = 0$. We also deduce that $M^0 \cap p M^1 = p M^0$ and it follows that $\operatorname{coker} d = M^1/M^0$ is flat, $p$-adically complete and separated.
 \end{proof}

 \subsubsection{Finiteness}

 Let $\mathcal{X}_{\Kl}(p)$ be the analytic adic space associated to $X_{\Kl}(p)$. Let $\mathcal{X}^{\geq 1}_{\Kl}(p)$ be the locus where $H$ is multiplicative (the generic fiber of
 $\mathfrak{X}^{\geq1}_{\Kl}(p)$). We may consider the following cohomology $\mathrm{R}\Gamma ( \mathcal{X}^{\geq 1}_{\Kl}(p), \omega(k_1, k_2)(-D)) = \mathrm{R}\Gamma ( \mathfrak{X}^{\geq 1}_{\Kl}(p), \omega(k_1, k_2)(-D)) \otimes^L_{\Zp} \Qp $ as well as its ordinary part
 \[ e\mathrm{R}\Gamma ( \mathcal{X}^{\geq 1}_{\Kl}(p), \omega(k_1, k_2)(-D)) = e \mathrm{R}\Gamma ( \mathfrak{X}^{\geq 1}_{\Kl}(p), \omega(k_1, k_2)(-D)) \otimes^L_{\Zp} \Qp.\]

 We may also consider the overconvergent cohomology $\mathrm{R}\Gamma ( \mathcal{X}^{\geq 1}_{\Kl}(p)^{\dag}, \omega(k_1,k_2)(-D))$. One checks that the $U$-operator is compact on this cohomology (see \cite[sect. 13. 2]{pilloni17}). In particular it makes sense to speak of the ordinary part :
 $e\mathrm{R}\Gamma ( \mathcal{X}^{\geq 1}_{\Kl}(p)^\dag, \omega(k_1,k_2)(-D))$, and this is a perfect complex of $\Qp$-vector spaces.

 \begin{lemma}
  If $k_2 \leq 0$, $e\mathrm{R}\Gamma ( \mathcal{X}^{\geq 1}_{\Kl}(p)^\dag, \omega(k_1,k_2)(-D))$ is concentrated in degree $1$ and the map
  \[
   eH^1( \mathcal{X}^{\geq 1}_{\Kl}(p)^\dag, \omega(k_1,k_2)(-D))
   \rightarrow eH^1(\mathcal{X}^{\geq 1}_{\Kl}(p), \omega(k_1, k_2)(-D))
  \]
  is surjective.
 \end{lemma}

 \begin{proof}
  The image of $\mathcal{X}^{\geq 1}_{\Kl}(p)$ in the minimal compactification is covered by two affinoids, say $U_1$ and $U_2$. Let $\pi$ be the projection from toroidal to minimal compactification.
  The complex
  \[ H^0(U_1, \pi_\star \omega(k_1, k_2)(-D)) \oplus H^0(U_2, \pi_\star \omega(k_1, k_2)(-D)) \rightarrow H^0(U_1 \cap U_2, \pi_\star \omega(k_1, k_2)(-D))
  \]
  computes $\mathrm{R}\Gamma ( \mathcal{X}^{\geq 1}_{\Kl}(p), \omega(k_1, k_2)(-D))$. The subcomplex of overconvergent sections
  \[ H^0(U_1^\dag, \pi_\star \omega(k_1, k_2)(-D)) \oplus H^0( U_2^\dag, \pi_\star \omega(k_1, k_2)(-D)) \rightarrow H^0(U_1^\dag \cap U_2^\dag, \pi_\star \omega(k_1, k_2)(-D))
  \]
  computes $\mathrm{R}\Gamma ( \mathcal{X}^{\geq 1}_{\Kl}(p)^\dag, \omega(k_1,k_2)(-D)).$ Therefore there is a surjective map $H^0(U_1 \cap U_2, \pi_\star \omega(k_1, k_2)(-D)) \rightarrow H^1( \mathcal{X}^{\geq 1}_{\Kl}(p), \omega(k_1, k_2)(-D))$ and overconvergent sections in $H^0(U_1 \cap U_2, \pi_\star \omega(k_1, k_2)(-D))$ are dense. The map $H^1( \mathcal{X}^{\geq 1}_{\Kl}(p)^\dag , \omega(k_1,k_2)(-D)) \rightarrow H^1( \mathcal{X}^{\geq 1}_{\Kl}(p), \omega(k_1, k_2)(-D))$ induces a continuous map
  \[ eH^1( \mathcal{X}^{\geq 1}_{\Kl}(p)^\dag , \omega(k_1,k_2)(-D)) \rightarrow eH^1( \mathcal{X}^{\geq 1}_{\Kl}(p), \omega(k_1, k_2)(-D))\]
  (by functoriality of slope $0$ decomposition) with dense image. The space $eH^1( \mathcal{X}^{\geq 1}_{\Kl}(p)^\dag, \omega(k_1,k_2)(-D))$ is a finite-dimensional $\Qp$-vector space, and therefore the map is surjective and the target is also finite dimensional. There is also an injective map $H^0(\mathcal{X}^{\geq 1}_{\Kl}(p)^\dag , \omega(k_1,k_2)(-D)) \rightarrow H^0( \mathcal{X}^{\geq 1}_{\Kl}(p), \omega(k_1, k_2)(-D))$ which induces an injective map $eH^0(\mathcal{X}^{\geq 1}_{\Kl}(p)^\dag , \omega(k_1,k_2)(-D)) \rightarrow eH^0( \mathcal{X}^{\geq 1}_{\Kl}(p), \omega(k_1, k_2)(-D))$. This last module is trivial and therefore $eH^0(\mathcal{X}^{\geq 1}_{\Kl}(p)^\dag , \omega(k_1,k_2)(-D)) =0$.
 \end{proof}

 \begin{corollary}
  If $k_2 \leq 0$, then $M^\bullet_{\kappa^{un}, k_2}[1]$ is a finite-rank projective $\Lambda$-module.
 \end{corollary}

 \begin{proof} By corollary \ref{lem-control-sheaf}, for any $k_1 \geq k_2$, we have $M^\bullet_{\kappa^{un}, k_2} \otimes^L_{\Lambda, k_1} \Zp = e \mathrm{R}\Gamma( \mathfrak{X}^{\geq 1}_{\Kl}(p), \omega(k_1, k_2)(-D))$. The corollary follows from the previous corollary and Nakayama's lemma for complexes \cite[Proposition 2.2.1]{pilloni17}.
 \end{proof}

 \subsubsection{Classicity} We recall the following classicity theorem :

 \begin{theorem}\cite[theorem 14.7.1]{pilloni17}
  \label{thm:OCclass}
  The map
  \[
   e \mathrm{R}\Gamma(\mathcal{X}_{\Kl}(p), \omega(k_1, k_2)(-D)) \rightarrow
   e \mathrm{R}\Gamma(\mathcal{X}^{\geq 1}_{\Kl}(p)^\dag, \omega(k_1, k_2)(-D))
  \]
  is a quasi-isomorphism if $k_1 + k_2 > 3$, and similarly without the $(-D)$.
 \end{theorem}

 We want to conclude :

 \begin{theorem} The map
  \[
   e \mathrm{R}\Gamma(\mathcal{X}_{\Kl}(p), \omega(k_1, k_2)(-D)) \rightarrow
   e \mathrm{R}\Gamma(\mathcal{X}^{\geq 1}_{\Kl}(p), \omega(k_1, k_2)(-D))
  \]
  is a quasi-isomorphism if $k_1 + k_2 > 3$ and $k_2 \leq 0$.
 \end{theorem}

 We already know that both complexes are concentrated in degree $1$ and that $e H^1(\mathcal{X}_{\Kl}(p), \omega(k_1, k_2)(-D)) \rightarrow e H^1(\mathcal{X}^{\geq 1}_{\Kl}(p), \omega(k_1, k_2)(-D))$ is surjective. It will therefore suffice to prove injectivity.

 \subsubsection{Injectivity if $k_1$ is large enough}

 We will first show that injectivity holds if $k_1$ is very large. The idea is to prove that the ordinary cohomology is isomorphic to ordinary cohomology at spherical level, and then to use the following type of result:

 \begin{theorem}\label{thmSGA} Let $S$ be a Cohen--Macaulay scheme and $\mathcal{L}$ be a locally free sheaf over $S$. Let $U \subset S$ be an open subscheme such that $S - U$ has codimension $i$. Then the map $H^j (S, \mathcal{L}) \rightarrow H^j(U, \mathcal{L})$ is bijective if $j < i$ and injective for $i=j$.
 \end{theorem}

 \begin{proof} SGA 2, III, lem. 3.1 and prop. 3.3.
 \end{proof}

 We  need to define the relevant ordinary projector at spherical level. This is a somewhat involved calculation.  Let $Y \rightarrow \Spec \Zp$ be the Shimura variety with spherical level at $p$ and tame level $K^p$ away from $p$. Let $Y_p$ be the Shimura variety with paramodular level at $p$ and tame level $K^p$ away from $p$. Let $Y_{\Kl}(p)$ be the Shimura variety with Klingen level at $p$ and tame level $K^p$ away from $p$. Over $Y_{\Kl}(p)$ we have a chain $A \rightarrow A/L_1 \rightarrow A$ where $L_1$ is a subgroup of order $p^3$ of $A[p]$ and the total map is multiplication by $p$. Therefore $A/L_1$ carries a degree $p^2$ polarization $\lambda$ and  $\mathrm{ker} (A/L_1 \rightarrow A)$ is a subgroup of order $p$ of the kernel of the polarization. We have morphisms  $p_1 : Y_{\Kl}(p) \rightarrow Y$ defined by $(A, L_1) \mapsto A$ and $p_2 : Y_{\Kl}(p) \rightarrow Y_p$ defined by $(A,L_1) \mapsto A/L_1$.

 We now consider toroidal compactifications  $X$, $X_{\Kl}(p)$ and $X_{p}$ of $Y$, $Y_{\Kl}(p)$ and $Y_p$, such that the maps $p_1$ and $p_2$ extend. We make further  specifications. We can choose a  smooth cone decomposition $\Sigma$ for $X$, a smooth cone decomposition  $\Sigma'$ for  $X_p$ and we take for $X_{\Kl}(p)$ the same cone decomposition $\Sigma'$.

 We now take formal completion of all these spaces and restrict to the $p$-rank $\ge 1$ locus. We denote by $\mathfrak{X}^{\geq 1}$, $\mathfrak{X}^{\geq 1}_p$ and $\mathfrak{C}_1$ the resulting spaces (we use $\mathfrak{C}_1$ for the $p$-rank $\ge 1$ locus in $\mathfrak{X}_{\Kl}(p)$, to avoid any confusion with $\mathfrak{X}^{\geq 1}_{\Kl}(p)$).
 The spaces $\mathfrak{X}^{\geq 1}$ and $\mathfrak{X}^{\geq 1}_p$ are smooth over $\Spf \Zp$. The morphism $\mathfrak{C}_1 \rightarrow \Spf \Zp$ is Cohen--Macaulay.

 We remark that the maps $p_1$ and $p_2$ are quasi-finite  away from the boundary (this is a consequence of the fact that abelian surfaces of $p$-rank at least one over algebraically closed fields have only a finite number of subgroups of order $p$). It follows from miracle flatness that $p_1$ and $p_2$ are finite flat away from the boundary. Actually, by our choice of cone decomposition, $p_2$ is quasi-finite at the boundary and therefore $p_2$ is finite flat.

 We let $\mathfrak{C} = \mathfrak{C}_1 \times_{p_2, \mathfrak{X}_p^{\geq 1}, p_2} \mathfrak{C}_1$.
 Over $\mathfrak{C}$ we have a chain of isogenies $A \rightarrow A/L_1 \rightarrow A/L$ where the first isogeny has degree $p^3$, the second isogeny has degree $p$.  We have two projections $q_1, q_2 : \mathfrak{C} \rightarrow \mathfrak{X}^{\geq 1}$, defined by $q_1(A,L) = A$, $q_2(A,L) = A/L$.
 We have a natural map $q_2^\star \omega(k_1, k_2)[1/p] \rightarrow q_1^\star \omega(k_1, k_2)[1/p]$ arising from the differential of the isogeny $A \rightarrow A/L$ which is \'etale after inverting $p$.

 \begin{lemma}\label{lem-justabove} The relative dualizing sheaf  $q_1^!  \cO_{\mathfrak{X}^{\geq 1}}$ is a Cohen--Macaulay sheaf. We  have a fundamental class $q_1^\star \cO_{\mathfrak{X}^{\geq 1}} \rightarrow q_1^! \cO_{\mathfrak{X}^{\geq 1}}$ which extends the usual trace map on the complement of the boundary.
 \end{lemma}

 \begin{proof}
  Since $\mathfrak{C}_1$ is Cohen--Macaulay,  and the map $\mathfrak{C}_1 \rightarrow \mathfrak{X}_p^{\geq 1}$  if finite flat, we deduce from \cite{MR1011461}, corollary on page 181, that the map $\mathfrak{C}_1 \rightarrow \mathfrak{X}_p^{\geq 1}$ is a Cohen--Macaulay morphism (i.e a flat morphism with Cohen--Macaulay fibres). A  base change of a Cohen--Macaulay morphism is still a Cohen--Macaulay morphism. Therefore, the morphism $\mathfrak{C}  \rightarrow \mathfrak{C}_1$ is a Cohen--Macaulay morphism. Since $\mathfrak{C}_1$ is Cohen--Macaulay, we deduce that $\mathfrak{C}$ is Cohen--Macaulay and that the structural morphism $ g : \mathfrak{C} \rightarrow \Spf \Zp$ is a Cohen--Macaulay morphism. On the other hand the morphism $h : \mathfrak{X}^{\geq 1} \rightarrow \Spf \Zp$ is smooth. We find that $q_1^! \cO_{\mathfrak{X}^{\geq 1}} \otimes h^! {\Zp}  = g^! \Zp$ and it follows that $q_1^! \cO_{\mathfrak{X}^{\geq 1}}$ is a CM sheaf (see \cite[Lemma 2.2]{Fakhruddin-Pilloni}).  The usual trace map away from the boundary  extends to  give a morphism  $q_1^\star \cO_{\mathfrak{X}^{\geq 1}} \rightarrow q_1^! \cO_{\mathfrak{X}^{\geq 1}}$ by \cite[Proposition 2.6]{Fakhruddin-Pilloni}.
 \end{proof}

 Therefore we get a map  $\Theta : q_2^\star \omega(k_1, k_2)[1/p] \rightarrow q_1^! \omega(k_1, k_2)[1/p]$ obtained by composing the map $q_2^\star \omega(k_1, k_2)[1/p] \rightarrow q_1^\star \omega(k_1, k_2)[1/p]$ and the map $q_1^\star \omega(k_1, k_2)[1/p] \rightarrow q_1^! \omega(k_1, k_2)[1/p]$ (simply obtained by tensoring the map of lemma \ref{lem-justabove} with $q_1^\star \omega(k_1, k_2)[1/p]$). By adjunction, this is also  a map $(q_1)_\star p_2^\star \omega(k_1, k_2)[1/p] \rightarrow  \omega(k_1, k_2)[1/p]$.

 We now optimize integrality:  we determine the optimal integer $s$ for which $p^s \Theta$ extends to a map  $q_2^\star \omega(k_1, k_2) \rightarrow q_1^! \omega(k_1, k_2)$. It is enough to determine the value of $s$ over the ordinary locus (and even away from the boundary) because a map from a locally free sheaf to  a CM sheaf which is defined up to a codimension $2$ closed subset is defined everywhere.

 Let us say that the isogeny $A \rightarrow A/L_1$ is \emph{as \'etale as possible} if $L_1$ has \'etale rank $2$ and \emph{as multiplicative as possible} if $L_1$ has multiplicative rank $2$. We have $\mathfrak{C}^{=2} = \mathfrak{C}^{et,et} \cup \mathfrak{C}^{m,et} \cup \mathfrak{C}^{et,m} \cup \mathfrak{C}^{m,m}$, where the first superscript refers to $A \rightarrow A/L_1$ being as \'etale or as multiplicative as possible, and the second superscript to $A/L_1 \rightarrow A/L$ being multiplicative or \'etale.

 We denote by $\Theta^{\star, \star \star}$ the projection of $\Theta$ on the connected components $\mathfrak{C}^{ \star, \star \star}$.

 We check that :

 \begin{itemize}
  \item $\Theta^{et , et } : q_2^\star \omega(k_1, k_2) \rightarrow p^{k_2+3} q_1^! \omega(k_1, k_2)$,
  \item $\Theta^{et , m } : q_2^\star \omega(k_1, k_2) \rightarrow p^{k_2+k_1 + 2} q_1^! \omega(k_1, k_2)$,
  \item $\Theta^{m, et } : q_2^\star \omega(k_1, k_2) \rightarrow p^{k_2+k_1 + 1} q_1^! \omega(k_1, k_2)$,
  \item $\Theta^{m, m } : q_2^\star \omega(k_1, k_2) \rightarrow p^{2k_2+k_1} q_1^! \omega(k_1, k_2)$
 \end{itemize}

 Let us briefly explain where these numbers arise from.

 \begin{itemize}
  \item $\Theta^{et , et }  : q_2^\star \omega(k_1, k_2) \rightarrow p^{k_2} q_1^\star \omega(k_1, k_2) \rightarrow  p^{k_2+3} q_1^! \omega(k_1, k_2)$ where the first $p^{k_2}$ factor arises from the differential of the isogeny $q_2^\star \omega_{A/L} \rightarrow p^{k_2} q_1^\star \omega_A$ (the group $L$ has multiplicative rank $1$) and the second $p^3$ factor arises from the trace.
  \item $\Theta^{m , et } : q_2^\star \omega(k_1, k_2) \rightarrow p^{k_2 + k_1} q_1^\star \omega(k_1, k_2)  \rightarrow p^{k_2+k_1 + 1} q_1^! \omega(k_1, k_2)$ where the first factor $p^{k_2+k_1}$ arises from the differential of the isogeny. The group $L$ contains $A[p]^{mult}$. The factor $p$ arises from the trace.
  \item $\Theta^{et, m }  q_2^\star \omega(k_1, k_2) \rightarrow p^{k_2 + k_1} q_1^\star \omega(k_1, k_2)  \rightarrow p^{k_2+k_1 + 2} q_1^! \omega(k_1, k_2)$ where the first factor $p^{k_2+k_1}$ arises from the differential of the isogeny. The group $L$ contains $A[p]^{mult}$. The factor $p^2$ arises from the trace.
  \item $\Theta^{m, m } : q_2^\star \omega(k_1, k_2) \rightarrow p^{2k_2+k_1} q_1^\star  \omega(k_1, k_2) \rightarrow p^{3k_2+k_1} q_1^! \omega(k_1, k_2)$ where the first factor $p^{2k_2+k_1}$ arises from the isogeny. The group $L$ contains a subgroup isomorphic (for the \'etale topology) to $\mu_p \times \mu_{p^2}$.
 \end{itemize}

 Therefore,  under the hypothesis that $k_1  \geq 2$ and $k_2 + k_1 \geq 3$ we have an operator $T = p^{-k_2-3} \Theta : q_2^\star \omega(k_1, k_2)  \rightarrow q_1^! \omega(k_2,k_1)$.    We  also denote by $T \in \mathrm{End}( \mathrm{R}\Gamma ( \mathfrak{X}^{\geq 1}, \omega(k_1, k_2)))$ the associated Hecke operator. We will especially be interested in the reduction modulo $p$ of this operator $T \in \mathrm{End}( \mathrm{R}\Gamma ({X}_1^{\geq 1}, \omega(k_1, k_2)))$.

 There is a natural map $\mathrm{R}\Gamma (X^{\geq 1}_1, \omega(k_1, k_2)) \rightarrow \mathrm{R}\Gamma (X_{\Kl}^{\geq 1}(p)_1, \omega(k_1, k_2))$ and we will compare the ordinary parts of these cohomology for $T$ and $U$ respectively. We first look at  the cohomology over the ordinary locus :

 \begin{lemma}\label{lem-UTord} If $k_1   >  2$ and $k_2 + k_1 >  3$, there is an operator $\tilde{T} : H^0( X^{\geq 2}_{\Kl}(p)_1, \omega(k_1, k_2)) \rightarrow H^0( X^{\geq 2}_{\Kl}(p)_1, \omega(k_1, k_2))$  fitting in a  commutative diagram :
  \[
   \begin{tikzcd} H^0( X^{\geq 2}_{\Kl}(p)_1, \omega(k_1, k_2)) \ar[r, "\tilde{T}"] \ar[rd] &  H^0( X^{\geq 2}_{\Kl}(p)_1, \omega(k_1, k_2)) \\
    H^0( X^{\geq 2}_1, \omega(k_1, k_2)) \ar[u] \ar[r, "T"] &  H^0( X^{\geq 2}_1, \omega(k_1, k_2)) \ar[u]
   \end{tikzcd}
  \]
  Moreover, if $k_1-k_2 >0$, $U \circ \tilde{T}  = U^2$.
 \end{lemma}

 \begin{proof}
  Under the assumptions $k_1 + k_2 +1  >  k_2 +3$ and $2k_2 + k_1 > k_2 +3$, only the ``\'etale'' part of the correspondence contributes modulo $p$: the cohomological correspondence $T$ is supported on the component $\mathfrak{C}^{et,et} \times \Spec \FF_p$. The second projection $q_2 : \mathfrak{C}^{et,et} \rightarrow \mathfrak{X}^{\geq 1}$ lifts to $\tilde{q_2} : \mathfrak{C}^{et,et} \rightarrow \mathfrak{X}^{\geq 1}_{\Kl}(p)$ by sending $(A,L)$ to $(A/L, (A[p]+L)/L)$.  It follows that there is a commutative diagram:
  \[
   \begin{tikzcd} H^0( X^{\geq 2}_{\Kl}(p)_1, \omega(k_1, k_2)) \ar[rd] &  \\
    H^0( X^{\geq 2}_1, \omega(k_1, k_2)) \ar[u] \ar[r, "T"] &  H^0( X^{\geq 2}_1, \omega(k_1, k_2))
   \end{tikzcd}
  \]
  which can be completed into the diagram of the proposition.
  It remains to see that $U \circ \tilde{T} = U^2$. We can argue exactly as in the proof of \cite[Lemma 11.1.3]{pilloni17}.
 \end{proof}

 \medskip

 We now examine the situation on the $p$-rank $1$ locus. We have two Hecke operators :
 \[ U \in \mathrm{End}(\mathrm{R}\Gamma(X^{=1}_{\Kl}(p)_1, \omega(k_2,k_1)))\]
 for all $k_1 \geq k_2$, and
 \[ T \in \mathrm{End}(\mathrm{R}\Gamma(X^{=1}_1, \omega(k_2+p-1,k_1+p-1))) \]
 for $k_1 >  2$ and $k_2 + k_1 > 3$ (see \cite[section 7.4 and 10.5.2]{pilloni17} to see how we can obtain these operators by restriction of the cohomological correspondence).

 Observe that the obvious projection $p_1 : X_{\Kl}(p)_1  \rightarrow X_1$ induces an isomorphism  $X^{=1}_{\Kl}(p)_1\simeq X^{=1}_1$ (because over the $p$-rank one locus there is a unique choice for a multiplicative subgroup $C$). We can therefore compare $U$ and $T$.

 In order to do so, let us consider $\mathfrak{X}^{=1}$, $\mathfrak{X}_p^{=1}$, $\mathfrak{C}_1^{=1}$ and $\mathfrak{C}^{=1}$  the completion along the rank one locus of all the formal schemes $\mathfrak{X}^{\geq 1}$, $\mathfrak{X}^{\geq 1}_p$, $\mathfrak{C}_1$ and $\mathfrak{C}$. We have $\mathfrak{X}_p^{=1} = \mathfrak{X}_p^{=1, oo} \cup \mathfrak{X}_p^{=1, m/et}$. This is the decomposition according to whether the kernel of the polarization is a connected group or an extension of \'etale by multiplicative group.

 We have $\mathfrak{C}_1^{=1} = \mathfrak{C}_1^{=1, m} \cup \mathfrak{C}_1^{=1, et} \cup \mathfrak{C}_1^{=1,o}$ according to whether $L_1^\bot$ is multiplicative, \'etale or  bi-infinitesimal.

 We have $\mathfrak{C}^{=1} = \mathfrak{C}^{=1, m,et} \cup \mathfrak{C}^{=1, et,et} \cup \mathfrak{C}^{=1,m,m} \cup \mathfrak{C}^{=1,o,o}$  according to the isogeny $A \rightarrow A/L_1 \rightarrow A/L$.  Namely :

 \begin{itemize}
  \item in case $m,et$, the group $L_1^\bot$ is multiplicative and $L/L_1$ is \'etale,
  \item in case $et,et$, the group $L_1^\bot$ is \'etale and $L/L_1$ is multiplicative,
  \item in case $m,m$, the group $L_1^\bot$ is multiplicative and $L/L_1$ is multiplicative,
  \item in case $o,o$, the group $L_1^\bot$ is bi-infinitesimal and $L/L_1$ is also bi-infinitesimal.
 \end{itemize}

 \begin{lemma}\label{lemUTnord} If  If $k_1   >  p+1$, $k_2 + k_1 >  1 + 2p$ and  $k_1-k_2  > 3(p+1)$, then $T=U$ on $H^0(X^{=1}_1, \omega(k_1, k_2))$.
 \end{lemma}

 \begin{proof} The proof is very similar to  \cite[Lemma 11.1.2.1]{pilloni17}.  We let $T^{\star, \star \star} : q_2^\star \omega(k_2,k_1) \rightarrow q_1^! \omega(k_2,k_1)$ the restriction of $T$ to a component of $\mathfrak{C}^{=1, \star, \star \star}$. Our objective is to prove that $T$ reduces to $T^{et,et}=U$ on $X^{=1}_1$.

  We have to look at all possible types of the isogeny:
  \begin{itemize}
   \item  We have $T^{et,m} :  q_2^\star \omega(k_1,k_2) \rightarrow p^{k_2 +k_1 - k_2-3} q_1^! \omega(k_2,k_1)$ because $A[p] \subset L$,

   \item We have $T^{m,et} :  q_2^\star \omega(k_1,k_2) \rightarrow p^{k_2 +k_1-k_2-3} q_1^! \omega(k_2,k_1)$ because $A[p]^0 \subset L$,

   \item We have $T^{m,m} : q_2^\star \omega(k_1,k_2) \rightarrow p^{k_2 +k_1-k_2-3} q_1^! \omega(k_2,k_1)$ because $A[p] \subset L$.
  \end{itemize}

  The case of the component $o,o$ remains.  Let $\xi = \Spec \bar{k} \rightarrow  X^{=1}_1$ be a generic point. Let $\tilde{\xi} : W(\bar{k}) \rightarrow \mathfrak{X}^{=1}$ be a lift of $\xi$. We may restrict the cohomological correspondence to $\tilde{\xi}$ and we have a map  $T^{o,o}  : (q_1)_\star q_2^\star \omega(k_1,k_2)_{\tilde{\xi}} \rightarrow \omega(k_1,k_2)_{\tilde{\xi}}$.
  For a section $f \in (q_1)_\star q_2^\star \omega(k_1,k_2)_{\tilde{\xi}}$, this map can be written as $T^{oo}f(A,\mu) = \frac{1}{p^{k_2 +3}} \sum_{L} f(A/L, \mu_L)$  where $L$ runs over all subgroups of $A[p^2]$ corresponding to the component $\mathfrak{C}^{=1, o,o}$ (defined over  $\cO_{\CC_p}$),  $\mu$ is a trivialization of $\omega_A$ and $\mu_L$ is a rational trivialization of $\omega_{A/L}$ such that $\pi_L^\star \mu_L = \mu$ for the isogeny $\pi_L : A \rightarrow A/L$.  For any $L$,  the map $\pi_L^\star : \omega_{A/L} \rightarrow \omega_A$  has elementary divisors $(p, \varpi)$ for an element $\varpi$ with $v(\varpi) \geq \frac{1}{p+1}$. We find that $f(A/L, \mu_L) \in p^{k_2} \varpi^{k_1-k_2} \cO_{\CC_p}$.  If $k_1-k_2 > 3(p+1)$ we find that $T^{o,o} ( f(A, \mu)) = 0 \pmod p$.
 \end{proof}

 The local finiteness of $T$ on $\mathrm{R}\Gamma(X,\omega(k_1,k_2)(-D))$ for $k_1$ large enough (depending on $k_2$) follows easily from the previous lemmas. We can now prove :

 \begin{proposition}\label{prop-cool} If $k_1   >  2$ and $k_2 + k_1 >  3$ and $k_1 -k_2 > 3(p-1)$, the map
  \[
   e(T) H^1(X_1^{\geq 1}, \omega(k_1, k_2)(-D)) \longrightarrow e(U) H^1(X^{\geq 1}_{\Kl}(p)_1, \omega(k_1, k_2)(-D))
  \]
  is an isomorphism.
 \end{proposition}

 \begin{proof}   The following complexes compute the cohomology of $\omega(k_1, k_2)(-D)$ over $X_1$ and $X^{\geq1}_{\Kl}(p)_1$:
  \[
   \begin{tikzcd}
   H^0( X^{\geq 2}_{\Kl}(p)_1, \omega(k_1, k_2)(-D)) \ar[r] & \displaystyle \varinjlim_{n} H^0( X^{\geq 1}_{\Kl}(p)_1, \omega{(k_1+n(p-1),k_2+n(p-1))}(-D)/\mathrm{Ha}^n) \\
      H^0( X^{\geq 2}_1, \omega(k_1, k_2)(-D)) \ar[r] \ar[u]& \displaystyle \varinjlim_{n} H^0( X^{\geq 1}_1, \omega{(k_1+n(p-1),k_2+n(p-1))}(-D)/\mathrm{Ha}^n) \ar[u]
   \end{tikzcd}
  \]
  We apply the projector for $U$ on the top and for $T$ on the bottom.

  It follows from lemmas \ref{lem-UTord} and  \ref{lemUTnord} that the left vertical map becomes  surjective  and the right vertical map an isomorphism after applying the projector. This is enough to conclude.
 \end{proof}

 \begin{lemma}\label{lem-inj1}   If $k_1   >  2$, $k_2 + k_1 >  3$ and $k_1 -k_2 > 3(p-1)$, the map $e(T) H^1(\mathfrak{X}^{\geq 1}, \omega(k_1, k_2)(-D)) \rightarrow e(U) H^1(\mathfrak{X}^{\geq 1}_{\Kl}(p), \omega(k_1, k_2)(-D))$ is an isomorphism.
 \end{lemma}

 \begin{proof}
  We first show that  the map is surjective. Since $H^2(\mathfrak{X}^{\geq 1}, \omega(k_1, k_2)(-D)) =0$, we deduce that $e(T) H^1(\mathfrak{X}^{\geq 1}, \omega(k_1, k_2)(-D))/p = e(T) H^1({X}_1^{\geq 1}, \omega(k_1, k_2)(-D))$ and by Nakayama's lemma  and proposition \ref{prop-cool}, we deduce the surjectivity. Let $K$ be the kernel of the map of the proposition. By reduction modulo $p$ we get an exact sequence (since  $ e(U) H^1(\mathfrak{X}^{\geq 1}_{\Kl}(p), \omega(k_1, k_2)(-D))$ is flat):
  \[
   0 \rightarrow K/p \rightarrow  e(T) H^1({X}^{\geq 1}_1, \omega(k_1, k_2)(-D)) \rightarrow e(U) H^1({X}^{\geq 1}_{\Kl}(p)_1, \omega(k_1, k_2)(-D))
  \]
  from which it follows that $K/p=0$ and that $K=0$.
 \end{proof}

 We now prove the analoguous result for classical cohomology.

 \begin{lemma}\label{lem-inj2} If   $k_1+k_2 -3\geq 2$ and $k_1 \geq 3$, the map
  \[ e(T)\mathrm{R}\Gamma(\mathcal{X}, \omega(k_1, k_2)(-D))  \rightarrow e (U)\mathrm{R}\Gamma(\mathcal{X}_{\Kl}(p), \omega(k_1, k_2)(-D))\]
  is an isomorphism.
 \end{lemma}

 \begin{proof}
  We shall consider the smooth admissible $\GSp_4(\Qp)$ representation $\varinjlim_n H^i( \mathcal{X}(p^n), \omega(k_1, k_2)(-D))$. The invariants under $\GSp_4(\Zp)$ are  $H^i(\mathcal{X}, \omega(k_1, k_2)(-D))$ and the invariants under  the Klingen parahoric are $H^i(\mathcal{X}_{\Kl}(p), \omega(k_1, k_2)(-D))$. Let $\pi$ be an irreducible representation with spherical invariants contributing to $\varinjlim_n H^i( \mathcal{X}(p^n), \omega(k_1, k_2)(-D))$. Then $\pi$ is a quotient of a principal series representation and its Hecke parameters are $(\alpha, \beta, \gamma, \delta)$ such that $\alpha \delta = \beta \gamma$. We may assume that they are in increasing $p$-adic valuation. Morover, by Proposition \ref{prop:newtonpoly}, the Newton polygon is above the Hodge polygon, which has slopes in increasing order $k_2-2, 0, k_1+k_2-3, k_1-1$ (or $0,k_2-2, k_1-1, k_1+k_2-2$ but the case of interest to us will be $k_2 \leq 0$ so we will assume for simplicity that we are in the first case), and they have the same initial and end point.
  The $T$ eigenvalue is $p^{2-k_2}( \alpha\beta + \alpha \gamma + \alpha \delta + \beta\delta + \gamma \delta) + p^{1-k_2} (1 + p + p^2) \alpha\delta$.
  Under the assumption that $k_1 \geq 3$, we find that $p^{1-k_2} (1 + p + p^2) \alpha\delta$ has $p$-adic valuation at least one and that  $\pi$ is $T$-ordinary if and only if $\alpha\beta$ has valuation $k_2-2$.
  Under the assumption that $k_1+k_2 -3\geq 2$ we find that $\xi^{-1} \xi' \neq p$ for all $\xi, \xi' \in \{ \alpha, \beta, \gamma, \delta\}$ or that $\beta = p \alpha$ and $p\gamma = \delta$.
  Therefore, if $\pi$ is $T$-ordinary it is either an irreducible principal series type I or a type III(b) representation in Schmidt's classification \cite{MR2114732}. Comparing with Corollary \ref{cor:sphericalordprojection}, we see that $e(T) \pi^{\GSp_4(\Zp)}$ is non-zero if and only if $e(U) \pi^{\Kl(p)}$ is non-zero, and the natural map between the two is an isomorphism.

  It follows that the map $e(T)\mathrm{R}\Gamma(\mathcal{X}, \omega(k_1, k_2)(-D))  \rightarrow e (U)\mathrm{R}\Gamma(\mathcal{X}_{\Kl}(p), \omega(k_1, k_2)(-D))$ induces an isomorphism on every graded piece of some Jordan--H\"older filtration, and hence is an isomorphism as required.
 \end{proof}

 \begin{corollary}   If  $k_1+k_2 -3\geq 2$, $k_1 \geq 3$ and $k_1 -k_2 > 3(p-1)$,  we have a commutative diagram:
  \[
   \begin{tikzcd}
    e(U){H}^1(\mathcal{X}_{\Kl}(p), \omega(k_1, k_2)(-D)) \ar[r] & e(U){H}^1(\mathcal{X}^{\geq 1}_{\Kl}(p), \omega(k_1, k_2)(-D)) \\
    e(T){H}^1(\mathcal{X}, \omega(k_1, k_2)(-D)) \ar[r] \ar[u] & e(T) {H}^1(\mathcal{X}^{\geq 1}, \omega(k_1, k_2)(-D)) \ar[u]
   \end{tikzcd}
  \]
  and the vertical maps are isomorphisms.
 \end{corollary}

 \begin{proof} This follows from lemmas \ref{lem-inj1} and \ref{lem-inj2}.
 \end{proof}

 \begin{lemma} If  $k_1+k_2 -3\geq 2$, $k_1 \geq 3$ and $k_1 -k_2 > 3(p-1)$, the map $eH^1(\mathcal{X}, \omega(k_1, k_2)(-D)) \rightarrow eH^1(\mathcal{X}^{\geq 1}, \omega(k_1, k_2)(-D))$  is injective.
 \end{lemma}

 \begin{proof} This  follows from theorem \ref{thmSGA}.
 \end{proof}

 \begin{corollary} If  $k_1+k_2 -3\geq 2$, $k_1 \geq 3$ and $k_1 -k_2 > 3(p-1)$, the map $eH^1(\mathcal{X}_{\Kl}(p), \omega(k_1, k_2)(-D)) \rightarrow e H^1(\mathcal{X}^{\geq 1}_{\Kl}(p), \omega(k_1, k_2)(-D))$  is injective.
 \end{corollary}

 We observe that for a fixed value of $k_2$, the conditions on the weight are satisfied for all $k_1$ large enough.

 \subsubsection{Injectivity for all values of $k_1$}

 The perfect complex $M^\bullet_{\kappa^{un}, k_2}$ admits an overconvergent variant $M^{\dag, \bullet}_{\kappa^{un}, k_2}$. Its main properties are (see section 13 and 14 of \cite{pilloni17}) :

 \begin{enumerate}

  \item If $k_1 + k_2 > 3$, $k_1-k_2 >2$, then $M^{\dag, \bullet}_{k_1, k_2} = e \mathrm{R}\Gamma ( \mathcal{X}_{\Kl}(p), \omega(k_1, k_2))(-D))$.

  \item If $k_1 + k_2 >3$ and $k_1 \geq k_2$, there is an  isomorphism
  \[ H^0(M^{\dag, \bullet}_{k_1, k_2}) =  e H^0 ( \mathcal{X}_{\Kl}(p), \omega(k_1, k_2))(-D)) =0\]
  and an injective map
  \[ H^1(M^{\dag, \bullet}_{k_1, k_2}) \leftarrow e H^1 ( \mathcal{X}_{\Kl}(p), \omega(k_1, k_2))(-D)).\]
 \end{enumerate}

 We consider the two functions defined for $k_1  > 3-k_2$ : $d(k_1) = \dim e H^1( \mathcal{X}_{\Kl}(p), \omega(k_1, k_2)(-D))$, $d'(k_1) = \dim H^1(M^{\dag, \bullet}_{k_1, k_2})$, and $d^{ord} (k_1) = \dim e H^1( \mathcal{X}^{\geq 1}_{\Kl}(p), \omega(k_1, k_2)(-D))$.

 We have proved that $d(k_1) \geq d^{ord}(k_1)$ and that equality holds if $k_1$ is large enough. We have $d(k_1) \leq d'(k_1)$ and equality holds if $k_1$ is large enough.    Moreover, $d^{ord}$  and $d'$ are locally constant if we endow $\ZZ_{\geq 4-k_2}$ with the  topology  of characters of $\Zp^\times$. We deduce that $d^{ord} = d'$ and finally that $d=d^{ord}$.

 Therefore $d(k_1) =  d^{ord}(k_1)$ for all values of $k_1$, completing the proof of Theorem \ref{thm:vincent}.\qed


\section{Construction of the p-adic pushforward map}
\label{sect:p-adic-pushfwd}


In this section, we'll define morphisms from spaces of $p$-adic modular forms for $H = \GL_2 \times_{\GL_1} \GL_2$ to the $p$-adic $H^1$ spaces for $G$ defined in the preceding section, and show that they interpolate pushforward maps for the usual coherent automorphic sheaves.

\subsection{Level groups}

\begin{notation}
 For $m \ge 1$, define subgroups of $H(\Zp)$ by
 \begin{align*}
 K_{H, 1}(p^m) &\coloneqq \{ h \in H(\Zp): h = \left(\stbt{1}{*}{}{*}, \stbt{1}{*}{}{*}\right) \pmod{p^m}\},\\
 K_{H, \Delta}(p^m) &\coloneqq \{ h: h = \left(\stbt{x}{*}{}{*}, \stbt{x}{*}{}{*}\right) \pmod{p^m} \text{ \textup{for some $x$}}\},\\
 K_{H, 0}(p^m) &\coloneqq \{ h: h = \left(\stbt{*}{*}{}{*}, \stbt{*}{*}{}{*}\right) \pmod{p^m}\}.
 \end{align*}
\end{notation}

The following elementary but important group-theoretic computation underpins our construction. Recall the subgroup $\Kl(p^m) = \left\{ \begin{smatrix}
* & * & * & * \\ &*&*&* \\ &*&*&* \\ &&&*\end{smatrix} \pmod{p^m}\right\} \subset G(\Zp)$ defined in \S\ref{sect:Gsetup} above.

\begin{proposition}
 Let $\gamma \in G(\Zp)$ be any element whose first column is $(1, 1, 0, 0)^T$. Then we have
 \[
 \pushQED{\qed}
 H(\Qp) \cap \gamma \Kl(p^m) \gamma^{-1} = K_{H, \Delta}(p^m).\qedhere
 \popQED
 \]
\end{proposition}

\begin{remark}
 The group $H$ acts on $G / P_{\Kl} \cong \mathbf{P}^3$ with exactly 3 orbits, two closed and one open, and the element $\gamma$ represents the open orbit.
\end{remark}

We write $X_{H, ?}(p^m)_{\QQ}$ for the Shimura variety for $H$ of level $K_H^p K_{H, ?}(p^m)$ (for our fixed prime-to-$p$ level $K^p$), and $\cX_{H, ?}(p^m)$ for the associated rigid-analytic spaces over $\Qp$. The proposition implies that there is a finite morphism of $\QQ$-varieties
\[ \iota_m: X_{H, \Delta}(p^m)_{\QQ} \to X_{G, \Kl}(p^m)_{\QQ}, \]
given by composing $\iota$ with right-translation by $\gamma$ (and this is a closed embedding away from the boundary if $K^p$ is small enough). Our goal is to interpolate pullback and pushforward maps for these embeddings $\iota_m$.

\subsection{Classical modular forms}

For integers $\ell_1, \ell_2 \ge 0$, $L$ any field extension of $\QQ$, we define
\begin{align*}
M_{(\ell_1, \ell_2)}\big(p^m, L\big) &\coloneqq H^0\big( X_{H, 0}(p^m)_L, \omega_H(\ell_1, \ell_2)\big),\\
S_{(\ell_1, \ell_2)}\big(p^m, L\big) &\coloneqq H^0\big( X_{H, 0}(p^m)_L, \omega_H(\ell_1, \ell_2)(-D_H)\big),
\end{align*}
with Hecke operators normalised as in Notation \ref{not:sheaves}, so that if $K_H^p$ has level $N$ and $q \ne p$ is a prime congruent to 1 modulo $N$, then the double coset of $\diag(q, q, q, q) \in H(\QQ_q)$ acts as multiplication by $q^{\ell_1 + \ell_2 - 4}$.

More generally, let $\chi_1, \chi_2$ be characters of $(\ZZ / p^m)^\times$ with values in $L$. We let $M_{(\ell_1, \ell_2)}(p^m, \chi_1, \chi_2, L)$ be the subspace of $H^0\big( X_{H, 1}(p^m)_L, \omega_H(\ell_1, \ell_2)\big)$ on which the quotient $K_{H, 0}(p^m) / K_{H, 1}(p^m)$ acts via the character $\left(\stbt{a}{*} {}{*}, \stbt{b}{*}{}{*}\right)\mapsto \chi_1(a) \chi_2(b)$, and similarly $S_{(\ell_1, \ell_2)}(p^m, \chi_1, \chi_2, L)$. Note that for $q$ as above, the action of $\diag(q, \dots, q) \in H(\QQ_q)$ on this space is now via $q^{\ell_1 + \ell_2 - 4} \cdot \chi_1\chi_2(q)^{-1}$.

\subsection{The Igusa tower for $H$}

Recall that $X_H$ denotes a smooth compactification over $\ZZ_{(p)}$ of the modular curve $Y_H$ of prime-to-$p$ level $K_H = K_H^p H(\Zp)$. Over $Y_H$ we have two elliptic curves $E_1, E_2$, which extend to semiabelian schemes over $X_H$; and the direct sum $E_1 \oplus E_2$ is the pullback via $\iota: X_H \to X_G$ of the universal semiabelian scheme $A_G / X_G$.

Let $\fX_H / \Spf \Zp$ be the $p$-adic formal completion of $X_H$, and $\fX_H^{\ord}$ the ordinary locus in $\fX_H$. Over $\fX_H^{\ord}$ we have an Igusa tower $\fIG_H(p^\infty) = \varprojlim_m \fIG_H(p^m)$, parametrising embeddings $\alpha_i: \mu_{p^n} \into E_i$ for $i= 1, 2$; this is evidently a $(\Gamma \times \Gamma)$-torsor.

\begin{notation}
 Let $\Delta_m$ denote the diagonally-embedded copy of $(\ZZ / p^m \ZZ)^{\times}$ in $\left((\ZZ / p^m \ZZ)^{\times}\right)^2$, and let $\fX_{H, \Delta}^\ord(p^m)$ denote the quotient $\fIG_H(p^m) / \Delta_m$.
\end{notation}

The generic fibres of the formal schemes $\fX_{H, \Delta}^\ord(p^m)$ and $\fIG_H(p^m)$ are naturally open subvarieties of the rigid-analytic varieties $\cX_{H, \Delta}(p^m)$ and $\cX_{H, 1}(p^m)$ respectively.

\subsection{$R$-adic sheaves}\label{sect:bigsheafH}
In analogy with the theory for $G$, we shall let $R$ be a $p$-adic ring with two characters $\kappa_1, \kappa_2: \Gamma \to R$, and consider the sheaf on $\fX_{H, \Delta}^\ord(p)$ defined by
\[
\mathfrak{F}_{H, R}(\kappa_1, \kappa_2) = \left( \pi_\star \cO_{\fIG_H(p^\infty)} \htimes R\right)[\Gamma \dot{\times} \Gamma = (\kappa_1, \kappa_2)] \otimes W_H(0, 0; -4)
\]
where $\pi: \fIG_H(p^\infty) \to \fX_{H, \Delta}^{\ord}(p)$ is the projection map, and $\Gamma \dot{\times} \Gamma$ denotes the group $\{(x, y) \in \Gamma \times \Gamma: x = y \bmod p\}$. This is a rank 1 locally free sheaf of $R$-modules on $\fX_H^\ord(p)$; and we define spaces of $p$-adic modular forms (resp.~cusp forms) by
\begin{align*}
\cM_{(\kappa_1, \kappa_2)}(R) &\coloneqq H^0\left(\fX_{H, \Delta}^\ord(p), \mathfrak{F}_{H, R}(\kappa_1, \kappa_2)\right)^Q,\\
\cS_{(\kappa_1, \kappa_2)}(R) &\coloneqq H^0\left(\fX_{H, \Delta}^\ord(p), \mathfrak{F}_{H, R}(\kappa_1, \kappa_2)(-D)\right)^Q
\end{align*}
where $Q$ is the quotient group $(\Gamma \times \Gamma) / (\Gamma \dot{\times} \Gamma) \cong (\ZZ / p)^\times$. These spaces (and their classical analogues) are independent of the choice of rpcd $\Sigma$, and have an action of Hecke operators away from $p$, with $\diag(q, q, q, q) \in H(\QQ_q)$ for $q = 1 \bmod N$ acting as $q^{\kappa_1 + \kappa_2 - 4}$.

\begin{remark}
 It might seem more natural to work with sheaves on $\fX_H^\ord$, not $\fX_{H, \Delta}^\ord(p)$, which would obviate the need for the tiresome quotient $Q$; but we choose to work over $\fX_{H, \Delta}^\ord(p)$ since it is easier to relate to the theory for $G$.
\end{remark}

For any $k_1, k_2 \in \ZZ$, there is a comparison isomorphism
\[ \operatorname{comp}:\omega_H(k_1, k_2) \xrightarrow{\cong} \mathfrak{F}_{H}(k_1, k_2). \]
Hence, for any finite extension $L/\Qp$ and characters $\chi_1, \chi_2: (\ZZ / p^m)^\times \to L^\times$, we obtain an injection
\[ M_{(k_1, k_2)}(p^m, \chi_1, \chi_2, L) \into \cM_{(k_1 + \chi_1, k_2 + \chi_2)}(\cO_L) \otimes L.
\]
This is \emph{not} Hecke-equivariant in general (unsurprisingly, because the central characters do not match). Rather, it intertwines the action of a double coset $[K^p_H h K^p_H]$, for $h \in H(\Af^p)$, on the left-hand side with
$\widehat{\chi}_1 \widehat{\chi}_2(\det h)^{-1} \cdot [K^p_H h K^p_H]$ on the right-hand side, where $\widehat{\chi}$ is the adelic character associated to $\chi$ as in \S\ref{sect:dirichlet}. However, we shall only use this if $\chi_1 \chi_2 = 1$, in which case the map does indeed become Hecke-equivariant.

There are also versions of these comparisons for cusp forms. Crucially, however, a classical modular form may be a $p$-adic cusp form without being a cusp form in the classical sense. These ``fake cusp forms'' will be vital in our construction of $p$-adic $L$-functions below (as a substitute for our inability to prove an analogue of Theorem \ref{thm:vincent} for cohomology of non-cuspidal sheaves on $G$).

\subsection{Maps between Igusa towers}

The morphism $\fX_{H, \Delta}^{\ord}(p^m) \to \fX_H^{\ord}$ represents the functor of isomorphisms $E_1[p^m]^\circ \cong E_2[p^m]^\circ$, where $E_i[p^m]^\circ$ is the identity component of $E_i[p^m]$. Note that $E_1 \oplus E_2$ is the pullback (via $\iota_{\Sigma}$) of the universal semiabelian surface $A$ over $\fX_G$. We can therefore consider the following diagram of morphisms of formal schemes, for any $m \ge 1$:
\[
\begin{tikzcd}[row sep=small]
 \fIG_H(p^m) \ar[r, "{\tilde\iota_m}"] \ar[d] & \fIG_G(p^m)\ar[d] \\
 \fX_{H, \Delta}^{\ord}(p^m) \ar[r, "\iota_m"]\ar[d] & \fX_{G, \Kl}^{\ge 1}(p^m)\ar[d]\\
 \fX_H \ar[r, "\iota"] & \fX_G
\end{tikzcd}
\]
where the vertical arrows are the natural degeneracy maps, and the horizontal arrows are defined as follows:
\begin{itemize}
 \item $\iota = \iota_{\Sigma}$ is as defined in \S\ref{sect:toroidalfunct} above;
 \item $\iota_m$ maps an isomorphism $E_1[p^m]^\circ \cong E_2[p^m]^\circ$ to its graph, considered as a multiplicative subgroup of $(E_1 \oplus E_2)[p^m]$;
 \item $\tilde{\iota_m}$ maps a pair of embeddings $(\alpha_1, \alpha_2)$ to $\alpha_1 + \alpha_2: \mu_{p^m} \into (E_1 \oplus E_2)[p^m]$.
\end{itemize}

\begin{proposition}
 The morphisms $\iota_m$ and $\tilde{\iota}_m$ are closed embeddings.
\end{proposition}

\begin{proof}
 We first treat $\tilde\iota_m$. The commutativity of the above diagram implies that $\tilde{\iota}_m$ factors through the fibre product $\mathfrak{Z}_m = \fX_H \mathop{\times}_{\fX_G} \fIG_G(p^m)$. Since $\fX_H \into \fX_G$ is a closed embedding, so is the second projection $\mathfrak{Z}_m \into \fIG_G(p^m)$. Hence it suffices to show that the morphism $\fIG_H(p^m) \to \mathfrak{Z}_m$ is a closed embedding.

 We shall show that $\fIG_H(p^m)$ is in fact a component of $\mathfrak{Z}_m$. The space $\mathfrak{Z}_m$, by construction, classifies embeddings $\alpha: \mu_{p^n} \into (E_1 \oplus E_2)[p^m]$. Composing $\alpha$ with the first and second projections, we obtain morphisms of group schemes $\alpha_i: \mu_{p^m} \to E_i[p^m]$ over $\mathfrak{Z}_m$. By Theorem IX.6.8 of SGA III, there is an open-and-closed formal subscheme $\mathfrak{Z}_m^\circ \subseteq \mathfrak{Z}_m$ over which the maps $\alpha_i$ are both embeddings\footnote{Note that this is a specific property of homomorphisms $G \to H$ of group schemes with $G$ of multiplicative type; it is not true for more general finite flat group schemes $G$. We are grateful to Laurent Moret-Bailly for pointing out this reference, in response to a MathOverflow question of ours (\href{https://mathoverflow.net/questions/309837/kernels-of-homomorphisms-of-group-schemes}{link}).}. The map $\tilde\iota_m$ clearly factors through $\mathfrak{Z}_m^\circ$, and the projections $\alpha_i$ determine a map $\mathfrak{Z}_m^\circ \to \fIG_H(p^m)$ which is an inverse to $\tilde\iota_m$. Hence $\tilde\iota_m$ defines an isomorphism between $\fIG_H(p^m)$ and $\mathfrak{Z}_m^\circ$, and in particular it is a closed embedding into $\mathfrak{Z}_m$ and hence also into $\fIG_G(p^m)$.

 To obtain the statement for $\iota_m$, we note that $\tilde\iota_m$ intertwines the action of $\Delta_m \subset ((\ZZ / p^m \ZZ)^\times)^2$ on $\fIG_H(p^m)$ with the natural $(\ZZ / p^m \ZZ)^\times$-action on $\fIG_G(p^m)$, so we can recover $\iota_m$ by passage to the quotient.
\end{proof}

\begin{lemma}
 \label{lem:cartesian}
 For any $m \ge 1$, the diagrams
 \[
  \begin{tikzcd}
   \fIG_H(p^m) \rar["\tilde\iota_m"] \dar & \fIG_G(p^m)\dar \\
   \fX_{H, \Delta}^\ord(p) \rar["\iota_1"] & \fX_{G, \Kl}^{\ge 1}(p)
  \end{tikzcd}
  \quad\text{and}\quad
  \begin{tikzcd}
   \fX_{H, \Delta}^\ord(p^m) \rar["\iota_m"] \dar & \fX_{G, \Kl}^{\ge 1}(p^m)\dar\\
   \fX_{H, \Delta}^\ord(p) \rar["\iota_1"] & \fX_{G, \Kl}^{\ge 1}(p)
  \end{tikzcd}
 \]
 are Cartesian.
\end{lemma}

\begin{proof}
 The first statement follows from the fact that if $\alpha: \mu_{p^m} \into (E_1 \oplus E_2)[p^m]$ is an embedding and the projections $\alpha_i$ of $\alpha$ to the $E_i$ are injective restricted to $\mu_p$, then the $\alpha_i$ are in fact injective, so $\alpha$ comes from a point of $\fIG_H(p^m)$. The second statement is similar (replacing the embeddings with their images).
\end{proof}

We can restate this in terms of a compatibility of ``big'' sheaves on $\fX_{H, \Delta}^\ord(p)$. We let $\pi_G$ denote the projection $\fIG_G(p^\infty) \to \fX_{G, \Kl}^{\ge 1}(p)$; and we let $\pi_H$ denote the projection $\fIG_H(p^\infty) \to \fX_{H, \Delta}^{\ord}(p)$. Then the Cartesian diagram above translates into the following equality of sheaves on $\fX_{H, \Delta}^{\ord}(p)$:
\[
\iota_1^\star\left( \pi_{G, \star} \cO_{\fIG_G(p^\infty)} \right) = \pi_{H, \star} \cO_{\fIG_H(p^\infty)}.
\]
We note that $\pi_{G, \star} \cO_{\fIG_G(p^\infty)}$ is a sheaf of $\Gamma$-modules. On the other hand, $\pi_{H, \star} \cO_{\fIG_H(p^\infty)}$ has an action of the larger group $\Gamma \dot{\times} \Gamma \coloneqq \{ (a, b) \in \Gamma \times \Gamma: a = b \bmod p\}$; if we regard it as a sheaf of $\Gamma$-modules by restriction to the diagonal $\Delta_\infty \cong \Gamma$, then the above isomorphism is $\Gamma$-equivariant.

\subsection{Pullback of $R$-adic sheaves}

We now let $R$ be any $p$-adic ring equipped with two continuous characters $\lambda_1, \lambda_2: \Gamma \to R^\times$; and let $k_2 \in \ZZ$. If we define\footnote{We find it convenient to use additive notation for the group law on characters, since we shall frequently use the embedding of $\ZZ$ into the character group.} $\kappa = \lambda_1 + \lambda_2 - k_2$, then we have defined above the sheaf $\fF_{G, R}(\kappa, k_2)$ of $R$-modules on $\fX_{G, \Kl}^{\ge 1}(p)$ by
\[
\fF_{G, R}(\kappa, k_2) = \left( \pi_{G, \star} \cO_{\fIG_G(p^\infty)} \hat\otimes R\right)[\Gamma = \kappa - k_2] \otimes \omega_G(k_2, k_2; 2k_2- 6).
\]
The Cartesian property of the first diagram of Lemma \ref{lem:cartesian} gives an isomorphism of $R$-module sheaves on $\fX_H^\ord(p)$,
\[
\iota_1^\star\left(\fF_{G, R}(\kappa, k_2)\right) =  \left( \pi_{H, \star} \cO_{\fIG_H(p^\infty)} \hat\otimes R\right)[\Delta_\infty = \kappa-k_2] \otimes \omega_H(k_2, k_2; 2k_2 - 6).
\]
Since the character $\kappa-k_2 = \lambda_1 + \lambda_2-2k_2$ of $\Delta_\infty$ is the restriction of the character $(\lambda_1 - k_2, \lambda_2 - k_2)$ of $\Gamma \dot{\times} \Gamma$, we have a canonical inclusion
\[ \left( \pi_{H, \star} \cO_{\fIG_H(p^\infty)} \hat\otimes R\right)[\Gamma \dot{\times} \Gamma = (\lambda_1 - \kappa_1, \lambda_2 - \kappa_2)] \into \left( \pi_{H, \star} \cO_{\fIG_H(p^\infty)} \hat\otimes R\right)[\Delta_\infty = \kappa-k_2],\]
and hence an inclusion of sheaves
\[ \fF_{H, R}(\lambda_1, \lambda_2) \{-2\} \into \iota_1^\star\left(\fF_{G, R}(\kappa, k_2)\right), \]
where $\{-2\}$ indicates tensor product with the line bundle $\omega_H(0, 0; -2)$. (This is trivial as an abstract sheaf, but twists the Hecke action by the character $\|\det\|$, so that for $h \in H(\Af^p)$, the action of $[K_H^p h K_H^p]$ on the cohomology of $\fF_{H, R}(\lambda_1, \lambda_2)$ corresponds to the action of $\|\det h\|^{-1} [K_H^p h K_H^p]$ on the cohomology of $\iota^* \fF_{G, R}(\kappa, k_2)$.)

\begin{proposition}
 \label{prop:pullbackcomp}
 For any integers $\ell_1, \ell_2, k_1, k_2$ with $\min(\ell_1, \ell_2, k_1) \ge k_2$ and $\ell_1 + \ell_2 =  k_1 + k_2$, we have the following commutative diagram of sheaves on $\fX_{H, \Delta}^{\ord}(p)$:
 \[
  \begin{tikzcd}
   \omega_H(\ell_1, \ell_2)\{-2\} \dar \ar[hook, r] & \iota_1^\star \omega_G(k_1, k_2)\dar\\
   \fF_H(\ell_1, \ell_2)\{-2\} \ar[hook, r]  & \iota_1^\star\left(\fF_G(k_1, k_2)\right)
  \end{tikzcd}
 \]

 Here the right-hand vertical arrow is the pullback via $\iota_1$ of the comparison morphism for $G$, and the left-hand arrow the analogous comparison morphism for $H$.
\end{proposition}

\begin{proof}
 We may assume without loss of generality that $k_2 = 0$. It suffices to prove the statement modulo $p^n$ for each $n$. This arises (via adjunction) from the following diagram of sheaves on the mod $p^n$ reduction of $\fIG_H(p^n)$ (where we omit the central character terms for brevity)
 \[
  \begin{tikzcd}
   \omega_H(\ell_1, \ell_2) \ar[hook, r]\dar & \omega_G(k_1, 0)\dar\\
   \cO_{\fIG_H(p^n)_n} \rar[equal] &\tilde\iota_n^\star\left(\cO_{\fIG_G(p^n)_n}\right)
  \end{tikzcd}
 \]
 Here the right vertical arrow is given by pullback to the universal $\mu_{p^n}$-subgroup composed with the Hodge--Tate period map. By construction, this restricts to the standard trivialisation of either $\omega_H(1, 0)$ or of $\omega_H(0, 1)$ modulo $p^n$.
\end{proof}

\begin{remark}
 As noted in \cite[Remark 9.4.1]{pilloni17}, the comparison morphism $\omega_G(k_1, k_2) \to \fF_G(k_1, k_2)$ can be interpreted as projection onto the highest weight vector in the representation $W_G(k_1, k_2)$ of $\GL_2$ (the Levi factor of the Siegel parabolic in $\operatorname{Sp}_4$). In this optic, our choice of embedding of $\fIG_H(p^\infty)$ into $\fIG_G(p^\infty)$ corresponds to acting on this highest weight vector by the element $\stbt{1}{}{1}{1}$ of $\GL_2$, so that it projects nontrivially to every direct summand for the action of the subgroup $\GL_1 \times \GL_1$ (the Levi factor of the Borel in the derived subgroup of $H$). This is an analogue for coherent automorphic sheaves of the branching computations for \'etale sheaves in \cite{LSZ17}.
\end{remark}

\subsection{Pushforward of $R$-adic sheaves}

Let $X^{\ge 1}_{G, \Kl}(p)_n$ denote the mod $p^n$ reduction of the formal scheme $\fX^{\ge 1}_{G, \Kl}(p)$. Note that this is a smooth and quasi-projective $(\ZZ / p^n)$-scheme, and thus its relative dualising complex is defined, and is isomorphic to a shift of the sheaf of top-degree differentials; and similarly for $\fX_{H, \Delta}^{\ord}(p)$. Carrying out the same construction as in \S\ref{sect:pfwdsheaf} for these truncated schemes, we obtain homomorphisms
\[
H^0\left( X_{H, \Delta}^{\ord}(p)_n,\iota_1^\star(M) \otimes \omega_H(-1, -1;0)(-D_H)\right)
\longrightarrow
H^1\left(X_{G, \Kl}^{\ge 1}(p)_n, M(-D_G) \right),
\]
for $M$ any flat quasi-coherent sheaf on $X^{\ge 1}_{G, \Kl}(p)_n$ (not necessarily of finite rank). If $\mathfrak{M} = (M_n)_{n \ge 1}$ is a flat formal Banach sheaf on $\fX_{G, \Kl}(p)$ in the sense of \cite[\S 12.6]{pilloni17}, then the above morphisms are compatible as $n$ varies and assemble into a map
\[
\iota_{\star}: H^0\left( \fX_{H, \Delta}^{\ord}(p), \iota_1^\star(\mathfrak{M}) \otimes \omega_H(-1, -1;0)(-D_H)\right)
\longrightarrow
H^1\left(\fX_{G, \Kl}^{\ge 1}(p), \mathfrak{M}(-D_G) \right).
\]
In particular, this applies to the sheaf $\fF_{G, R}(\kappa, k_2)$ defined above (assuming the coefficient ring $R$ to be $\Zp$-flat). If we choose characters $\lambda_1, \lambda_2: \Zp^\times \to R^\times$ with $\lambda_1 + \lambda_2 = \kappa + k_2 - 2$, then the preceding section gives a morphism
\[ \fF_H(\lambda_1+1, \lambda_2+1) \otimes \omega_H(0, 0; -2) \to \iota_1^\star \left(\fF_G(\kappa, k_2)\right).\]
Composing this with the above coboundary morphism, and noting that $\fF_H(\lambda_1+1, \lambda_2+1) \otimes \omega_H(-1, -1; -2) = \fF_H(\lambda_1, \lambda_2)$, gives a map of $R$-modules
\begin{equation}
\label{eq:pwfdR}
\iota_\star: \cS_{(\lambda_1, \lambda_2)}(R) \longrightarrow
H^1\left(\fX_{G, \Kl}^{\ge 1}(p), \fF_G(\kappa, k_2)(-D_G) \right).
\end{equation}
By construction, this map is compatible with base-change in $(R, \lambda_1, \lambda_2)$.

\subsection{Comparison with classical pushforward}

 We want to compare the map $\iota_\star$ of \eqref{eq:pwfdR} with the analogous construction for algebraic varieties. For our purposes, it suffices to carry this out after inverting $p$, so we shall use the language of rigid-analytic spaces. We choose integers $k_1, k_2, \ell_1, \ell_2$ such that $\min(\ell_1+1, \ell_2 + 1, k_1) \ge k_2$ and $\ell_1 + \ell_2 = k_1 + k_2 - 2$.

Let $\cX^{\ge 1}_{G, \Kl}(p^m)$ be the rigid-analytic generic fibre of the formal scheme $\fX^{\ge 1}_{G, \Kl}(p^m)$. On the algebraic side, let $X_{G, \Kl}(p^m)_{\Qp}$ be the base-change to $\Qp$ of the toroidal compactification of the Shimura variety of level $K_{G, \Kl}(p^m)$ (using the same rpcd $\Sigma$ we used at spherical level; this may not be smooth for level $K_{G, \Kl}(p^m)$, but this does not matter). If $\cX_{G, \Kl}(p^m)$ is the associated rigid space, then we have a natural open immersion
\[
\cX^{\ge 1}_{G, \Kl}(p^m) \into \cX_{G, \Kl}(p^m).
\]

We can argue similarly with $H$ in place of $G$ to obtain an open immersion
\[ \cX_{H, \Delta}^{\ord}(p^m) \into \cX_{H, \Delta}(p^m).\]
There is a finite morphism of $\QQ$-varieties $X_{H, \Delta}(p^m)_{\QQ} \to X_{G, \Kl}(p^m)_{\QQ}$, injective away from the boundary, given by composing the standard embedding $\iota: H \into G$ with translation by the element $\gamma$ above. We denote this morphism also by $\iota_m$. We can therefore obtain two finite maps of rigid spaces $\cX_H^\ord(p^m) \to \cX_{G, \Kl}^{\ge 1}(p^m)$: either as the generic fibre of the formal-scheme morphism $\iota_m$ of the previous section, or as the restriction of the analytification of the map of $\QQ$-varieties we have just constructed. A simple check using the explicit description of the action of $G(\hat{\ZZ})$ on the moduli problem shows that these two morphisms coincide.

From this equality (and the compatibility of the formal-analytic and rigid-analytic dualizing complexes) we obtain the following commutative diagram. Suppose $\cO$ is the ring of integers of a finite extension $L / \Qp$, and $\chi_1, \chi_2$ are characters $(\ZZ / p^m)^\times \to \cO^\times$ (not necessarily primitive) with $\chi_1 \chi_2 = \mathrm{id}$, we have a commutative diagram
\[
 \begin{tikzcd}
  S_{(\ell_1, \ell_2)}(p^m, \chi_1, \chi_2, L) \rar["\iota_{m, \star}"]\dar[hook] &
  H^1\left(X_{G, \Kl}(p^m)_{L}, \omega_G(k_1, k_2)(-D_G) \right) \dar \\
  \cS_{(\ell_1 + \chi_1, \ell_2 + \chi_2)}(\cO_L) \otimes L \rar["\iota_\star"] &
  H^1\left(\fX_{G, \Kl}^{\ge 1}(p), \fF_G(k_1, k_2)(-D_G) \right)\otimes L.
 \end{tikzcd}
\]
where the lower horizontal arrow is the map of \eqref{eq:pwfdR}.

\begin{remark}
 We could also relax the assumption that $\chi_1 \chi_2$ be the identity, and replace $\fF_G(k_1, k_2)$ on the right-hand side with $\fF_G(k_1 + \chi, k_2)$ where $\chi = \chi_1\chi_2$. However, we shall not use this additional generality in the present paper.
\end{remark}


  \section{Automorphic cohomology and periods}
\label{sect:coherentcoh}


We now use the $p$-adic theory of the previous sections to study the $p$-adic variation of period integrals, given by pairing a class in $H^0$ of an automorphic vector bundle over $H$ with a class in $H^2$ over $G$.

\subsection{Discrete-series automorphic representations}

Let $r_1 \ge r_2 \ge 0$ be integers. Then there exists a unique \emph{generic} discrete series representation $\Pi_\infty$ of $\GSp_4(\RR)$ whose central character has finite order, and which has non-zero $(\mathfrak{g}, K)$-cohomology with coefficients in the appropriate twist of $V(r_1, r_2)$. This is the representation denoted $\Pi_{k_1, k_2}^W$ in \cite[\S 10]{LSZ17}, taking the parameters $(k_1, k_2)$ of \emph{op.cit.} to be $(r_1 + 3, r_2 + 3)$.

We fix -- for the remainder of this paper -- a globally generic automorphic representation $\Pi$ of $\GSp_4(\AA_\QQ)$ whose component at $\infty$ is this discrete series representation. We also assume that $\Pi$ is not a Saito--Kurokawa lifting. Via the classification of cuspidal automorphic representations of $\GSp_4$ announced in \cite{arthur04} and proved in \cite{geetaibi18}, this leaves exactly two possibilities:

\begin{itemize}
 \item (``Yoshida type'') $\Pi$ lifts to a non-cuspidal automorphic representation of $\GL_4$ of the form $\pi_1 \boxplus \pi_2$, where the $\pi_i$ are cuspidal automorphic representations of $\GL_2$ generated by holomorphic modular forms of weights $r_1 + r_2 + 4$ and $r_1 - r_2 + 2$ respectively.

 \item (``General type'') $\Pi$ lifts to a cuspidal automorphic representation of $\GL_4$.
\end{itemize}

 Note that in the present work, unlike its predecessor \cite{LSZ17}, we \emph{shall} allow the case of Yoshida lifts, since it presents no additional difficulties to do so (although in this special case there exist other, simpler proofs of our main results).

 The ``automorphic'' normalisations are such that $\Pi$ is unitary, and its central character is of finite order, thus equal to $\widehat{\chi}_{\Pi}$ for some Dirichlet character $\chi_\Pi$. Note that the sign of this character is $(-1)^{r_1 - r_2}$. We are most interested in the ``arithmetic'' normalisation of the finite part, $\Pif' \coloneqq \Pif \otimes \|\cdot\|^{-(r_1 + r_2)/2}$; this is definable over a number field (while $\Pif$ in general is not).

\subsection{Coherent cohomology}

 We define $L_i$, for $0 \le i \le 3$, to be the irreducible $M_{\Sieg}$-representations with the following highest weights:
 \begin{align*}
 L_0 &: \lambda(r_1 + 3, r_2 + 3; m) & L_1 &: \lambda(r_1 + 3, 1-r_2; m) \\
 L_2 &: \lambda(r_2 +2, -r_1; m) & L_3 &: \lambda(-r_2, -r_1; m),
 \end{align*}
 where $m = r_1 + r_2$.

 \begin{remark}
  Note that the highest weights of the representations $L_i$ lie in a Weyl-group orbit, for a suitably shifted action; geometrically, they are vertices of an octagon centred at $(1, 2)$.
 \end{remark}

 The following result follows by standard methods from Arthur's classification:

 \begin{theorem}\label{thm:cohcoh}
  If $\Pi$ is of general type, then for each $0 \le i \le 3$, $\Pif'$ appears with multiplicity 1 as a Jordan--H\"older factor of the $G(\Af)$-representations
  \[ H^i\left(X_{G, \QQ}, [L_i](-D)\right) \otimes \CC \quad\text{and}\quad H^i\left(X_{G, \QQ}, [L_i]\right) \otimes \CC; \]
  moreover, it appears as a direct summand of both representations, and the map between the two is an isomorphism on this summand. If $\Pi$ is of Yoshida type, then the preceding statements apply for $i = 1,2$, while for $i = 0, 3$, $\Pif'$ does not appear as a Jordan--H\"older factor of either representation.

  If $0 \le i \le 3$ and $L$ is any irreducible representation of $M_{\Sieg}$ which is \emph{not} isomorphic to $L_i$, then the localisations of $H^i\left(X_{G, \QQ}, [L](-D)\right)$ and $H^i\left(X_{G, \QQ}, [L]\right)$ at the maximal ideal of the spherical Hecke algebra associated to the $L$-packet of $\Pi$ are zero for all $i$.\qed
 \end{theorem}

 \begin{notation}
  \label{not:HiPif}
  We shall write $H^i(\Pif)$ for the $\Pif'$-isotypical subspace of $H^i\left(X_{G, E}, [L_i](-D)\right)$, for some number field $E$ over which $\Pif'$ is definable.
 \end{notation}

 By Serre duality, we have $\GSp_4(\Af)$-equivariant perfect pairings
 \begin{equation}
  \label{eq:serredual}
  H^{3-i}\left(X_{G, \QQ}, [L_{3-i}]\right) \times H^i\left(X_{G, \QQ}, [L_i](-D)\right) \to \QQ\{r_1 + r_2\},
 \end{equation}
 for each $i$, where $\{m\}$ denotes the character mapping a uniformizer at a prime $\ell$ to $\ell^m$.  In particular, if we choose a vector $\eta \in H^2(\Pif)$, then we may regard it as a linear functional on $H^1\left(X_{G, E}, [L_1]\right)$, factoring through projection to $H^1(\Pif^\vee)$, where $\Pi^\vee$ is the contragredient of $\Pi$.

\subsection{Definition of the period pairing}
 \label{sect:periodpairing}

 Let $(t_1, t_2)$ be integers $\ge 0$ such that $t_1 + t_2 = r_1 - r_2$. Then there exists a non-zero homomorphism of $B_H$-representations
 \[ W_H(1+t_1, 1+t_2; r_1+r_2)  \to L_1 |_{B_H} \otimes \alpha_{G/H}^{-1}, \]
 unique up to scalars (where the $B_H$-actions on both source and target factor through the Levi of $B_H$, identified with the diagonal torus $T$). Hence there is a pushforward map
 \[
 \iota_\star: H^0\left(X_{H, \QQ}, \omega_H(1+t_1, 1+t_2; r_1 + r_2)\middle)
 \to
 H^1\middle(X_{G, \QQ}, [L_1]\right).
 \]
 Note that the source of this map is simply the space of modular forms of weight $(1+t_1, 1+t_2)$ for $H$ (up to a twist by the norm character). Combining this with \eqref{eq:serredual} we obtain a bilinear, $H(\Af)$-equivariant \emph{period pairing}
 \[ H^0\left(X_{H, E}, \omega_H(1+t_1, 1+t_2)\right) \otimes H^2(\Pif) \to E\{r_1 - 2\}, \]
 mapping $f \otimes \eta$ to $\langle \iota_\star(f), \eta \rangle$.

 \begin{remark}
  Note that our period pairing arises by composing the pushforward $H^0(X_H)$ to $H^1(X_G)$ with the Serre duality pairing $H^1(X_G) \times H^2(X_G) \to E$, for suitable coefficient sheaves. This will allow us to study the degree 4 (``spinor'') $L$-series of $\Pi$. Serre duality also gives a pairing $H^0(X_G) \times H^3(X_G) \to E$, and this pairing plays a fundamental role in many works such as \cite{liu20} and \cite{zhang2018special}, where it is applied to study the degree 5 (``standard'') $L$-series of $\Pi$. (It would be interesting to compare the periods arising from the two constructions, but we shall not pursue this here.)
 \end{remark}

\subsection{Real-analytic comparison}
 \label{sect:archimedean}

 We now recall the comparison between the cohomological cup products above, and period integrals for automorphic forms, justifying the terminology ``period pairing''. (We shall prove a more general statement in \S\ref{sect:hodgesplitting} below, allowing our modular forms to be nearly-holomorphic rather than holomorphic, but we give the ``base case'' here for ease of reading.)

 Let $K_\infty = \RR^\times \cdot U_2(\RR)$ denote the maximal compact-mod-centre subgroup of $G(\RR)_+$. The representation $\Pi_\infty$ has two direct summands as a $G(\RR)_+$-representation, $\Pi_\infty = \Pi_{\infty, 1} \oplus \Pi_{\infty, 2}$, which have minimal $K_\infty$-types (Blattner parameters) $\tau_1 = (r_1 + 3, -r_2 - 1)$ and $\tau_2 =  (r_2 + 1, -r_1 - 3)$ respectively. Since the minimal $K_\infty$-type in an irreducible discrete series has multiplicity 1, we have $\dim \Hom_{K_\infty}(\tau_i, \Pi_\infty) = 1$ for $i = 1, 2$.

 \begin{remark}
  After base-extending to $\CC$ we can identify $K_\infty$ with a conjugate of the Siegel Levi $M_{\Sieg}$, and $\tau_1$ and $\tau_2$ correspond to the $M_{\Sieg}$-representations $W_G(r_1 + 3, -r_2-1)$ and $W_G(r_2 + 1, -r_1-3)$ up to twists by the norm character.
 \end{remark}

 Meanwhile, let $H^i(\Pif)_{\CC} = H^i\left(X_{G, \CC}, [L_i](-D)\right)$ denote the base-extension to $\CC$ of the $E$-vector spaces defined in Notation \ref{not:HiPif}.

\begin{theorem}[Harris, Su]
 \label{thm:harris} \
 \begin{enumerate}
  \item [(a)] For $j = 1, 2$ there is a canonical isomorphism of irreducible smooth $G(\Af)$-representations
  \[
   \Hom_{K_\infty}( \tau_j, \Pi)\left\{\tfrac{r_1 + r_2}{2}\right\} \xrightarrow{\ \cong\ }
    H^j(\Pif)_{\CC}.
  \]
  \item [(b)] If $\eta \in H^2(\Pif)_{\CC}$, corresponding to some $K_\infty$-homomorphism $F_\eta: \tau_2 \to \Pi$, then for any integers $t_1, t_2 \ge 0$ with $t_1 + t_2 = r_1 - r_2$, and any holomorphic modular forms $f, g$ of weights $1+t_1, 1+t_2$ respectively, then we have
 \[ \langle \iota_\star\left( f \boxtimes g\right), \eta \rangle = \frac{1}{(2\pi i)^3} \int_{\RR^\times H(\QQ) \backslash H(\AA)} F_{\eta}(v_{t_1, t_2}) f(h_1) g(h_2)\, \mathrm{d}h, \]
 where $v_{t_1, t_2}$ is the standard basis vector of $\tau_2$ of weight $(-t_1-1, -t_2-1)$. Here we identify $f$ and $g$ with functions on $\GL_2(\QQ) \backslash \GL_2(\AA)$.
 \end{enumerate}
\end{theorem}

\begin{proof}
 For weights $(r_1, r_2)$ sufficiently far from the walls of the Weyl chamber, this is proved in \cite{harriskudla92}, as an application of general results in \S\S 3.5 and 3.8 of \cite{harris90b}. It follows from the results of \cite{su-preprint} that the result in fact applies for all regular weights.
\end{proof}

 \subsection{Hecke eigenvalues at $p$}
 \label{sect:hecke}

  Let $p$ be a prime such that $\Pi_p$ is unramified. Recall that we chose above a number field $E$ such that $\Pif' = \Pif \otimes \|\cdot\|^{-(r_1 + r_2)}$ is definable over $E$. We let $\{\alpha, \beta, \gamma, \delta\}$ be the elements of $\overline{E}$ (unique modulo the action of the Weyl group) such that the spin $L$-factor is given by
  \[ L(\Pi_p', s - \tfrac{3}{2}) = L(\Pi_p, s - \tfrac{r_1+r_2+3}{2}) = \left[(1 - \alpha p^{-s}) \dots (1 - \delta p^{-s})\right]^{-1}.\]
  We order these such that $\alpha \delta = \beta \gamma = p^{(r_1+r_2+3)} \chi_{\Pi}(p)$. As noted in the proof of Proposition \ref{prop:newtonpoly} above, the fact that $\Pi_p$ is unramified and generic implies that $(\Pi_p')^{\Kl(p)}$ is 4-dimensional, and the two operators
  \begin{align*}
  \cU_{p, \Kl}  &= p^{-r_2} \cdot \left[ \Kl(p) \diag(p^2, p, p, 1) \Kl(p)\right], &
  \cU_{p, \Kl}' &= p^{-r_2} \cdot \left[ \Kl(p) \diag(1, p, p, {p^2}) \Kl(p)\right]
  \end{align*}
  acting on this space both have eigenvalues $\left\{ \tfrac{\alpha\beta}{p^{r_2 + 1}}, \tfrac{\alpha\gamma}{p^{r_2 + 1}}, \tfrac{\beta \delta}{p^{r_2 + 1}}, \tfrac{\gamma \delta}{p^{r_2 + 1}}\right\}$.

  More generally, we may consider the Hecke operator $\cU_{p, \Kl}$ on the $\Kl(p^m)$-invariants, for any $m \ge 1$. One checks easily that the operators $\cU_{p, \Kl}$ at different levels are compatible with respect to the natural inclusion maps (since they admit a common set of coset representatives), and one has the following compatibility:
  \begin{proposition}
   For $m \ge 2$, the endomorphism $\cU_{p, \Kl}$ at level $\Kl(p^m)$ factors through the canonical inclusion of the $\Kl(p^{m-1})$ invariants in the $\Kl(p^m)$ invariants.\qed
  \end{proposition}

  It follows that the space $(\Pi_p')^{\Kl(p^m)}$ has a $\cU_{p, \Kl}$-equivariant direct-sum decomposition as the sum of $(\Pi_p^{\prime})^{\Kl(p)}$, on which $\cU_{p, \Kl}$ is invertible, and a complementary subspace on which $\cU_{p, \Kl}$ is nilpotent.

  Similar statements hold for the dual $\Pi^\vee$ in place of $\Pi$, and the transpose of $\cU_{p, \Kl}'$ with respect to the natural pairing $(\Pi_p')^{\Kl(p^m)} \times (\Pi_p^{\prime \vee})^{\Kl(p^m)}$ is given by $\chi_\Pi(p) \cdot  \cU_{p, \Kl}$.

  \begin{definition}
   If $\eta \in H^2(\Pif)^{\Kl(p)}$, and $m \ge 1$, then we define $\eta_m \in H^2(\Pif)^{\Kl(p^m)}$ to be the unique vector such that the linear functional $\langle -, \eta_m\rangle$ on $H^1(\Pif^{\vee})^{\Kl(p^m)}$ has the following properties:
   \begin{itemize}
    \item it vanishes on the $\cU_{p, \Kl} = 0$ generalised eigenspace,
    \item it agrees with $\eta$ on the image of the $\Kl(p)$-invariants.
   \end{itemize}
  \end{definition}

  It is clear that the $\eta_m$ are compatible under the ``normalised trace'' maps,
  \[ \eta_m = \tfrac{1}{p^3}\sum_{\gamma \in \Kl(p^m) / \Kl(p^{m+1})} \gamma \eta_{m+1}.
  \]

  We have the following reasonably concrete formula:

  \begin{proposition}
   If $\eta_1 = \eta$ lies in the $\cU_{p, \Kl}' = \lambda$ eigenspace, then we have
   \[ \eta_m = \left(\frac{p^{3-r_2}}{\lambda}\right)^{m-1} \sum_{a \in \ZZ / p^{m-1}} \begin{smatrix}1 \\ &1 \\ &&1 \\ p^ma&&&1 \end{smatrix} {\dfour{1}{p}{p}{p^2}}^{m-1} \eta. \]
  \end{proposition}

  \begin{proof}
   Let $\eta'_m$ denote the right-hand side of the above formula. For $m = 1$ we evidently have $\eta_1' = \eta_1$.

   Since the operator $\cU_{p, \Kl}'$ at level $\Kl(p^m)$ (for any $m \ge 1$) has a set of single-coset representatives which are lower-triangular and commute with $\begin{smatrix}1 \\ &1 \\ &&1 \\ p^ma&&&1 \end{smatrix}$, it follows easily that $\eta_m'$ is a $\cU_{p, \Kl}'$-eigenvector at level $\Kl(p^m)$ with the same eigenvalue as $\eta_1$. Moreover, we can group the coset representatives in the normalised trace map for $\Kl(p^{m-1}) / \Kl(p^{m})$ in such a way that the inner sum becomes $\cU'_{p, \Kl}$ acting on $\eta_1$, and using the eigenvector property of $\eta_1$ gives the trace-compatibility of the $\eta_m'$. Hence $\eta_m' = \eta_m$ for all $m$.
  \end{proof}

  \begin{definition}
   Choose a place $v$ of $E$ above $p$. We say that $\Pi$ is \emph{Klingen-ordinary} at $p$ (with respect to $v$) if the operator $\cU_{p, \Kl}$ on $(\Pi_p')^{\Kl(p)}$ has an eigenvalue $\lambda$ which is a $p$-adic unit.
  \end{definition}

  It follows from the proof of Lemma \ref{lem-inj2} above that the unit eigenvalue $\lambda$ is unique if it exists (and the corresponding eigenspace is 1-dimensional). In particular, $\lambda$ lies in $E_v$, rather than in some finite extension.

 \subsection{Interpolation of cup-products}
  \label{sect:interpcupprod}

  We now apply the above constructions to interpolate cup-products in families. Let $L$ be the completion of $E$ at our place $v \mid p$, and $\cO$ its ring of integers. We shall assume that $r_2 \ge 1$ and $\Pi$ is Klingen-ordinary at $p$ (with respect to the place $v$); and we choose some class $\eta \in H^2\left(\Pif\right)$, lying in the ordinary eigenspace for $\cU_{p, \Kl}'$. Thus $\eta$ defines a linear functional
  \[ e_{\Kl}\cdot H^1\left(X_{G, \Kl}(p)_L, [L_1]\right) \to L. \]

  We also choose a flat $p$-adic $\cO$-algebra $R$, and $\tau_1, \tau_2 : \Zp^\times \to R^\times$ continuous characters such that $\tau_1 + \tau_2 = r_1 - r_2 + 2$. With these notations, \eqref{eq:pwfdR} gives a map
  \begin{align*}
  \iota_\star: \cS_{(\tau_1, \tau_2)}(R) &\to H^1\left( \fX^{\ge 1}_{G, \Kl}(p), \fF_{R}(3+r_1, 1-r_2)(-D)\right)\\
  &= R \htimes_{\Zp} H^1\left( \fX^{\ge 1}_{G, \Kl}(p), \fF_{\Zp}(3+r_1, 1-r_2)(-D)\right).
  \end{align*}

  \begin{definition}
   Given $\cE \in \cS_{(\tau_1, \tau_2)}(R)$, we define an element $\left\langle \iota_\star\left(\cE\right), \eta \right\rangle \in R[1/p]$ as follows: it is the image of $\iota_\star(\cE)$ under the composition of maps
   \begin{align*}
   R[1/p] \htimes_{\Zp} H^1\left( \fX^{\ge 1}_{G, \Kl}(p), \fF(3+r_1, 1-r_2)(-D)\right)
   &\xrightarrow{e_{\Kl}} R[1/p] \htimes_{\Zp} e_{\Kl}\cdot H^1\left( \fX^{\ge 1}_{G, \Kl}(p), \fF(3+r_1, 1-r_2)(-D)\right) \\
   &\xrightarrow{\cong}  R[1/p] \otimes_{\Qp} e_{\Kl} \cdot H^1\left( X_{G, \Kl}(p)_{\Qp}, [L_1](-D)\right) \\
   &\xrightarrow{\varpi} R[1/p] \otimes_{\Qp} e_{\Kl} \cdot H^1\left( X_{G, \Kl}(p)_{\Qp}, [L_1]\right)\\
   &\xrightarrow{\langle -, \eta\rangle} R[1/p].
   \end{align*}
   Here the second map is the isomorphism of Theorem \ref{thm:vincent}, and the map $\varpi$ is the``forget cuspidality'' map, given by the natural inclusion of sheaves $[L_1](-D) \into [L_1]$.
  \end{definition}

  \begin{proposition}
   \label{prop:interp1}
   Let $\phi: R \to L$ be a continuous ring homomorphism such that
   \begin{itemize}
    \item the composites $\phi \circ \tau_i: \Zp^\times \to L^\times$ are the algebraic characters $x \mapsto x^{t_i}$, for some integers $t_i \ge 1$ with $t_1 + t_2 = r_1 - r_2 + 2$;
    \item the $p$-adic modular form $\cE_\phi \in \cS_{(t_1, t_2)}(K^p_H; L)$ given by specialising $\cE$ is the restriction to $\cX_H^{\ord}$ of a classical holomorphic form $E_\phi \in M_{(t_1, t_2)}(p, L)$.
   \end{itemize}
   Then we have
   \[
   \phi\left( \left\langle \iota_{\star}\left(\cE\right), \eta \right\rangle \right) = \left\langle \iota_{1, \star}(E_\phi), \eta \right\rangle.
   \]
  \end{proposition}

  \begin{proof}
   The construction of the pairing $\langle -, - \rangle$ is compatible with base-change in $R$, so it suffices to assume that $R = \cO$, $\phi$ is the identity map, and $\cE$ is the $p$-adic modular form associated to a classical modular form $E$.

   If $E$ is cuspidal as a classical modular form, then the result is essentially a restatement of the commutativity of the diagrams in the previous section (in the simple case where $m = 1$ and the $\chi_i$ are trivial). However, we are working in a slightly greater degree of generality, where we allow $E$ to be a ``fake cusp form'' in the sense of \S\ref{sect:bigsheafH} above, so it can be non-vanishing at some components of $D_H$ lying outside the ordinary locus; and this is rather more delicate.

   We first summarise why the obvious argument does not work. Let $(k_1, k_2) = (r_1 + 3, 1-r_2)$. From $E$ we can form the following two cohomology classes:
   \begin{itemize}
    \item An ``algebraic'' class $z^{\mathrm{alg}} = \iota_{1, \star}(E) \in H^1\left(X_{G, \Kl}(p)_{L}, \omega_G(k_1, k_2)\right)$.
    \item A ``analytic'' class $z^{\mathrm{an}} = \iota_{\star}(\cE) \in  H^1\left(\cX_{G, \Kl}^{\ge 1}(p), \cF_G(k_1, k_2)(-D_G) \right)$.
   \end{itemize}
   These are compatible, in the sense that the restriction of $z^{\mathrm{alg}}$ to $\cX_{G, \Kl}^{\ge 1}(p)$ coincides with the image of $z^{\mathrm{an}}$ in $H^1\left(\cX_{G, \Kl}^{\ge 1}(p), \cF_G(k_1, k_2) \right)$ under the ``forget cuspidality'' map $\varpi$.

   The classicity theorem for ordinary cohomology, used in the construction of our pairing, shows that there is some $z^{\mathrm{alg}}_{\mathrm{cusp}} \in H^1\left(X_{G, \Kl}(p)_{L}, \omega_G(k_1, k_2)(-D)\right)$ whose restriction to $\cX_{\Kl}^{\ge 1}(p)$ is $e_{\Kl} \cdot z^{\mathrm{an}}$. If we can show that the image of $z^{\mathrm{alg}}_{\mathrm{cusp}}$ under the ``forget cuspidality'' map $\varpi$ coincides with $e_{\Kl} \cdot z^{\mathrm{alg}}$, then we are done. However, all that we know is that $\varpi(z^{\mathrm{alg}}_{\mathrm{cusp}})$ and $e_{\Kl} \cdot z^{\mathrm{alg}}$ have the same image in $H^1\left(\cX_{G, \Kl}^{\ge 1}(p), \cF_G(k_1, k_2) \right)$, and since we have no analogue of Theorem \ref{thm:vincent} for the non-cuspidal $p$-adic cohomology, we do not know that the restriction map $H^1\left(X_{G, \Kl}(p)_L, \omega_G(k_1, k_2) \right) \to H^1\left(\cX_{G, \Kl}^{\ge 1}(p),\cF_G{(k_1, k_2)} \right)$ is injective on the ordinary part.

   We work around this by using \emph{overconvergent} cohomology. Theorem \ref{thm:OCclass} gives a classicity result for both of the spaces
   \[ H^1\left(\cX^{\ge 1}_{G, \Kl}(p)^\dagger, \omega_G(k_1, k_2) \right)\quad\text{and}\quad H^1\left(\cX^{\ge 1}_{G, \Kl}(p)^\dagger, \omega_G(k_1, k_2)(-D) \right). \]
   Performing the same construction as before for the overconvergent spaces, we obtain a class $z^{\dag} \in H^1\left(\cX^{\ge 1}_{G, \Kl}(p)^\dagger, \omega_G(k_1, k_2)(-D) \right)$ with the following properties: the restriction of $z^{\dag}$ to the $p$-rank $\ge 1$ locus (forgetting the overconvergence) is $z^{\mathrm{an}}$; and the image of $z^{\dag}$ in $H^1\left(\cX^{\ge 1}_{G, \Kl}(p)^\dagger, \omega_G(k_1, k_2)\right)$ (forgetting the cuspidality) coincides with the restriction of $z^{\mathrm{alg}}$.

   The classes $\varpi(z^{\mathrm{alg}}_{\mathrm{cusp}})$ and $e_{\Kl} \cdot z^{\mathrm{alg}}$ have the same image in $H^1\left(\cX^{\ge 1}_{G, \Kl}(p)^\dagger, \omega_G(k_1, k_2)\right)$ (namely $e_{\Kl} \cdot \varpi(z^{\dag})$). Using the classicity theorem for this overconvergent cohomology, we conclude that $\varpi(z^{\mathrm{alg}}_{\mathrm{cusp}})$ and $e_{\Kl} \cdot z^{\mathrm{alg}}$ coincide.
  \end{proof}

 \subsection{Higher level specialisations}

  We also have a version of this theorem for specialisations that are classical of higher $p$-power levels, using the vectors $\eta_m$ at level $\Kl(p^m)$ constructed from $\eta$ as in \S\ref{sect:hecke}.

  \begin{theorem}
   \label{thm:interp2}
   Let $\phi: R \to L$ be a continuous ring homomorphism such that $\phi \circ \tau_i$ has the form $x \mapsto x^{t_i} \chi_i(x)$, for some $t_i$ as before and finite-order characters $\chi_1, \chi_2$ with $\chi_1 \chi_2 = 1$. Suppose that $\phi(\cE)$ is the image of a classical modular form $E_\phi \in M_{(t_1 + 1, t_2 + 1)}(K_H(p^m), L)$, for some $m \gg 0$.

   Then we have
   \[ \phi\left(\langle \iota_\star(\cE), \eta \rangle\right) =  \left\langle \iota_{m, \star}(E_\phi), \eta_m \right\rangle. \]
  \end{theorem}

  \begin{proof}
   By the same arguments as above, $\phi\left(\langle \iota_\star(\cE), \eta \rangle\right)$ is the image of $\iota_{m, \star}(E_\phi)$ under the unique linear functional on $H^1(X_{G, \Kl}(p^m)_{L}, [L_1])$ which factors through the ordinary idempotent $e_{\Kl}$ and agrees with $\langle -, \eta\rangle$ on the image of $H^1(X_{G, \Kl}(p)_{L}, [L_1])$. This is exactly the definition of $\eta_m$.
  \end{proof}


\section{``Nearly'' coherent cohomology}
\label{sect:nearly}

 It is well-known that nearly-holomorphic modular forms can be considered as $p$-adic modular forms. In this section, we shall formulate an analogous statement for our $p$-adic cohomology spaces for $\GSp_4$, and show that the $\Lambda$-adic pushforward map constructed in the previous section is compatible with this additional structure.

 \subsection{``Nearly'' sheaves}

  In this section, we'll consider coherent cohomology with coefficients in certain indecomposable representations of $P_{\Sieg}$. This is needed in order to study pushforwards of non-holomorphic Eisenstein series; we regard it as an analogue for $H^1$ of Siegel varieties of Shimura's theory of nearly-holomorphic modular forms.

  \begin{definition}
   If $V$ is a finite-dimensional algebraic representation of $P_{\Sieg}$, we let $\Fil^n V$ denote the direct sum of the eigenspaces of $W$ on which the torus $\dfour{x}{x}{1}{1}$ acts with weights $\ge n$ (i.e. the weight spaces of weight $\lambda(r_1, r_2; c)$ with $r_1 + r_2 + c \ge 2n$).
  \end{definition}

  This filtration is stable under $P_{\Sieg}$, and the $P_{\Sieg}$-action on the graded pieces $\Gr^n V = \Fil^n V / \Fil^{n+1} V$ factors through $M_{\Sieg}$. We shall apply this to $V = V(r_1, r_2; r_1 + r_2)$, for some integers $r_1 \ge r_2 \ge 0$; in this case the non-zero graded pieces are in degrees $0 \le n \le r_1 + r_2$.

  \begin{definition} \
   \begin{itemize}
    \item Let $\tilde L_1$ denote the representation $V / \Fil^{r_1 + 1} V \otimes W(3, 1; 0)$, and $L_1'$ the smallest filtration subspace $\Gr^{r_1} V \otimes W(3, 1;0)$ of $\tilde L_1$.
    \item Dually, let $\tilde L_2$ denote the $P_{\Sieg}$-representation $\Fil^{r_2} V \otimes W(2, 0; 0)$, and $L_2'$ the top filtration quotient $\Gr^{r_2} V \otimes W(2, 0; 0)$ of $\tilde L_2$.
   \end{itemize}
  \end{definition}

  There is a natural projection $L_1' \twoheadrightarrow L_1$, and a natural inclusion $L_2 \into L_2'$ (both of which are split, but we do not need this).

  \begin{proposition}
   The natural maps
   \[ H^2\left(X_{G, E}, [L_2](-D)\right) \to H^2\left(X_{G, E}, [L'_2](-D)\right) \]
   and
   \[ H^2\left(X_{G, E}, [\tilde L_2](-D)\right) \to H^2\left(X_{G, E}, [L'_2](-D)\right), \]
   induced by the maps of $P_{\Sieg}$-representations $\tilde L_2 \twoheadrightarrow L'_2 \hookleftarrow L_2$, are both isomorphisms on the $\Pif'$ generalised eigenspace (and similarly without $(-D)$).
  \end{proposition}

  \begin{proof}
   This follows readily from the final statement of Theorem \ref{thm:cohcoh}, since the kernel of $\tilde L_2 \twoheadrightarrow L'_2$ and the cokernel of $L_2 \into L'_2$ do not contribute to the $\Pif'$-eigenspace of $H^2$.
  \end{proof}

  We write $\tilde\eta$ for the unique class in $H^2\left(X_{G, E}, [\tilde L_2](-D)\right)[\Pif]$ corresponding to $\eta$ under the above isomorphisms. The linear functional $\langle -, \tilde\eta\rangle$ is thus a homomorphism
  \[ H^1\left(X_{G, E}, [\tilde L_1]\right) \to E, \]
  characterised as follows:
  \begin{itemize}
   \item it factors through the $\Pif^\vee$-isotypical part;
   \item its restriction to $H^1\left(X_{G, E}, [L_1^{\prime}]\right)$ is the composite of the projection $L_1^{\prime} \onto L_1$ and pairing with $\eta$.
  \end{itemize}

  This extended homomorphism will play a role analogous to Shimura's ``holomorphic projection'' operator in the $\GL_2$ theory.

 \subsection{Pullback and pushforward of ``nearly'' sheaves}

  \begin{proposition}[{cf.~\cite[Proposition 4.3.1]{LSZ17}}]
   Suppose $(r_1, r_2; c)$ is a $G$-dominant weight. The restriction to $H$ of the irreducible representation $V_G(r_1, r_2; c)$ of $G$ is a direct sum of distinct irreducible $H$-representations. The representation $V_H(t_1, t_2; c)$ appears as a direct summand if and only if the integers $t_1, t_2$ satisfy $t_1 + t_2 = r_1 + r_2 \bmod 2$ and
   \[ \pushQED{\qed}
   r_1 - r_2 \le t_1 +t_2 \le r_1 + r_2,\quad  |t_1 - t_2| \le r_1 - r_2.\qedhere
   \popQED\]
  \end{proposition}

  Since the cocharacter used to define the filtration is strictly dominant with respect to $B_H$, any representation of $B_H$ has a canonical $B_H$-invariant filtration, with the action on the graded pieces factoring through $T$; and this is compatible with the above branching from $G$ to $H$.

  We shall be particularly interested in $H$-subrepresentations of the form $V_H(t_1, t_2; c)$ with $t_1 + t_2 = r_1 - r_2$; note that there are precisely $r_1 - r_2 + 1$ such subrepresentations. We shall call these subrepresentations \emph{small}.

  Recall now the larger sheaves $\tilde L_2 \onto L_2' \hookleftarrow L_2$ defined above. Since $V_H(t_1, t_2; r_1 + r_2)$ is a direct summand of $V$, we have a projection map
  \[ \tilde L_2 |_{B_H} = \left(\Fil^{r_2} V \otimes W(2, 0; 0)\right)|_{B_H} \to V_H \otimes \lambda(1, 1; 0) \]
  given by the tensor product of the projections $V |_H \to V_H$ and $W(2, 0; 0)|_H \to \lambda(1, 1;0 )$. This determines a homomorphism
  \[
  \iota_{\text{nearly}}^\star:
  H^2\left(X_{G, \QQ}, [\tilde L_2](-D)\middle)
  \to H^2\middle(X_{H, \QQ}, [V_H] \otimes \omega_H(1, 1; 0)(-D)\right).
  \]

  Note that the composition of $\iota_{\text{nearly}}^\star$ with the projection $V_H \to V_H / \Fil^{r_2 + 1} V_H \cong W_H(1-t_1, 1-t_2; r_1 + r_2)$ factors through $ H^2\left(X_{G, \QQ}, [L_2'](-D)\right)$, and we can (and do) normalise such that this composite agrees with $\iota^\star$ on the image of $L_2 \into L_2'$.

  Dually, we obtain a map
  \[
  \iota^{\text{nearly}}_\star: H^0\left(X_{H, \QQ}, [V_H] \otimes \omega_H(1, 1; 0)\middle) \to
  H^1\middle(X_{G, \QQ}, [\tilde L_1]\right).
  \]
  The source of this map is the space of \emph{nearly-holomorphic} modular forms for $H$ of weight $(1+t_1, 1+t_2)$. We thus have an extended period pairing
  \[
  H^0\big(X_{H, E}, [V_H] \otimes \omega_H(1, 1; 0)\big) \otimes H^2(\Pif) \to E,\quad (f, \eta) \mapsto \langle \iota^{\text{nearly}}_\star(f), \tilde\eta \rangle.
  \]
  From the construction of $\tilde\eta$, we see that the restriction of this pairing to the space of holomorphic forms agrees with the period pairing of \S\ref{sect:periodpairing} above.

 \subsection{Archimedean theory: the Hodge splitting}
  \label{sect:hodgesplitting}

  We can compute $H^2(X_{G, \CC}, [\tilde L_2](-D))$ in terms of automorphic forms, using results of Su \cite{su-preprint}. As above, let $K_\infty$ be the standard maximal compact subgroup of $G(\RR)_+$, fixing the point $h = i I_2 \in \cH_2$. The Shimura cocharacter determines a decomposition of $\operatorname{Lie}(G)_{\CC}$ as $\mathfrak{k} \oplus \mathfrak{p}^+ \oplus \mathfrak{p}^-$, where $\mathfrak{k} = \operatorname{Lie}(K_\infty)_{\CC}$, and $\mathfrak{p}^+$ (resp.~$\mathfrak{p}^-$) is identified with the holomorphic (resp.~antiholomorphic) tangent space of $\mathcal{H}_2$ at $h$. The parabolic $P_h \subset G(\CC)$ with Lie algebra $\mathfrak{p} = \mathfrak{k} \oplus \mathfrak{p}^-$ is a conjugate of the Siegel parabolic $P_{\Sieg}$, so we may identify any $P_{\Sieg}$-representation $V$ with a representation of $P_\infty$.

  The main theorem of \emph{op.cit.} gives a canonical and Hecke-equivariant isomorphism
  \[ H^*(X_{G, \CC}, [V]) \cong H^*\!\left(\mathfrak{p}, K_\infty; \mathcal{A}(G)^K \otimes V\right), \]
  for any level $K$ and any algebraic representation $V$ of $P_{\Sieg}$, where $\mathcal{A}(G)$ is the space of automorphic forms on $G$ (twisted by an appropriate power of the norm character so that the central characters match). The relative Lie algebra cohomology on the right-hand side can be computed as the cohomology of the complex with $j$-th term
  \[
   \Hom_{K_\infty}\left( \textstyle{\bigwedge^j}(\mathfrak{p}^-) \otimes V^\vee, \mathcal{A}(G)^K\right).
  \]

  \begin{lemma}
   Let $\Pi_{\infty, 2}$ be the $G(\RR)_+$-submodule of $\Pi_\infty$ described in \S \ref{sect:archimedean} above. Then we have
   \[
    \Hom_{K_\infty}\left(\textstyle{\bigwedge^j}(\mathfrak{p}^-) \otimes  (\tilde L_2)^\vee,\ \Pi_{\infty, 2}\right) = 0
   \]
   for all $j \ne 2$, and for $j = 2$ this space is 1-dimensional and maps isomorphically to its image with $L_2$ in place of $\tilde L_2$.
  \end{lemma}

  \begin{proof}
   An explicit description of the $K_\infty$-types appearing in $\Pi_{\infty, 2}$ is given in \cite{schmidt17}. The minimal $K_\infty$-type is $\tau_2 = (r_2 + 1, -r_1-3)$, and the other $K_\infty$-types lie in the convex cone $\{ \tau_2 + m \cdot (1, 0) + n \cdot (-1, -1): m, n \ge 0\}$.

   On the other hand, the $K_\infty$-types appearing in $(\tilde L_2)^\vee$ are all contained in a different convex cone, $\{ (r_2 + 2, -r_1) + m \cdot (1, 1) + n \cdot (-1, 1): m, n \ge 0\}$. So the weights of $(\tilde L_2)^\vee \otimes \bigwedge^j(\mathfrak{p}^-)$ are contained in the translate of this cone by the highest weight of $\bigwedge^j(\mathfrak{p}^-)$; this translation is by $(0, -2)$ if $j = 1$, by $(-1, -3)$ if $j = 2$, and by $(-3, -3)$ if $j = 3$. So if $j \ne 2$, these regions have empty intersection, whereas if $j = 2$ the intersection consists only of $\tau_2$, which appears in both representations with multiplicity 1.
  \end{proof}

  Thus the $\Pif$-isotypical part of the $(\mathfrak{p}, K_\infty)$-cohomology is represented by a complex concentrated in degree 2, giving a a canonical space of Dolbeault differential forms representing $H^2\big(X_{G, \CC}, [\tilde L_2]\big)[\Pif]$: the vector-valued automorphic forms whose coordinate projections lie in the minimal $K_\infty$-type of $\Pif \otimes \Pi_{\infty, 2}$. Moreover, these differentials in fact take values in $[L_2]$, regarded as a subsheaf of $[\tilde L_2]$ via the Hodge splitting: the projections of these differentials to the other graded pieces of $[\tilde L_2]$ are trivial, since the corresponding $K_\infty$-types do not appear in $\Pi_{\infty, 2}$. Since $\Pi$ is cuspidal, these differential forms are rapidly-decreasing.

  Now let $t_1, t_2$ be integers with $t_1 + t_2 = r_1 - r_2$, as before. The analogue of this $K_\infty$-equivariant splitting for $H$ in place of $G$ is the \emph{Hodge splitting} in the category of $C^\infty$ vector bundles,
  \[ [V_H] \otimes \omega_H(1, 1; 0) \to \omega_H(1 + t_1, 1 + t_2), \]
  which allows nearly-holomorphic modular forms to be interpreted as scalar-valued real-analytic functions on $H(\QQ) \backslash H(\Af)$, transforming under $K_{H,\infty}$ via the character $(1 + t_1, 1 + t_2)$. Evidently these two splittings are compatible (since they are both given by projection to eigenspaces for the action of the centre of $K_\infty$, which is contained in $K_{H, \infty}$), so we deduce the following theorem:

  \begin{theorem}
   Let $E \in H^0\left(X_{H, \CC}, [V_H] \otimes \omega_H(1, 1; 0)\right)$ and let $E_\infty \in H^0\left(X_{H, \CC}, \omega(t_1+1, t_2+1)_{C^\infty}\right)$ be its image under the Hodge splitting. If $\eta$ and $F_{\eta}$ are as in Theorem \ref{thm:harris}, then we have
   \[
   \langle \iota^{\text{nearly}}_\star(E), \tilde\eta \rangle = \frac{1}{(2\pi i)^3}\int_{\RR^{\times} H(\QQ)  \backslash H(\AA)} F_\eta(v_{t_1, t_2}) E_\infty(h)  \, \mathrm{d}h.
   \]
  \end{theorem}

  \begin{proof}
   We know that the $[\tilde L_1]$-valued current representing $\iota^{\text{nearly}}_\star(E)$ is the direct sum of $\iota_\star(E_\infty)$ and some other terms lying in other graded pieces of the sheaf $[\tilde L_1]$. By the previous lemma, $\tilde \eta$ pairs to zero with the latter; on the other hand, we clearly have $\langle \iota^{\text{nearly}}_\star(E), \tilde\eta \rangle = \langle \iota_\star(E_\infty), \eta \rangle$ since $\tilde\eta$ maps to $\eta$ in $[L_2]$.
  \end{proof}

  \begin{remark}
   This construction is strongly analogous to one appearing in Harris' work \cite{harris04}, which also considers period integrals of automorphic forms for $G$ multiplied by nearly-holomorphic Eisenstein series on $H$. However, our treatment differs from Harris' in the following point: in defining $\tilde{\eta}$, we used $G(\Af)$-equivariance to split a cohomology exact sequence on $G$ (relating the cohomology of $[L_2]$ and $\tilde L_2]$). Harris uses instead a splitting on $H$, characterised by equivariance for the action of $H(\Af)$ (see the proof of Proposition 1.10.3 of \emph{op.cit}). It is not clear to us whether these constructions agree in general. 
  \end{remark}

 \subsection{A $p$-adic splitting}

  We now define a partial $p$-adic analogue of the Hodge splitting. Let us briefly recall the definition of the torsor $\cT_G$ defined in \S \ref{sect:toroidal} above. Over the integral model of the open Shimura variety $Y_G$ of prime-to-$p$ level $K^p \cdot G(\Zp)$, we have a locally free sheaf $\cH^1_{\dR}(A)$ of rank 4, with a rank 2 locally free subsheaf $\Fil^1 \cH^1_{\dR}(A)$. There is a canonical extension of $\cH^1_{\dR}(A)$ to a locally free sheaf $\cH^1_{\dR}(A)^{\mathrm{can}}$ on $X_G$, fitting into a short exact sequence of locally free sheaves
  \[ 0 \to \omega_{A^\Sigma} \to \cH^1_{\mathrm{dR}}(A)^\mathrm{can} \to \omega_{A^\Sigma}^\vee \to 0, \]
  where $A^\Sigma$ is the semiabelian variety over $X_G$ extending $A$, and $\omega_{A^\Sigma}$ its conormal sheaf at the identity section. Moreover, if $\fIG_G(p^\infty)^{\ord}$ denotes the preimage in $\fIG_G(p^\infty)$ of the ordinary locus $\fX_G^{\ord} \subset \fX_G$, then over $\fX_G^{\ord}$ there is a splitting
  \[ \cH^1_{\mathrm{dR}}(A)^\mathrm{can} \cong \omega_{A^\Sigma} \oplus \mathcal{U}, \]
  where $\mathcal{U}$ is the \emph{unit root subsheaf} \cite[\S 3.12]{liu-nearly}.

  Over the Igusa tower $\fIG_G(p^\infty)$, we have a morphism of $p$-divisible groups $\alpha: \mu_{p^\infty} \into A^\Sigma[p^\infty]$, and hence a canonical map $\omega_{A_G^\Sigma} \to \cO_{\fIG_G(p^\infty)}$, i.e.~a class $[\alpha] \in H^0(\fIG_G(p^\infty), \omega_{A^\Sigma}^\vee)$.

  \begin{proposition}
   There exists a unique lifting of $[\alpha]$ to a class $[\alpha]^{\mathrm{ur}} \in H^0(\fIG_G(p^\infty), \cH^1_{\dR}(A)^{\can})$ with the following property: its restriction to $\fIG_G(p^\infty)^{\ord}$ takes values in the subsheaf $\cU$.
  \end{proposition}

  \begin{proof}
   By construction, there is a unique class $[\alpha]^{\ord} \in H^0(\fIG_G(p^\infty)^{\ord}, \cH^1_{\dR}(A)^{\can})$ lifting the restriction of $[\alpha]$ and taking values in the unit-root subsheaf $\cU$. Since $\fIG_G(p^\infty)^{\ord}$ is dense in $\fIG_G(p^\infty)$, it follows that $[\alpha]^{ur}$ is unique if it exists.

   Let $\fIG_G(p^\infty)^{\circ}$ denote the preimage in $\fIG_G(p^\infty)$ of the open Shimura variety $Y_G \pmod p$. Since the complement of $\fIG_G(p^\infty)^{\circ} \cup \fIG_G(p^\infty)^{\ord}$ has codimension $\ge 2$ in $\fIG_G(p^\infty)$ (and the sheaf is locally free), it suffices to find a second section $[\alpha]^{\circ}$ over $\fIG_G(p^\infty)^{\circ}$ which coincides with $[\alpha]^{\ord}$ where both are defined.

   We shall carry this out using a generalisation of the construction of the unit-root splitting given in \cite{iovita00}. Recall that if $A/S$ is an abelian scheme over an arbitrary base $S$, then there exists a \emph{universal vectorial extension} of $A$, which is universal among short exact sequences of $S$-group schemes
   \[ 0 \to V \to I \to A \to 0 \]
   where $V$ is a vector group. Moreover, the space of invariant differentials $\operatorname{Inv}(I/S)$ is isomorphic to $H^1_{\dR}(A/S)$, with the subspace $\operatorname{Inv}(A/S)$ corresponding to $\Fil^1$. Taking associated formal groups, we have a short exact sequence $0 \to \hat V \to \hat I \to \hat A \to 0$.

   If we are given a morphism $\alpha: \hat{\mathbf{G}}_m \into \hat A$, and $\hat I_\alpha$ denotes the pullback of this extension along $\alpha$, we have a diagram of exact sequences
   \[
    \begin{tikzcd}
     0 \rar & \hat{V} \rar & \hat{I}    \rar & \hat{A} \rar& 0\\
     0 \rar & \hat{V} \rar\uar[equals]    & \hat{I}_\alpha \rar\uar[hook] & \hat{\mathbf{G}}_m \rar \uar[hook] & 0
    \end{tikzcd}
   \]
   However, any extension of a formal multiplicative group by an additive one must be split, so there is a (necessarily unique) map $\hat{\mathbf{G}}_m \to \hat I_\alpha$ splitting the lower sequence. Thus the composite
   \[ v_\alpha: \operatorname{Inv}(I/S) \to \operatorname{Inv}(I_\alpha/S) \to \operatorname{Inv}\left(\hat{\mathbf{G}}_m/S\right) \]
   is an extension of the pullback map $\operatorname{Inv}(A/S) \to \operatorname{Inv}\left(\hat{\mathbf{G}}_m/S\right) \cong \cO_S$ to $H^1_{\dR}(A/S)$. Moreover, if $p$ is topologically nilpotent on $S$ and $A$ is ordinary, then $\hat A$ is itself of multiplicative type. Hence we have a splitting of the top row, which is clearly compatible with the splitting of the bottom row; and the main result of \cite{iovita00} shows that this construction recovers the unit-root splitting. This gives the required section $[\alpha]^\circ$.
  \end{proof}

  \begin{notation}
   Let $Q$ be the subgroup $P_{\Sieg} \cap \overline{P}_{\Kl}$ of $G$, where $P_{\Sieg}$ is the standard Siegel parabolic subgroup and $\overline{P}_{\Kl} = J^{-1} P_{\Kl} J$ is the lower-triangular Klingen.
  \end{notation}

  Note that $T$ is a Levi subgroup of $Q$, so any weight $\lambda(r_1, r_2; c)$ can be regarded as a representation of $Q$.

  \begin{definition}
   Let $\cT^{\Kl}$ denote the sheaf over $\fX_{G, \Kl}^{\ge 1}(p^\infty)$ parametrising bases $f_1, \dots, f_4$ of $\cH^1_{\dR}(A)^{\can}$ compatible with the filtration and polarisation, and with the additional property that the pullback of $f_4$ to $\fIG_{G, \Kl}^{\ge 1}(p^\infty)$ is a scalar multiple of $[\alpha]^{\mathrm{ur}}$. This is is a reduction of structure of the $P_{\Sieg}$-torsor $\cT$ to a $Q$-torsor over $\fX_{G, \Kl}^{\ge 1}(p^\infty)$.
  \end{definition}

  One can interpret the comparison morphism of \S \ref{sect:Gsetup} using this torsor. Let $k_1 \ge k_2$ be integers. We shall see in Lemma \ref{lemma:HWtheory}(i) below that the inclusion of the highest weight space $\lambda(k_1, k_2)$ into $W_G(k_1, k_2)$ has a $Q$-equivariant splitting. If we pull back further to $\fIG_G(p^\infty)$, the sheaf $[\lambda(k_1 - k_2, 0)]$ has a canonical trivialisation, so we obtain a morphism of sheaves
  \[ [ W_G(k_1, k_2) ] \to [W_G(k_2, k_2)] = (\det \omega_A)^{k_2}. \]
  Pushing back down to $\fX_{G, \Kl}^{\ge 1}(p)$, this gives us a morphism
  \[ [ W(k_1, k_2) ] \to \left( \pi_\star \cO_{\fIG_G(p^\infty)} \right)^{\Gamma = k_1 - k_2} \otimes (\det \omega_A)^{k_2}, \]
  which is the sought-after comparison morphism. The advantage of this new interpretation is that we can easily see how to extend the comparison morphism to larger coefficient sheaves, using some simple Lie-theoretic computations.

  \begin{lemma} \
   \label{lemma:HWtheory}
   \begin{enumerate}[(i)]
    \item Let $k_1 \ge k_2$. Then there is a unique (up to scalars) $Q$-equivariant map
    \[ W(k_1, k_2; c) \to \lambda(k_1, k_2; c). \]

    \item Let $r_1 \ge r_2 \ge 0$. Then there is a unique (up to scalars) $Q$-equivariant map
    \[ V(r_1, r_2; r_1 + r_2) \to \lambda(r_1, -r_2; r_1 + r_2).\]
    This map factors through $V / \Fil^{r_1 + 1} V$.

    \item The restriction of the homomorphism (ii) to $\Gr^{r_1} V$ is non-trivial, and factors as the composite of the projection from $\Gr^{r_1} V$ to its unique $M_{\Sieg}$-summand isomorphic to $W(r_1, -r_2; r_1 + r_2)$, composed with the map from (i) for this subrepresentation.
   \end{enumerate}
  \end{lemma}

  \begin{proof}
   Part (i) is clear, since the image of $Q$ in the Levi quotient of $P_{\Sieg}$ is the lower-triangular Borel. For part (ii), we argue analogously, using the fact that $Q$ is contained in a Weyl-group conjugate of the standard Borel subgroup of $G$, and $\lambda(r_1, -r_2; c)$ is the lowest-weight vector for this conjugate Borel. The compatibility (iii) is clear by considering the weights appearing in each factor.
  \end{proof}

  \begin{proposition}
   \label{prop:Qmorphism}
   The morphism of $Q$-representations
   \[ \Gr^{r_1} V \otimes W(3, 1; 0) \longrightarrow W(r_1 + 3, 1-r_2; r_1 + r_2) \longrightarrow \lambda(r_1 + 3, 1-r_2; r_1 + r_2) \]
   has a unique $Q$-equivariant extension to the whole of $\tilde L_1$.
  \end{proposition}

  \begin{proof}
   This is an easy consequence of Lemma \ref{lemma:HWtheory}(iii).
  \end{proof}

  In terms of sheaves, this gives the following.
  \begin{definition}
   For $r_1 \ge r_2 \ge 0$, we define a map
   \[
   [\widetilde L_1] \to \left( \pi_\star \cO_{\fIG_G(p^\infty)} \right)^{\Gamma = r_1 + r_2 + 2} \otimes (\det \omega_A)^{1-r_2}
   \]
   using the $Q$-equivariant morphism of Proposition \ref{prop:Qmorphism}.
  \end{definition}

  \begin{remark}
   If $r_2 = 0$, so that $[V] = \Sym^{r_1} \cH^1_{\dR}(A)$, this morphism has a simple explicit description: it is given by restriction to $\Sym^{r_1} \cH^1_{\dR}(\cG_1)$ and the trivialisation $\cG_1 \cong \mu_{p^\infty}$ over $\fIG_G(p^\infty)$.
  \end{remark}

  \begin{corollary}[``Unit root splitting'']
   \label{cor:unitroot}
   For any $r_1 \ge r_2 \ge 0$, we have a diagram of cohomology groups
   \[
    \begin{tikzcd}
     H^1\left(\fX_{G, \Kl}(p)^{\ge 1}, [\tilde L_1](-D)\right)\ar[rd, dashed]\\
     H^1\left(\fX_{G, \Kl}(p)^{\ge 1}, [L_1](-D)\right) \rar \uar[hook] &
     H^1\left(\fX_{G, \Kl}(p)^{\ge 1}, \fF(r_1 + 3, 1-r_2)(-D)\right).
    \end{tikzcd}
   \]
  \end{corollary}

 \subsection{Branching to $H$}

  Recall that we have defined a homomorphism of formal schemes
  \[ \iota_\infty: \fX_{H}^{\ord}(p^\infty) \to \fX_{G, \Kl}^{\ge 1}(p^\infty). \]
  The space $\fX_{H}^{\ord}(p^\infty)$ classifies pairs of \emph{ordinary} elliptic curves $(E_1, E_2)$ with prime-to-$p$ level structure and isomorphisms
  \[ \gamma: E_1[p^\infty]^\circ \xrightarrow{\ \cong\ } E_2[p^\infty]^\circ.\]
  Hence there is a reduction of $\iota_\infty^*(\cT^{\Kl})$ to a torsor for the group $\{\diag(x, x, y, y) : x, y \in \cO^\times\}$.


  \begin{theorem}
   Let $V_H$ be a small subrepresentation of $V$ given by a pair $(t_1, t_2)$ as above. Then the composite
   \[
   H^0(\fX_H^{\ord}, \omega^{1, 1} \otimes [V_H](-D)) \to H^1(\fX_G^{\ge 1}, \tilde L_1(-D)) \to H^1\left(\fX_G^{\ge 1}, \fF(r_1 + 3, 1-r_2)(-D)\right)
   \]
   coincides with the composite of the unit-root splitting
   \[ H^0(\fX_H^{\ord}, \omega^{1, 1} \otimes [V_H](-D)) \to H^0(\fX_H^{\ord}, \omega^{(t_1 + 1, t_2 + 1)}(-D)) \]
   and the pushforward
   \[ H^0(\fX_H^{\ord}, \omega^{(t_1 + 1, t_2 + 1)}(-D)) \to H^1(\fX_G^{\ge 1}, L_1(-D)) \to H^1\left(\fX_G^{\ge 1}, \fF(r_1 + 3, 1-r_2)(-D)\right).\]
  \end{theorem}

  \begin{proof}
   We need to compare two maps of sheaves on $\fX_H^{\ord}(p^\infty)$. The first is the composite
   \[
   V_H(t_1, t_2) \to \iota_\infty^*(V / \Fil^{r_1 + 1}) \to \iota_\infty^*\left( [\lambda(r_1, -r_2)]\right),
   \]
   where the second map is given by the splitting of Corollary \ref{cor:unitroot}. The second is the map
   \[ V_H(t_1, t_2) \to \omega^{t_1, t_2} \to \iota^*\left( [\lambda(r_1, -r_2)]\right),\]
   where the first map comes from the unit-root splitting on $H$. Via our torsor formalism we are reduced to checking the equality of two morphisms of representations of the group $\{\dfour{x}{x}{y}{y} : x, y \in \cO^\times\}$, and this is obvious.
  \end{proof}

  \begin{corollary}
   Theorem \ref{thm:interp2} holds as stated if we assume only that $\phi(\cE)$ is the image under the $p$-adic unit-root splitting of a classical nearly-holomorphic modular form.\qed
  \end{corollary}

  \begin{remark}
   If $V_H$ is any irreducible $H$-subrepresentation of $V$ (not necessarily small), then the same argument as above extends to show that the composite $H^0(\fX_H^{\ord}, \omega^{1, 1} \otimes [V_H](-D)) \to H^1(\fX_G^{\ge 1}, \tilde L_1(-D)) \to H^1\left(\fX_G^{\ge 1}, \fF(r_1 + 3, 1-r_2)(-D)\right)$ factors through projection to $\omega^{1, 1} \otimes [\Gr^{r_1} V_H]$, regarded as a subsheaf of $[V_H]$ over $\fX_H^{\ord}$ via the unit root splitting. The small subrepresentations are precisely those where $\Gr^{r_1} V_H$ is the smallest nonzero filtration step.
  \end{remark}


\section{Families of Eisenstein series}


 The theory developed in the previous sections allows us to define $p$-adic measures interpolating the cohomological periods $\langle \iota_{m, \star}(E), \eta\rangle$, as $E$ varies over the specialisations of some family of cuspidal $p$-adic modular forms $\cE$ on $H$. In this section, we shall specify the particular families $\cE$ which we shall consider, which will be built up from Eisenstein series and cusp forms for $\GL_2$. We shall prove in the remaining sections that the resulting periods $\langle \iota_{m, \star}(E), \eta\rangle$ are special values of $L$-functions.

 \subsection{Non-holomorphic Eisenstein series}
  \label{sect:nonholoEis}

  Let $\Phi_\f \in \cS(\Af^2, \CC)$ be a Schwartz function.

  \begin{definition}
   For $k \ge 1$, $\tau = x + iy$ in the upper half-plane, and $s \in \CC$ with $\Re(s) \ge 1$, we define
   \[ E^{(k, \Phi_\f)}(\tau; s) \coloneqq \frac{\Gamma(s + \tfrac{k}{2})}{(-2\pi i)^k \pi^{s- \tfrac{k}{2}}} \sum_{(m, n) \in \QQ^2 - (0, 0)} \frac{\Phi_\f(m, n) y^{(s - \tfrac{k}{2})}}{(m\tau + n)^k |m\tau + n|^{2s-k}}. \]
   This can be extended to all $s \in \CC$ by analytic continuation in $s$.
  \end{definition}

  \begin{remark}
   If $\Phi_\f$ is the indicator function of $(0, \alpha) + \hat\ZZ^2$, this series is $E^{(k)}_\alpha(\tau, s - \tfrac{k}{2})$ in the notation of \cite{LLZ14}.
  \end{remark}

  If $\chi$ is a Dirichlet character, and $\widehat{\chi}$ the corresponding adelic character as in \S\ref{sect:dirichlet}, we define
  \[  R_\chi(\Phi_\f) \coloneqq \int_{a \in \widehat{\ZZ}^\times} \widehat{\chi}(a) \left(\stbt{a}{0}{0}{a} \cdot \Phi_\f\right)\, \mathrm{d}^\times a,\]
  the projection of $\Phi_\f$ to the $\widehat{\chi}^{-1}$-isotypical subspace for $\hat{\ZZ}^\times$, and we set
  \(
  E^{(k, \Phi_\f)}(\tau; \chi, s) \coloneqq E^{(k, R_\chi(\Phi_\f))}(\tau; s).
  \)
  Note that $E^{(k, \Phi_\f)}(-; \chi, s)$, vanishes if $(-1)^k \chi(-1) \ne 1$.

  We can interpret $E^{(k, \Phi_\f)}(-; s)$ and $E^{(k, \Phi_\f)}(-; \chi, s)$ as $C^\infty$ sections of a line bundle on the $\GL_2$ Shimura variety. More precisely, the $C^\infty$ sections of the automorphic line bundle $\omega(k)$ are the smooth functions $f: \GL_2(\Af) \times \cH \to \CC$ satisfying
  \[
   f(g, \tau) = (ad - bc) (c\tau + d)^{-k} f\left(\stbt{a}{b}{c}{d} g, \tfrac{a\tau + b}{c\tau + d}\right)\quad\text{for all} \stbt{a}{b}{c}{d} \in \GL_2^+(\QQ).
  \]
  With these notations, $E^{(k, \Phi_\f)}(s) \coloneqq (g, \tau) \mapsto \|\det g\|^{s + 1 - k/2} E^{(k, g \cdot \Phi)}(\tau; s)$ is a $C^\infty$ section of $\omega(k)$, and similarly with $\chi$; and $E^{(k, \Phi_\f)}(\chi, s)$ transforms under the centre of $\GL_2(\Af)$ by $\widehat{\chi}^{-1}\|\cdot\|^{2-k}$.

  \subsection{Nearly holomorphic specialisations}

  \begin{proposition}
   For integers $j \in [0, k-1]$, the $C^\infty$ section $E^{(k, \Phi_\f)}(\tfrac{k}{2} - j)$ of $\omega(k)$ is nearly-holomorphic; and if $\Phi_\f$ takes values in a number field $E$, this section is defined over $E$.
  \end{proposition}

  \begin{notation}
   We let $\psi$ denote the unique additive character $\AA / \QQ \to \CC^\times$ satisfying $\psi(x_\infty) = \exp(-2\pi i x_\infty)$ for $x_\infty \in \RR$. For $\Phi_\f \in \cS(\Af)$, we write $\Phi_\f'(u, v)$ for the Fourier transform in the second variable only:
   \[ \Phi'_\f(u, v) \coloneqq \int_{\Af} \Phi_\f(u, w) \psi(vw) \, \mathrm{d}w.\]
  \end{notation}

  If we define the ``$q$-expansion'' of a nearly-holomorphic form to be the Fourier expansion of its holomorphic part, then for $n > 0$ we have
  \[
  a_n\left(E^{(k, \Phi_\f)}(\tfrac{k}{2}-j)\right) =  \sum_{\substack{(u, v) \in (\QQ^\times)^2 \\ uv = n}} u^j v^{(k-1-j)} \sgn(u) \Phi_\f'(u, v).
  \]
  The constant term $a_0\left(E^{(k)}_{\Phi_\f}(\tfrac{k}{2}-j)\right)$ is zero unless $j = 0$ or $j = k-1$, in which case it is given by a special value of a linear combination of Hurwitz zeta functions. For our purposes it suffices to note that if $\Phi_\f(-, 0) = \Phi_\f(0, -) = 0$, then the constant term is zero for all $j$.

 \subsection{Families of p-adic Eisenstein series}

  We shall now define a family of local Schwartz functions $\Phi_{p, \mu, \nu}$, depending on a choice of two Dirichlet characters $\mu$ and $\nu$ of $p$-power conductor.

  \begin{definition}
   We define $\Phi_{p, \mu, \nu} \in \cS(\Qp^2, \CC)$ as the unique function such that
   \[ \Phi'_{p, \mu, \nu}(x, y) =
   \begin{cases}
   \mu(x) \nu(y) &\text{if $x, y \in \Zp^\times$,}\\
   0 &\text{otherwise.}
   \end{cases}
   \]
  \end{definition}

  (The values of $\Phi_{p, \mu, \nu}$ can be made explicit in terms of Gauss sums, but we do not need this.) Note that $\Phi_{p, \mu, \nu}$ transforms under $\stbt{a}{0}{0}{d}$, $a,d \in \Zp^\times$, by $\mu(a) \nu^{-1}(d) = \widehat{\mu}^{-1}(a) \widehat{\nu}(d)$.

  Now let $\Phi^{(p)}$ be a Schwartz function on $(\Af^p)^2$ with values in a number field $E$, and $\chi^{(p)}$ a Dirichlet character of prime-to-$p$ conductor such that $\Phi^{(p)}$ satisfies $\stbt{a}{0}{0}{a} \cdot \Phi^{(p)} = \widehat{\chi}^{(p)}(a)^{-1} \Phi^{(p)}$ for $a \in (\widehat{\ZZ}^{(p)})^\times$.

  Let $L$ be the completion of $E$ at some prime above $p$. If $\mu, \nu$ are Dirichlet characters of $p$-power conductor as above, and we set $\chi = \chi^{(p)} \mu \nu^{-1}$, then the function $\Phi_{\mu, \nu} = \Phi^{(p)} \Phi_{p, \mu, \nu}$ is in the image of the idempotent $R_\chi$.

  \begin{theorem}
   Let $R = \Lambda_L(\Zp^\times \times \Zp^\times)$, with its two canonical characters $\kappa_1, \kappa_2$. For each $\Phi^{(p)}$ taking values in $L$, there exists an element $\mathcal{E}^{\Phi^{(p)}}\left(\kappa_1, \kappa_2; \chi^{(p)}\right) \in \mathcal{S}_{\kappa_1 + \kappa_2 + 1}(R)$, whose specialisation at $(a + \mu, b + \nu)$, for integers $a, b\ge 0$, is the $p$-adic modular form associated to the algebraic nearly-holomorphic form
   \[
   (g, \tau) \mapsto  \widehat{\nu}(\det g)^{-1} \cdot E^{(a+b+1, \Phi_{\mu, \nu})}\left(g, \tau; \chi^{(p)} \mu \nu^{-1}, \tfrac{b-a+1}{2}\right)\in M_{a+b+1}^{\mathrm{nh}}.
   \]
   For $g \in \GL_2(\Af^p)$, we have $g \cdot \mathcal{E}^{\Phi^{(p)}} =\|\det g\|^{1-\kappa_1} \cdot \mathcal{E}_{g \cdot \Phi^{(p)}}$.
   The $q$-expansion of $\mathcal{E}^{\Phi^{(p)}}$ at $\infty$ is
   \[
   \sum_{\substack{u, v \in (\ZZ_{(p)}^\times)^2 \\ uv > 0}} \sgn(u) u^{\kappa_1} v^{\kappa_2} (\Phi^{(p)})'(u, v) q^{uv}.
   \]
   The Eisenstein series $\mathcal{E}_{\Phi^{(p)}}(-; \chi^{(p)})$ is identically 0 on the components of $\Spec R$ where $\kappa_1(-1) \kappa_2(-1) \ne -\chi^{(p)}(-1)$.
  \end{theorem}

  This result is essentially a restatement of Katz's theory of the Eisenstein measure; we have stated it in a slightly unusual form in order to spell out precisely the $\GL_2\left(\Af^p\right)$-equivariance properties of the construction. Note that the factor $\widehat{\nu}(\det g)^{-1}$ is required in order that the right-hand side be invariant under $\stbt{1}{*}{}{*} \subset \GL_2(\Zp)$, which is a prerequisite for it to be the specialisation of a $p$-adic modular form.

 \subsection{Input to the machine: the case of $\GSp_4$}
  \label{sect:input1}

  Suppose we are given an automorphic representation $\Pi$ as before, which is globally generic and cohomological with coefficients in $V(r_1, r_2)$, and unramified and Klingen-ordinary at $p$. We set $d = r_1 - r_2 \ge 0$, and $\chi_{\Pi}$ the Dirichlet character such that $\Pi$ has central character $\widehat{\chi}_{\Pi}$.

  Let $R = \Lambda_L(\Zp^\times \times \Zp^\times)$ with its two canonical characters $\mathbf{j}_1, \mathbf{j}_2$. Given $\Phi_1^{(p)}, \Phi_2^{(p)}$, we consider the family of $p$-adic modular forms for $H$ given by
  \[
   \cE^{\Phi_1^{(p)}}(d-\mathbf{j}_1, \mathbf{j}_2; \chi_\Pi) \boxtimes \cE^{\Phi_2^{(p)}}(0, \mathbf{j}_1- \mathbf{j}_2; \mathrm{id}).
  \]
  By construction, the specialisation of this family at $(a_1 + \rho_1, a_2 + \rho_2)$, for integers $a_1, a_2$ such that $d \ge a_1 \ge a_2 \ge 0$ and Dirichlet characters $\rho_i$ of $p$-power conductor, is the product of two Eisenstein series of the form
  \[ \widehat{\nu}_i(\det g)^{-1} E^{(k_i, \Phi_i)}(\chi_i, s_i), \]
  where the parameters are given by
  \begin{align*}
  k_1 &= 1+d -a_1 + a_2,          & k_2 &= 1 + a_1 - a_2,\\
  \chi_1 &= \rho_1^{-1} \rho_2^{-1} \chi_\Pi,        & \chi_2 &= \rho_1^{-1} \rho_2, \\
  s_1 &= \tfrac{1-d+a_1 + a_2}{2}, & s_2 &= \tfrac{1 + a_1 - a_2}{2}.
  \end{align*}
  The Schwartz functions $\Phi_{\f, i}$ are given by $\Phi_{\f, i} = \Phi_i^{(p)} \times \Phi_{p, \mu_i, \nu_i}$, where $\Phi_i^{(p)}$ are chosen arbitrarily, and the characters $\mu_i, \nu_i$ are given by
  \begin{align*}
  \mu_1 &= \rho_1^{-1},   & \mu_2 &= \mathrm{id},\\
  \nu_1 &= \rho_2,        & \nu_2 &= \rho_1 \rho_2^{-1}.
  \end{align*}

  \begin{remark}
   Note that our choices are made such that the Eisenstein series on the second factor of $H$ is holomorphic, although the one on the first factor is only nearly-holomorphic in general.
  \end{remark}

  We also choose a vector $\eta \in H^2(\Pif)$, lying in the ordinary $\cU_{p, \Kl}'$-eigenspace at level $\Kl(p)$. From the theory of \S \ref{sect:p-adic-pushfwd}, we obtain a measure $\mathcal{L} \in R$, whose specialisation at $(a_1 + \rho_1, a_2 + \rho_2)$ interpolates the period integral of $\eta$ against the specialisation of the above family.

  We shall see in the following chapters that if $(-1)^{a_1}\rho_1(-1) \ne (-1)^{a_2}\rho_2(-1)$, this period integral will be (up to various elementary factors) equal to the product of $L$-values
  \[ L\left(\Pi \otimes \rho_1^{-1}, \frac{1-d}{2} + a_1\right) \cdot L\left(\Pi \otimes \rho_2^{-1}, \frac{1-d}{2} + a_2\right). \]
  On the other hand, if $(-1)^{a_1}\rho_1(-1) = (-1)^{a_2}\rho_2(-1)$, then both of the Eisenstein series will be identically 0. This shows that we can only interpolate products of twisted $L$-values of opposite parity, a condition which is familiar from the setting of Kato's $\GL_2$ Euler system.

 \subsection{Input to the machine: the case of $\GSp_4 \times \GL_2$}
  \label{sect:input2}

  Now let us suppose we are given an auxiliary automorphic representation $\sigma$ of $\GL_2$, corresponding to a holomorphic modular form of weight $\ell \ge 1$, and suppose that
  \[ d' = r_1 - r_2 + 1 - \ell \ge 0.\]
  For $\lambda$ a holomorphic modular form in the space of $\sigma$, we consider the family of $p$-adic modular forms over $R = \Lambda_L(\Zp^\times)$ given by
  \[ \cE^{\Phi^{(p)}}(d' - \mathbf{j}, \mathbf{j}; \chi_\Pi \chi_\sigma)\boxtimes \lambda .\]
  Note that at $\mathbf{j} = a + \rho$ with $0 \le a \le d'$, this specialises to
  \[ \widehat{\rho}(\det h)^{-1} E^{(\Phi_\f, d' + 1)}\left(h_1; \rho^{-2} \chi_\Pi \chi_\sigma, \tfrac{1-d'}{2} + a\right) \lambda(h_2)   \]
  where $\Phi_\f = \Phi^{(p)} \Phi_{(p, \rho^{-1}, \rho)}$. We shall see in the next chapter that the cup-product of this series with $\eta$ gives a zeta-integral computing
  \[ L\left(\Pi \otimes \sigma \otimes \rho^{-1}, \tfrac{1-d'}{2} + a\right).\]

  \begin{remark}
   One checks that for $\rho$ any Dirichlet character, the critical values of $L(\Pi \otimes \rho, s)$ are exactly the values $s = \tfrac{1-d}{2} + a$, for integers $0 \le a \le d$. Similarly, the critical values of $L(\Pi \otimes \sigma \otimes \rho, s)$ are the $s = \tfrac{1-d'}{2} + a$ for $0 \le a \le d'$. So in each case we hit the full interval of critical values of the relevant $L$-function (and not any others!).

   If we have $\ell \ge r_1 - r_2 + 2$, then the $L$-function for $\Pi \otimes \sigma$ may still have some critical values; but we do not see them by this method.
  \end{remark}


\section{Integral formulae for L-functions I: local theory}
 \label{sect:localintegrals}

 In this section and the next, we recall a general formula which will be used to relate the cohomological periods studied above to critical values of $L$-functions, based on work of Novodvorsky and others. In this section, we let $F$ be an arbitrary local field, and $\psi_F$ a non-trivial additive character $F \to \CC^\times$. We use $\psi_F$ to define a character $\psi: N(F) \to \CC$ by
 by
 \[
 \psi(n) = \psi_F(x+y), \ \ n=\left(\begin{smatrix} 1 & x & * & * \\ & 1 & y & * \\ & & 1 & -x \\ &&&1\end{smatrix}\right).
 \]

 \subsection{Siegel sections}

 If $\Phi$ is a Schwartz function on $F^2$, $\chi$ a smooth unitary character of $F^\times$, and $g \in \GL_2(F)$, we define
 \[
 f^{\Phi}(g; \chi, s) \coloneqq |\det g|^s \int_{F^\times} \Phi(( 0, a)g) \chi(a) |a|^{2s} \mathrm{d}^\times a.
 \]
 This integral converges absolutely for $\Re(s)>0$ and defines an element of
 the principal series representation $I(|\cdot|^{s-\half},\chi^{-1}|\cdot|^{\half-s})$ of $\GL_2(F)$, the normalized induction of the representation $\stbt{x}{*}{}{y} \mapsto |x/y|^{s-\half} \chi(y)^{-1}$ of the Borel. It has meromorphic continuation in $s$; if $v$ is a finite place, it is even a rational function of $q^{s}$ (where $q$ is the cardinality of the residue field, as usual).

 \begin{remark}
  Note that this is an \emph{un-normalised} Siegel section, and hence is not necessarily entire.
 \end{remark}

 We define the \emph{Whittaker transform}  $W^{\Phi}(g; \chi, s)$ of $f^{\Phi}(g; \chi, s)$ to be the function
 \[
 W^{\Phi}(g;\chi,s) = \int_ {F} f^{\Phi}\left(\eta \stbt{1}{x}{0}{1} g; \chi, s\right) \psi_{F}(x) \mathrm{d}x, \ \ \eta = \begin{smatrix} 0 & 1 \\ -1 & 0 \end{smatrix}.
 \]
 If $s$ is such that $I(|\cdot|^{s-\half}, \chi^{-1}|\cdot|^{\half-s})$ is irreducible, then every vector in this representation is $f^{\Phi}$ for some $\Phi$, and the map $f^{\Phi} \mapsto W^{\Phi}$ gives the isomorphism from $I(|\cdot|^{s-\half}, \chi^{-1}|\cdot|^{\half-s})$ to its Whittaker model (with respect to the inverse character $\bar{\psi}_{F}$).

 For fixed $g$ and $\Phi$, $W^{\Phi}(g;\chi,s)$ is entire as a function of $s$; if $v$ is nonarchimedean, $\chi$ and $\psi_{F}$ are unramified, and $\Phi$ is the characteristic function of $\cO_{F}^2$, then $W^{\Phi}(1; \chi,s)$ is identically 1. These functions satisfy a functional equation: if $\hat{\Phi}$ is the 2-variable Fourier transform defined by
 \[
 \hat{\Phi}(x, y) = \int_{F^2} \Phi(u, v) \psi_F(xv - yu)\, \mathrm{d}u \, \mathrm{d}v,
 \]
 then we compute that
 \begin{equation}
  \label{eq:whittaker-fcl-eq}
  W^{\hat\Phi}(g; \chi^{-1}, 1-s) = \chi(\det g) W^{\Phi}(g; \chi, s).
 \end{equation}

 \begin{remark}
  The values of $W^{\Phi}$ on the diagonal torus can be given in terms of the ``partial'' Fourier transform
  \begin{equation}
   \label{eq:partialfourier}
   \Phi'(x, y) \coloneqq \int_{F} \Phi(x, w) \psi_{F}(yw) \, \mathrm{d}w.
  \end{equation}
  With this notation, we have
  \[ W^{\Phi}(\stbt{x}{}{}{y}; \chi,s) = \chi(-1) |x|^s|y|^{s-1}\int_a  \Phi'\left(xa, \tfrac{1}{ya}\right) \chi(a) |a|^{2s-1}\, \mathrm{d}^\times a. \qedhere\]
%
%
 \end{remark}

 \subsection{Definition of the local integrals}
 \label{local-zeta-int}

  Let $\pi$ be an irreducible smooth representation of $G(F)$, with unitary central character $\chi_\pi$. We assume $\pi$ is \emph{generic}, i.e.~that $\pi$ is isomorphic to a space of functions $G(F) \to \CC$ satisfying $W(ng) = \psi(n) W(g)$ for $n \in N(F)$. Such a model is unique (see \cite[Thm.~3]{Rodier-padic} if $F$ is nonarchimedean, and \cite[Thm.~8.8(1)]{Wallach-real} if $F$ is archimedean). We fix a choice of isomorphism between $\pi$ and its Whittaker model, and for $\varphi \in \pi$, we write $W_\varphi$ for the corresponding Whittaker function.

 \subsubsection*{The two-parameter $\GSp_4$ integral}

  Let $\chi_1, \chi_2$ be smooth characters of $F^\times$ such that $\chi_1 \chi_2 = \chi_{\pi}$.

 \begin{definition}
  We define
  \[
  Z(\varphi, \Phi_1, \Phi_2, s_1, s_2) \coloneqq
  \int_{Z_G(F)N_H(F)\backslash H(F)}
  W_{\varphi}(h)
  f^{\Phi_1}(h_1;\chi_1,s_1)
  W^{\Phi_2}(h_2;\chi_2,s_2)
  \, \mathrm{d} h,
  \]
  where $\varphi \in \pi$ and each $\Phi_{i}$ is a Schwartz function on $F^2$.
 \end{definition}
 Substituting in the definition of $W^{\Phi_2}(-)$ as an integral, we obtain the following alternative formula:

 \begin{proposition}
  For $\varphi \in \pi$, the integral
  \[
  B_{\varphi}(g, s) \coloneqq \int_{F^\times} \int_{F} W_{\varphi}\left(\begin{smatrix} a \\ &a \\ &x&1\\ &&&1 \end{smatrix}w_2 g\right) |a|^{s-\tfrac{3}{2}} \chi_2^{-1}(a)\, \mathrm{d}x\, \mathrm{d}^\times a,\qquad w_2 = \begin{smatrix}1\\&&1 \\ &-1 \\&&&1\end{smatrix}
  \]
  converges for $\Re(s) \gg 0$ and has meromorphic continuation to all values of $s$; and we have
  \[
  Z(\varphi, \Phi_1, \Phi_2, s_1, s_2) = \int_{D N_H \backslash H}B_\varphi(h; s_1 - s_2 + \tfrac{1}{2}) f^{\Phi_1}(h_1;\chi_1,s_1) f^{\Phi_2}(h_2;\chi_2,s_2)\, \mathrm{d}h,
  \]
  where $D$ denotes the torus $\{\diag(p, q, p, q): p, q\in F^\times\}$.
 \end{proposition}

 The function $B_{\varphi}(g; s)$ is an element of the \emph{Bessel model} of the representation $\pi$: it transforms by a character under left-translation by the Bessel subgroup $D N_{\Sieg}$. On the other hand, the integral defining $B_\varphi$ in terms of $W_\varphi$ is an instance of Novodvorsky's local zeta integral for $L(\pi \otimes \chi_2^{-1}, s)$ \cite{novodvorsky79}. These two facts will be crucial in our analysis of the two-parameter integral $Z(\varphi, \Phi_1, \Phi_2, s_1, s_2)$.

 \subsubsection*{The $\GSp_4 \times \GL_2$ integral}

 Similarly, if $\pi$ is as before and $\sigma$ is a generic representation of $\GL_2(F)$, we let $\chi = \chi_{\pi} \chi_{\sigma}$ and define
 \[
 Z(\varphi, \lambda, \Phi, s) \coloneqq
 \int_{Z_G(F)N_H(F)\backslash H(F)}
 W_{\varphi}(h)
 f^{\Phi}(h_1;\chi, s)
 W_{\lambda}(h_2)
 \,  \mathrm{d} h,
 \]
 where $\varphi \in \pi$ and $\lambda \in \sigma$, and $W_{\lambda}$ denotes the image of $\lambda$ in the Whittaker model of $\sigma$, again with respect to the opposite character $\bar{\psi}_{F}$.

 \begin{remark}
  Note that if $\chi_2  |\cdot|^{2s_2 - 1} \notin \{ |\cdot|^{\pm 1}\}$ (so that $\sigma = I(|\cdot|^{s_2 - 1/2}, |\cdot|^{1/2 - s_2} \chi_2^{-1})$ is irreducible), we can regard the first integral as a special case of the second, taking $\lambda = f^{\Phi_2}(-;\chi_2, s_2)$. However, it is convenient to treat the first integral separately in order to understand the reducible cases, and the variation in $s_2$.
 \end{remark}

 Essentially the same analysis of the local zeta integrals as in \cite{Soudry-GSpGL} and \cite{Gel-PS-explicit} yields the following:

 \begin{proposition}\label{local-zeta-prop} \
  \begin{itemize}
   \item[(i)] There exists a positive real number $R > 0$, depending on $\pi$ and $\sigma$, such that the local zeta integral $Z(\varphi, \lambda, \Phi, s)$ converges absolutely for $\Re(s) > R$, for all choices of $(\varphi, \lambda, \Phi)$. Furthermore, each $Z(\varphi, \lambda, \Phi, s)$
   has a meromorphic continuation in $s$, and is a rational function of $q^{s}$ if $F$ is nonarchimedean.

   \item[(ii)] There exists a positive real number $R > 0$, depending on $\pi$ and $s_2$, such that the local zeta integral $Z(\varphi, \Phi_1, \Phi_2, s_1, s_2)$ converges absolutely for $\Re(s_1) > R$, for all choices of $(\varphi,  \Phi_1, \Phi_2)$. Furthermore, each $Z(\varphi, \Phi_1, \Phi_2, s_1, s_2)$ extends to a meromorphic function of $(s_1, s_2) \in \CC^2$, and is a rational function of $(q^{s_1}, q^{s_2})$ if $F$ is nonarchimedean.\qed
  \end{itemize}
 \end{proposition}

 \subsection{Nonarchimedean $L$-factors}

 In this section we assume $F$ is nonarchimedean.

 \begin{definition}
  Let $L(\pi, s)$ and $L(\pi \otimes \sigma, s)$ denote the local $L$-factors associated to $\pi$ and to $\pi \otimes \sigma$ via Shahidi's method, as in \cite[\S 4]{gantakeda11}.
 \end{definition}

 By construction, these $L$-factors coincide with the Artin $L$-factors of the Weil--Deligne representations $\mathrm{rec}(\pi)$ and $\mathrm{rec}(\pi) \otimes \mathrm{rec}(\sigma)$ respectively, where ``$\mathrm{rec}$'' denotes the local Langlands correspondence of \emph{op.cit.}.

 \begin{theorem}
  \label{thm:localzeta}
  (i) The vector space of functions on $\CC^2$ spanned by the $Z(\varphi, \Phi_1, \Phi_2, s_1, s_2)$, as the data $(\varphi, \Phi_1, \Phi_2)$ vary, is a fractional ideal of $\CC[q^{\pm s_1}, q^{\pm s_2}]$ containing the constant functions. This fractional ideal of $\CC[q^{\pm s_1}, q^{\pm s_2}]$ is generated by the product of $L$-factors
  \[ L\left(\pi, s_1 + s_2 - \tfrac12\middle) L\middle(\pi \otimes \chi_2^{-1}, s_1 - s_2 + \tfrac12\right).\]
  (ii) If $\pi$, $\chi_{i}$ and $\psi_{F}$ are unramified, $\varphi^0 \in \pi$ is the unique spherical vector such that $W_{\varphi^0}(1) = 1$, and $\Phi^0_1 = \Phi^0_2 = \Ch(\cO_{F}^2)$, then
  \[
  Z(\varphi^0, \Phi^0_1, \Phi^0_2, s_1, s_2) =
  L\left(\pi, s_1 + s_2 - \tfrac12\middle)
  L\middle(\pi \otimes \chi_2^{-1}, s_1 - s_2 + \tfrac12\right).
  \]
 \end{theorem}

 \begin{proof}
  It suffices to prove the analogous result result for fractional ideals in $\CC[q^{\pm s_1}]$, for a fixed value of $s_2$. It follows from the computations of Takloo-Bighash \cite{takloobighash00} that the ``lowest common denominator'' of Novodvorsky's zeta integrals agrees with Shahidi's definition of the $L$-factor. That is, if we define
  \[ \widetilde B_\varphi(g) = L(\pi \otimes \chi_2^{-1}, s_1 - s_2 + \tfrac12)^{-1} \cdot B_\varphi(g), \]
  then $\widetilde B_\varphi(g) \in \CC[q^{\pm s_1}]$ for all $g$, and there is no $s_1$ such that $\widetilde B_\varphi(g)$ vanishes for all $g$ and $\varphi$.

  It follows from the theory summarized in \cite{robertsschmidt16} that the space of functions $\{ \widetilde B_\varphi(-) : \varphi \in \pi\}$, is the (unique) \emph{split Bessel model} of $\pi$, with respect to the character of the Bessel torus given by $( |\cdot|^{s_2 - s_1}\chi_2, |\cdot|^{s_1 - s_2} \chi_1)$. The integral $\int_{D N \backslash H}\widetilde{B}_\varphi(h) f^{\Phi_1}(h_1) f^{\Phi_2}(h_2)\, \mathrm{d}h$ is precisely Piatetski-Shapiro's zeta integral \cite{piatetskishapiro97} for the local $L$-factor $L(\pi, s)$, with $s = s_1 + s_2 - \tfrac12$. So to prove (i) we are reduced to the question of whether the Piaetski-Shapiro's zeta integral always agrees with the Shahidi and Novodvorsky definitions of $L(\pi, s)$, independently of the choice of character used in the definition of the Bessel model. This is exactly the main result of \cite{roesnerweissauer17}.

  The unramified computation (ii) now follows from the corresponding computation of Novodvorsky and Piatetski-Shapiro's integrals in the unramified case.
 \end{proof}

 For the $\GSp_4 \times \GL_2$ integral we have a less complete result:

 \begin{theorem}
  (i) The vector space of functions on $\CC$ spanned by the $Z(\varphi, \lambda, \Phi, s)$, as the data $(\varphi,\lambda, \Phi)$ vary, is a fractional ideal of $\CC[q^{\pm s}]$ containing the constant functions. If \emph{either}
  \begin{itemize}
   \item $\sigma$ is principal series, or
   \item $\pi$ arises from a pair $(\tau_1, \tau_2)$ of generic irreducible representations of $\GL_2$ with the same central character (via the theta-lifting from $\operatorname{GSO}_{2, 2}$), and if $\sigma$ is supercuspidal, then neither of the $\tau_i$ is isomorphic to an unramified twist of $\sigma^\vee$,
  \end{itemize}
  then this fractional ideal is generated by the $L$-factor $L(\pi \otimes \sigma, s)$. In particular, this holds if at least one of $\pi$ and $\sigma$ is unramified.

  (ii) If $\pi$, $\sigma$ and $\psi_{F}$ are all unramified, $\varphi^0 \in \pi$ and $\lambda^0 \in \sigma$ are the spherical vectors normalised such that $W_{\varphi^0}(1) = W_{\lambda^0}(1)=1$, and $\Phi^0 = \Ch(\cO_{F}^2)$, then
  \[
  Z(\varphi^0, \lambda^0, \Phi^0, s) = L(\pi \otimes \sigma, s).
  \]
 \end{theorem}

 \begin{remark}
  The representations of $\GSp_4(F)$ which are theta-lifts from the split orthogonal group $\operatorname{GSO}_{2, 2}(F) \cong (\GL_2(F) \times \GL_2(F)) / \{ (z, z^{-1}): z \in F^\times\}$ are tabulated in \cite{gantakeda11b}. Note that this class of representations includes all irreducible principal series, and all representations irreducibly induced from supercuspidal representations of the Siegel Levi subgroup.
 \end{remark}

 \begin{proof}
  Firstly, we note that if $\sigma$ is an irreducible principal series representation, then after replacing $\pi$ and $\sigma$ with $\pi \otimes \beta$ and $\sigma \otimes \beta^{-1}$ for a suitable character $\beta$, we can arrange that $\sigma = I(|\cdot|^{s_2 - 1/2}, |\cdot|^{1/2 - s_2} \chi_2^{-1})$ for some $s_2$. Then we have $L(\pi \otimes \sigma, s) = L\left(\pi, s_1 + s_2 - \tfrac12\middle)
  L\middle(\pi \otimes \chi_2^{-1}, s_1 - s_2 + \tfrac12\right)$ and part (i) in this case follows from the previous theorem. Part (ii) follows similarly, noting that the Whittaker transform of $\Phi^0 = \Ch(\cO_F^2)$ satisfies $W^{\Phi^0}(1) = 1$, so the two natural normalisations of the spherical test data coincide.

  The assertion concering $\pi$ lifted from $\operatorname{GSO}_{2, 2}$ follows from the results of \cite{Soudry-GSpGL}. Soudry shows that in this case the fractional ideal of values of $Z(\varphi, \lambda, \Phi, s)$ is generated by the product of $\GL_2 \times \GL_2$ Rankin--Selberg $L$-factors, $L(\tau_1 \otimes \sigma, s) L(\tau_2 \otimes \sigma, s)$; since the Gan--Takeda local Langlands correspondence is (by construction) compatible with theta-lifting from $\operatorname{GSO}_{2, 2}$, this Rankin--Selberg $L$-factor coincides with the Shahidi $L$-factor.

  It remains to check that the fractional ideal contains 1 in all cases. This follows from the ``$P_3$ theory'' of \cite[\S 2.5]{robertsschmidt07}, which can be used to construct $\varphi$ such that the restriction of $W_{\varphi}$ to the Klingen parabolic is non-zero, but has arbitrarily small support modulo the centre and the unipotent radical of the Borel. (Compare the proof of Proposition 2.6.4 of \emph{op.cit.}.) We can then choose the Schwartz function $\Phi$ such that $f^{\Phi}(h_1)$ is supported (modulo the Borel of $H$) in a small neighbourhood of the identity in $\mathbf{P}^1$, cf.~the proof of Lemma 14.7.5 of \cite{jacquet72}.
 \end{proof}

 \subsubsection{Rationality}
  \label{sect:rationality}

  Let us suppose $\pi$ is definable over a number field $E$, i.e.~we have isomorphisms $\pi \otimes_{\CC, \sigma} \CC \cong \pi$ for all $\sigma \in \operatorname{Aut}(\CC / E)$. From the uniqueness of the Whittaker model it follows easily that $\cW(\pi)$ is the base-extension to $\CC$ of the $E$-vector space
  \[
   \cW(\pi)_{E} = \{ W \in \cW(\pi):  W(g)^\sigma = W\left( w(\kappa_\ell(\sigma)) g\right)\ \forall \sigma \in \operatorname{Aut}(\CC / E)\},
  \]
  where $w(x) = \diag(x^3, x^2, x, 1) \in G(F)$, and $\kappa_\ell$ is the $\ell$-adic cyclotomic character. Note that if $\pi$ is unramified, the normalised spherical vector $W_{\varphi^0}$ of Theorem \ref{thm:localzeta}(ii) is in $\cW(\pi)_{E}$.

  \begin{proposition}
   In the above setting, suppose $\chi_1, \chi_2$ take values in $E^\times$, and $\chi_2$ is unramified. If $W_{\varphi} \in \cW(\pi)_{E}$, and $\Phi_1, \Phi_2$ are $E$-valued, then $Z(\varphi, \Phi_1, \Phi_2; s_1, s_2) \in E[q^{\pm(s_1 + s_2)}, q^{\pm(s_1 - s_2)}]$.
  \end{proposition}

  \begin{proof} This is a routine check from the definition of the zeta-integral. \end{proof}

  It follows that $L(\pi, s - \tfrac{1}{2}) = P(q^{-s})$ for a polynomial $P$ with coefficients in $E$; and moreover, we can find a vector in $\cW(\pi)_E \otimes \cS(F^2, E) \otimes \cS(F^2, E)$ whose image under  $Z(-; s_1, s_2)$ is exactly $L\left(\pi, s_1 + s_2 - \tfrac12\middle) L\middle(\pi \otimes \chi_2^{-1}, s_1 - s_2 + \tfrac12\right)$. This will be used in \S \ref{sect:final} below to show that our $L$-values lie in $E$ after renormalising by a suitable period.

  (A similar statement can be formulated for the $\GSp_4 \times \GL_2$ zeta integral, using $E$-rational Whittaker models of both $\pi$ and $\sigma$; we leave the details to the reader.)


 \subsection{Local calculations at $p$}
  \label{sect:zeta-p}

 We carry out an evaluation of the local zeta integral for particular choices of (ramified) data. This calculation will be used to identify the Euler factors at $p$ in our final formula for the values of the $p$-adic $L$-function. To simplify notation we assume that $F = \Qp$, and that $\psi_F$ is unramified (i.e.~trivial on $\Zp$ but not on $p^{-1}\Zp$). We also assume that the local representation $\pi$ of $G(\Qp)$ is an unramified principal series.

 In this section (\emph{only}) we shall deal exclusively with the Klingen parabolic $P_{\Kl}$ and its Levi $M_{\Kl}$, so we shall drop the subscripts and denote them simply by $P$ and $M$. We identify $M$ with $\GL_2\times\GL_1$ by
 \[ \dthree{\lambda}{A}{\det(A)/\lambda} \mapsto (A,\lambda).
 \]
 Under this identification, the modulus character $\delta_P$ is the character of $\GL_2 \times \GL_1$ given by $(A, \lambda) \mapsto |\lambda^4 / \det(A)^2|$. (See e.g.~\cite[\S 2.2]{robertsschmidt07}). In particular, we have
 \[ \delta_P(\begin{smatrix} \det(A) \\ & A \\ &&1 \end{smatrix}) = |\det A|^2.\]

 \subsubsection{The vectors $\phi_r \in \pi_p$}

 Since $\pi$ is a principal series representation, it is in particular induced from the Klingen parabolic. Let us choose an unramified principal series $\tau$ of $\GL_2(\Qp)$, and a character $\theta: \Qp^\times \to \CC^\times$, such that $\pi$ is isomorphic to the induced representation $\Ind_P^G(\tau \boxtimes \theta)$ (normalised induction from the Klingen parabolic). With these notations, the central character of $\pi$ is given by $\chi_\pi = \theta \chi_\tau$ (where $\chi_\tau$ is the central character of $\tau$), and the spin $L$-factor of $\pi$ is given by
 \[  L(\pi, s) = L(\tau, s) L(\tau \otimes \theta, s).\]

 For $r \ge 0$, let $K_r\subset G(\Zp)$ be the depth $r$ Klingen parahoric subgroup:
 \[
 K_r \coloneqq \{ g\in G(\Zp) \ : \ g\bmod{p^r} \in P(\Zp/p^r\Zp)\}.
 \]

 If we identify the quotient $P(\Qp) \backslash G(\Qp)$ with $\mathbf{P}^3$, then $K_r$ is precisely the stabiliser of the point $(0:0:0:1) \in \mathbf{P}^3(\Zp/p^r\Zp)$, so we have the following characterisation of the coset $P(\Qp) K_r$:

 \begin{lemma}
  \label{support-lem}
  Suppose $r \ge 1$ and let $g = \begin{smatrix} * & * & * & * \\ b' & * & * & * \\ c' & * & * & * \\ a & b & c & d\end{smatrix} \in G(\Qp)$. Then $g \in P(\Qp)K_r$ if and only if $d\in\Qp^\times$ and  $d^{-1}a,d^{-1}b,d^{-1}c\in p^r\Zp$. \qed
 \end{lemma}

 %

 \begin{definition}
  For $r\geq 1$, let $\phi_r \in \pi_p^{K_r}$ be the function with support $P(\Qp)K_r$, and such that $\phi_r(1) = p^{3r}$.
 \end{definition}

 Note that the functions $\phi_r$ are trace-compatible, i.e.~we have
 \[ \frac{1}{[K_r: K_{r+1}]}\sum_{k \in K_{r}/ K_{r+1}} k \phi_{r+1} = \phi_r.\]
 Moreover, the restriction of $\phi_r$ to $M(\Qp)$ is independent of $r$ up to scaling: it is given by $(A, \lambda) \mapsto p^{3r} |\lambda^2/\det(A)| \xi(A) \theta(\lambda)$, where $\xi$ is the normalised spherical vector of the unramified $\GL_2(\Qp)$-representation $\tau$.

 A straightforward explicit computation gives the following:

 \begin{lemma}\label{phi1-phir-lem}
  Let $r \ge 1$, and let $t_P = \diag(1, p, p, p^2)$. Then
  \begin{enumerate}[(i)]
   \item The vector $\phi_r$ is an eigenvector for the operator $[K_r t K_r]$ on $\pi_p^{K_r}$, with eigenvalue $p^2 \chi_\tau(p)$.
   \item We have \[ \phi_r = \left(\tfrac{p}{\chi_\tau(p)}\right)^{r-1} \sum_{a\in p^r\Zp/p^{2r-1}\Zp} \begin{smatrix} 1 & & & \\ & 1 & & \\ & & 1 & \\ a & & & 1\end{smatrix} t_P^{r-1} \cdot \phi_1.\qed\]
  \end{enumerate}
 \end{lemma}

 Conversely, given any eigenvalue of $[K_1 t K_1]$ on $\pi_p^{K_1}$, we can always write $\pi_p$ in the form $\Ind_P^G(\tau \boxtimes \theta)$ such that $p^2 \chi_\tau(p)$ is the given eigenvalue.

 \subsubsection{Statement of the formula}

 Let $\sigma$ be any generic representation\footnote{We allow the case where $\sigma$ is a reducible principal series, in which case we simply \emph{define} $\cW(\sigma)$ to be the image of the Whittaker transform.} of $\GL_2(\Qp)$, and $\cW(\sigma)$ its Whittaker model with respect to $\psi^{-1}$. Let $\chi = \chi_{\pi} \chi_{\sigma}$. The goal of this section is the following computation:

 \begin{proposition}
  \label{prop:local-zeta-evaluation}
  Let $\Phi_1 \in \cS(\Qp^2)$ and $\lambda \in \sigma$. Then there is some $R$ (depending on $\Phi_1$ and $\lambda$) such that the zeta integral
  \begin{equation}
  \label{eq:localintegral}
  Z_p(\gamma \cdot \phi_r, \lambda, \Phi_1, s) =
  \int_{(Z_G N_H \backslash H)(\Qp)}
  W_{\phi_r}(h \gamma) f^{\Phi_1}(h_1; \chi, s) W_\lambda(h_2)\, \mathrm{d}h,
  \quad \gamma = \begin{smatrix} 1\\ 1 & 1\\  &  & 1  \\  &  & -1 & 1\end{smatrix}
  \end{equation}
  is independent of $r \ge R$. If $\Phi_1'(0, 0) = 0$, where $\Phi_1'$ denotes the partial Fourier transform as in \eqref{eq:partialfourier}, then this limiting value is given by
  \[
  \frac{p^3}{(p+1)^2(p-1)} \cdot \frac{L(\tau \times \sigma, s)}{L(\tau^\vee \times \sigma^\vee, 1-s) \epsilon(\tau\times\sigma, s)}
  \int_{\Qp^\times} W_\xi(\stbt{x}{0}{0}{1}) W^{\Phi_1}\left(\stbt{x}{0}{0}{1}; \chi, s\right)W_\lambda(\stbt {x}{0}{0}{1}) \frac{\theta(x)}{|x|} \,\mathrm{d}^\times x.
  \]
 \end{proposition}

 \begin{remark} \
  \begin{enumerate}[(i)]
   \item The assumption on $\Phi_1$ is probably not needed, but it is true in the cases of interest, and it simplifies the proof.

   \item In the present paper, we shall apply this formula with the data $\Phi_1$ and $\lambda$ chosen such that the integrand is a multiple of the characteristic function of $\Zp^\times$, so the integral is just the product of the three Whittaker functions at the identity. However, we shall need to consider more general choices of the local data in a sequel to the present paper (in preparation).

   \item Since $\tau$ is assumed to be unramified, we have
   \[ \epsilon(\tau \otimes \sigma, s) = \epsilon(\sigma)^2 \cdot (p^{-2s} \chi_\tau(p))^{\mathfrak{f}(\sigma)} \]
   where $\mathfrak{f}(\sigma)$ is the conductor of $\sigma$, and $\epsilon(\sigma) = \epsilon(\sigma, 0)$. In our applications, $\sigma$ will be a principal-series representation, so $\epsilon(\sigma)$ will be a product of Gauss sums, as in \S\ref{sect:dirichlet}. However, $\chi_\tau(p)$ is a more sophisticated invariant: we shall choose $\tau$ such that $\chi_\tau(p) = p^{-(r_1 + 2)}\lambda$, where $\lambda$ is the $p$-adic unit eigenvalue of the operator $\cU_{p, \Kl}$ acting on $\Pi_p$.
   \qedhere
  \end{enumerate}
 \end{remark}

 We shall first give the argument assuming the slightly stronger condition that $\Phi_1(0, x) = 0$ for all $x$.

 \subsubsection{Reduction to a $\GL_2 \times \GL_2$ integral}

 In this section, we shall use explicit formulae for the Whittaker transform to express the integral \eqref{eq:localintegral} as a Rankin--Selberg integral for $\GL_2 \times \GL_2$. For brevity we shall write $f_1(-) = f^{\Phi_1}(-; \chi, s)$. Recall that $\Phi_1$ is such that $f_1(1) = 0$, so $f_1$ is supported in the ``big cell'' $B_{\GL_2} \eta N_{\GL_2}$, where $\eta = \begin{smatrix} 0 & 1 \\ -1 & 0 \end{smatrix}$ as before. Moreover, the function $x \mapsto f_1\left( \eta \stbt{1}{x}{0}{1}\right)$ is compactly supported.

 \begin{lemma}
  The integral \eqref{eq:localintegral} equals
  \[ \frac{p}{p+1}\int_{a \in \Qp} f_1(\eta \stbt{1}{a}{0}{1}) \int_{h \in N_{\GL_2} \backslash \GL_2}W_{\phi_r}\left( \begin{smatrix} \det(h) & &  \\ & h & \\ & & & 1\end{smatrix} u_a \gamma\right) |\det(h)|^{s-1} W_\lambda(h)\, \mathrm{d}h\, \mathrm{d}a, \]
  where $u_a = \iota( \eta \stbt{1}{a}{0}{1}, 1)$.
 \end{lemma}

 \begin{proof}
  Note that if $(h_1, h_2) \in H$ with $h_1$ in the big Bruhat cell, then there exists $h \in \GL_2$, uniquely determined modulo $N_{\GL_2}$, and $a \in \Qp$ such that
  \[ N_H Z_G \cdot (h_1, h_2) = N_H Z_G \cdot  \left( \stbt{\det h}{}{}{1}, h\right) \cdot \left( \eta \stbt{1}{a}{}{1}, 1\right).\]
  Making this change of variables in the integral, and using the fact that $f_1$ transforms on the left under $\stbt{\det h}{}{}{1}$ by $(\det h)^{s}$, the lemma follows. (The $s-1$ in the exponent comes from the modulus character, because we are writing $N_H Z_G \backslash H$ as the product of two factors that do not commute; and the $p/(p+1)$ is the volume of the big cell for the unramified Haar measure.)
 \end{proof}

 The map from the model of $\pi_p = \Ind_B^G (\alpha)$ as an induced representation to the Whittaker model $\cW(\pi_p)$ is given explicitly as follows: for $\phi \in \pi_p$ we have
 \[
 W_\phi(g) = \int_{N_B(\Qp)} \phi_r(J ng) \psi(-m-x)\, \mathrm{d}n, \ \ n = \begin{smatrix} 1 & & & \\ & 1 & m & \\ & & 1 & \\ & & & 1\end{smatrix}
 \begin{smatrix} 1 & x & y & z \\ & 1 & & y \\ & & 1 & -x \\ & & & 1\end{smatrix} = n(m)n(x,y,z).
 \]
 Here $J$ is as in \S\ref{sect:groups}. We have $n(x,y,z) \dthree{\det h}{h}{1} = \dthree{\det h}{h}{1} n(u, v, w)$, where $(u, v) = (x, y) \cdot h (\det h)^{-1}$, $w = z/\det h$; thus $x = \langle (u, v), (0, 1)h \rangle$, where $\langle,\rangle : \Qp^2\times\Qp^2\rightarrow\Qp$ is the pairing $\langle (a,b), (c,d) \rangle = ad-bc$. Hence we have
 \begin{equation*}
 \begin{split}
 W_{\phi_r} & (\begin{smatrix} \det(h) & &  \\ & h & \\ & & & 1\end{smatrix} u_a \gamma) \\
 & = |\det(h)|^2 \int_{N_{B}}\phi_r\left(J n(m)
 \begin{smatrix} \det(h) & &  \\ & h & \\ & & & 1\end{smatrix} n(u,v,w) u_a\gamma\right)
 \psi(-m-\langle (u, v), (0, 1)h \rangle)\, \mathrm{d}m\, \mathrm{d}u\, \mathrm{d}v\, \mathrm{d}w.
 \end{split}
 \end{equation*}

 We note the following matrix identity. Let $u(m) = \stbt{1}{m}{0}{1}$. Then
 \begin{multline*}
 J n(m) \begin{smatrix} \det(h) & &  \\ & h & \\ & & & 1\end{smatrix} n(u,v,w) u_a\gamma = \\
 =\iota(\eta, \eta) \cdot \iota(1, u(m)) \cdot \iota(\stbt{\det h}{}{}{1}, h) \cdot n(u, v, w) \cdot \iota(\eta u(a), 1) \cdot  \gamma \\
 = \begin{smatrix} -1 & &  \\ & \eta u(m)h & \\ & & & -\det(h)\end{smatrix} \cdot \iota(\eta^{-1}, 1)\cdot n(u, v, w) \cdot \iota(\eta u(a), 1) \gamma.
 \end{multline*}
 One verifies that
 \[
 \iota(\eta^{-1}, 1)\cdot n(u, v, w) \cdot \iota(\eta u(a), 1) \cdot \gamma
 =
 \begin{smatrix}
 1& & -a& a \\  &1 &a &-a\\&&1 \\&&&1
 \end{smatrix}
 \cdot
 \begin{smatrix}
 wa + 1 & -ua & -(u+w)a^2 + (1-v)a & (u + w)a^2\\
 -wa + (-v + 1) &ua + 1& (u + w)a^2 + (2v - 2)a& (-u - w)a^2 + (-v + 1)a\\
 u &0& -ua + 1& ua\\
 u - w &u &wa + (v - 1) &-wa + 1\\
 \end{smatrix}
 \]

 If $k_a(u, v, w)$ denotes the last matrix above, then we deduce that
 \[
 J n(m)\begin{smatrix} \det(h) & &  \\ & h & \\ & & & 1\end{smatrix} n(u,v,w)
 u_a\gamma \in P(\Qp) k_a(u, v, w).
 \]
 It follows from the definition of $\phi_r$ that $\phi_r(J n(m)\begin{smatrix} \det(h) & &  \\ & h & \\ & & & 1\end{smatrix} n(u,v,w) u_a\gamma) = 0$ unless $k_a(u,v,w) \in PK_r$.

 \begin{lemma}
  Let $C = \max(-v_p(a), 0)$. If $r \ge \max(2C, C + 1)$, then $k_a(u,v,w) \in P(\Qp) K_r$ implies $k_a(u, v, w) \in K_r$.
 \end{lemma}

 \begin{proof}
  Suppose that $k_a(u, v, w) \in P K_r$. By Lemma \ref{support-lem}, we must have
  \[ 1-aw\in \Qp^\times\quad\text{and}\quad
  \tfrac{u}{1-aw}, \tfrac{w}{1-aw}, \tfrac{v-1+aw}{1-aw}\in p^r\Zp.\]
  If $\tfrac{w}{1-aw} = p^r b$ for $b \in \Zp$, then $w(1 + p^r a b) = p^rb$. Since $p^r a \in p \Zp$ by assumption, we can conclude that $(1 + p^r a b)$ is a unit and hence $w \in p^r \Zp$. Thus $1 - aw$ is also a unit, and we deduce that $u  \in p^r \Zp$. Since $v = 1 -aw \bmod p^r$ we have $v = 1 \bmod p^{r-C}$. In particular, $ua^2$,$wa^2$ and $(v-1)a$ all have valuation at least $r - 2C \ge 0$; hence $k_a(u, v, w)$ is integral (and even congruent to the identity modulo $p^{r-2C}$).
 \end{proof}

 We suppose that $r \ge \max(2C, C + 1)$, and let $R(a,r)$ denote the set of $(u, v, w)$ such that $k_a(u,v,w) \in K_r$. Then  we obtain the formula
 \begin{multline*}
 W_{\phi_r}  (\begin{smatrix} \det(h) & &  \\ & h & \\ & & & 1\end{smatrix} u_a \gamma) \\
 = |\det(h)|^2 \int_{R(a,r)}\int_{m} \phi_r\left( \begin{smatrix} -1 & &  \\ & \eta u(m)h & \\ & & & -\det(h)\end{smatrix}\begin{smatrix}
 1& & -a& a \\  &1 &a &-a\\&&1 \\&&&1
 \end{smatrix}\right) \psi(-m-\langle (u, v), (0, 1)h \rangle)\, \mathrm{d}m\, \mathrm{d}u\, \mathrm{d}v\, \mathrm{d}w.
 \end{multline*}

 Since the argument of $\phi_r$ lies in $P(\Qp)$, we have
 \[ \phi_r\left( \begin{smatrix} -1 & &  \\ & \eta u(m)h & \\ & & & -\det(h)\end{smatrix}\begin{smatrix}
 1& & -a& a \\  &1 &a &-a\\&&1 \\&&&1
 \end{smatrix}\right) = p^{3r}|\det h|^{-1} \xi\Big(\eta u(m) h u(a)\Big)  \]
 where $\xi$ is the spherical vector of $\tau$, as before. (Note that $\theta(-1) = 1$, since $\theta$ is assumed to be unramified.) Thus the inner integral over $m$ reduces to the Whittaker transform of $\xi$, evaluated at $h u(a)$; so we can write this as
 \[  W_{\phi_r}  (\begin{smatrix} \det(h) & &  \\ & h & \\ & & & 1\end{smatrix} u_a \gamma) =
 p^{3r} |\det(h)| \int_{(u,v,w) \in R(a,r)}W_\xi(hu(a)) \psi(-\langle (u, v), (0, 1)h \rangle)\, \mathrm{d}u\, \mathrm{d}v\, \mathrm{d}w. \]
 If $\Phi_{a,r}$ denotes the Schwartz function $\Phi_{a,r}(u, v) = \operatorname{vol}\{ w: (u, v, w) \in R(a,r)\}$, then we can collapse this down to
 \[  W_{\phi_r}  (\begin{smatrix} \det(h) & &  \\ & h & \\ & & & 1\end{smatrix} u_a \gamma) =
 p^{3r} |\det(h)|\cdot W_\xi(hu(a))\cdot \hat{\Phi}_{a, r}((0,1)h), \]
 where the Fourier transform for Schwartz functions is defined as above.

 \begin{remark}
  The function $\Phi_{a,r}(u, v)$ is rather explicit, and depends only on the valuation of $a$. If $C$ is as above, then the set $R(a, r)$ is given by $\{u = 1 \bmod p^r, v = 1 \bmod p^{r-C}, w = a^{-1}(1-v) \bmod p^{r+C}\}$, so $\Phi_{a,r} = p^{-(r+C)} \Ch( (p^r\Zp) \times (1 + p^{r-C} \Zp))$.
 \end{remark}

 Putting things together we have:

 \begin{proposition}
  Suppose $\Phi_1(1, -)$ is supported in $p^{-N}\Zp$, and let $r \ge \max(2N, N+1)$. Then we have
  \[ \eqref{eq:localintegral} = \tfrac{p^{3r+1}}{(p+1)}\int_{a \in \Qp} f_1(\eta \stbt{1}{a}{0}{1})\int_{h \in N_{\GL_2} \backslash \GL_2}  W_{u(a)\xi} (h) \cdot w_2(h)\cdot \hat{\Phi}_{a, r}((0,1)h) |\det(h)|^s \, \mathrm{d}h\, \mathrm{d}a,\]
  a finite linear combination of Rankin--Selberg integrals.\qed
 \end{proposition}

 \subsubsection{Application of the functional equation}

 By the functional equation for Rankin--Selberg integrals, due to Jacquet \cite[Theorem 14.7 (3)]{jacquet72}, for any $w_1 \in \cW(\xi)$ and $w_2 \in \cW(\sigma)$, and any $\Phi \in \cS(\Qp^2)$, we have
 \begin{multline*}
 \int_{N_2 \backslash \GL_2} w_1(h) w_2(h) \hat{\Phi}((0,1)h) |\det(h)|^{s} \, \mathrm{d}h =
 \\ \frac{1}{\gamma(\tau \times \sigma, s)}\int_{N_2 \backslash \GL_2} \frac{w_1(h) w_2(h)}{ (\chi_\tau\chi_\sigma)(\det h)} \Phi((0,1)h) |\det(h)|^{1-s} \, \mathrm{d}h,
 \end{multline*}
 where
 \[
 \gamma(\tau \times \sigma, s) = \frac{L(\tau^\vee \times \sigma^\vee, 1-s)}{L(\tau \times \sigma, s)} \epsilon(\tau\times\sigma, s)
 \]
 is Jacquet's local $\gamma$-factor. (More precisely, the integrals on the left and right sides are convergent for $\Re(s) \gg 0$, resp $\Re(s) \ll 0$, and both have meromorphic continuation to all $s$ and are equal as meromorphic functions.) Noting that $\chi_\tau \chi_\sigma = \chi / \theta$, and taking $\Phi = \Phi_{a, r}$, we see that \eqref{eq:localintegral} is given by
 \[
  \frac{p^{3r+1}}{(p+1)\gamma(\tau \times \sigma, s)} \int_{a \in \Qp} \!\!\!f_1(\eta \stbt{1}{a}{0}{1})\int_{h \in N_2 \backslash \GL_2} \!\!\! W_{u(a)\xi} (h) w_2(h) \Phi_{a, r}((0,1)h) (\theta/\chi)(\det h)|\det(h)|^{1-s} \, \mathrm{d}h\, \mathrm{d}a.
 \]

 The Schwartz function $\Phi_{a, r}$ has total integral $p^{-3r}$, and as $r \to \infty$, its support becomes concentrated in smaller and smaller neighbourhoods of $(0, 1)$. Taking into account a factor $p^2/(p^2-1)$ arising from comparison of unramified Haar measures, the limiting value is given by
 \[  \eqref{eq:localintegral} = \tfrac{p^3}{(p+1)^2(p-1)\gamma(\tau \times \sigma, s)}\int_{a \in \Qp} f_1(\eta \stbt{1}{a}{0}{1})\int_{x \in \Qp^\times}W_{u(a)\xi} (\stbt{x}{0}{0}{1}) w_2(\stbt{x}{0}{0}{1}) \frac{\theta(x)}{|x|^s \chi(x)} \, \mathrm{d}^\times x\, \mathrm{d}a.\]
 We have $W_{u(a)\xi} (\stbt{x}{0}{0}{1}) = \psi(xa) W_{\xi} (\stbt{x}{0}{0}{1})$, and we compute
 \[ \int_a  f_1(\eta \stbt{1}{a}{0}{1}) \psi(xa)\, \mathrm{d}a = \chi(x) |x|^{s-1} W^{\Phi_1}(\stbt{x}{0}{0}{1}).\]

 This concludes the proof of Proposition \ref{prop:local-zeta-evaluation} assuming that $\Phi_1(0, -)$ is identically 0.

 \subsubsection{Conclusion of the proof}

  We have proved the formula of Proposition \ref{prop:local-zeta-evaluation} for functions $\Phi_1$ satisfying $\Phi_1(0, x) = 0$ for all $x$, which is a more restrictive condition than stated in the proposition. We now reduce the general case to this.

  Consider the auxiliary zeta integral
  \[
   \widetilde{Z}_p(\varphi, \lambda, \Phi, s) \coloneqq
   \int_{(Z_GN_H\backslash H)(\Qp)}
   W_{\varphi}(h)
   f^{\Phi_1}(h_1;\chi^{-1}, s)
   W_{\lambda_v}(h_2)
   \chi(\det h)^{-1}
   \,  \mathrm{d} h.
  \]
  As before, this has meromorphic continuation to all $s$; and the functions $Z(\dots)$ and $\widetilde{Z}(\dots)$ are related by the following local functional equation \cite{Soudry-GSpGL}:
  \[ Z_p(\varphi, \lambda, \Phi, s) = \gamma(\pi \times \sigma, s)^{-1} \widetilde{Z}_p(\varphi, \lambda, \hat\Phi, 1-s),
  \]
  where
  \[ \gamma(\pi \times \sigma, s) = \gamma(\tau \times \sigma,s)\gamma(\tau\times\sigma \times \theta, s).\]
  Since $\widetilde{Z}_p$ can be obtained from $Z_p$ by replacing the representation $\sigma$ with $\sigma \otimes \chi^{-1}$, we deduce that if $\hat\Phi(0, x) = 0$ for all $x$, then
  \begin{align*}
   \widetilde{Z}_p(\varphi, \lambda, \hat\Phi, 1-s) &=
   \frac{V}{\gamma(\tau \times \sigma \times \chi^{-1}, 1-s)}
   \int_{\Qp^\times} W_\xi(\stbt{x}{0}{0}{1}) W^{\hat\Phi_1}\left(\stbt{x}{0}{0}{1}; \chi^{-1}, 1-s\right)W_\lambda(\stbt{x}{0}{0}{1}) \tfrac{\theta(x)}{\chi(x) |x|} \,\mathrm{d}^\times x\\
   &= \frac{V}{\gamma(\tau \times \sigma \times \chi^{-1}, 1-s)}
   \int_{\Qp^\times} W_\xi(\stbt{x}{0}{0}{1}) W^{\Phi_1}\left(\stbt{x}{0}{0}{1}; \chi, s\right)W_\lambda(\stbt{x}{0}{0}{1}) \tfrac{\theta(x)}{|x|} \,\mathrm{d}^\times x
  \end{align*}
  where $V = p^3/(p+1)^2(p-1)$ and the second equality comes from the functional equation \eqref{eq:whittaker-fcl-eq} relating $W^{\Phi_1}$ and $W^{\hat\Phi_1}$. Since $\chi = \chi_\tau \chi_\sigma \theta$, we have $\gamma(\tau\times\sigma\times \theta, s) \gamma(\tau\times\sigma\times \chi^{-1}, 1-s) = 1$. So we have shown that the formula of Proposition \ref{prop:local-zeta-evaluation} holds for all $\Phi_1$ such that $\hat\Phi_1(0, x) = 0$ for all $x$. Since
  \[ (\hat\Phi_1)'(x, y) = \Phi_1'(y, x), \]
  this is equivalent to requiring that $\Phi_1'(x, 0) = 0$ for all $x$. Since any Schwartz function $h$ on $\Qp^2$ such that $h(0, 0) = 0$ is a linear combination of functions vanishing identically along $\{0\} \times \Qp$ and $\Qp \times \{0\}$, this shows that our formula holds for all $\Phi_1$ such that $\Phi_1'(0, 0) = 0$, completing the proof of Proposition \ref{prop:local-zeta-evaluation}.\qed


\subsection{Local integrals at $\infty$}

We now consider the local integral at a real infinite place $v = \infty$. For simplicity, we shall now assume that $\psi_F(x) = \exp(-2\pi i x)$ for $x \in \RR$.
%
%
%

\subsubsection{Siegel sections}

We shall need to compute the functions $f^{\Phi}$ where $\Phi = \Phi_\infty^{(k)} \in \cS(\RR^2, \CC)$ is the function given by
\[ \Phi_{\infty}^{(k)}(x, y) \coloneqq 2^{1-k} (x + iy)^k e^{-\pi(x^2 + y^2)} \]
for some $k \in \ZZ_{\ge 1}$. (The factor $2^{1-k}$ is convenient for comparisions with algebraic Eisenstein series, as we shall shortly see.) We find readily that
\[ f^{\Phi}\left( \stbt a {}{} {a^{-1}}; \chi, s\right) =
\begin{cases}
0 & \text{if $\chi(-1) \ne (-1)^k$},\\
2^{1-k} i^k \Gamma(s + \tfrac{k}{2}) \pi^{-(s + k/2)} a^{2s} & \text{otherwise}.
\end{cases}
\]

\subsubsection{Moriyama's formula}

A formula for the archimedean Whittaker function of a vector in the lowest $K$-type subspace of a generic discrete-series representation $\Pi_\infty$ is given in \cite[\S 3.2]{moriyama04}. Let $(\lambda_1, \lambda_2)$ be integers such that $1-\lambda_1 < \lambda_2 < 0$. Then there is a unique generic discrete-series representation of $\GSp_4(\RR)$ whose central character has finite order and which is isomorphic as an $\operatorname{Sp}_4(\RR)$-representation to the direct sum $D_{(\lambda_1, \lambda_2)} \oplus D_{(-\lambda_2, -\lambda_1)}$, in Moriyama's notation. For $0 \le k \le d$, we let $v_k$ be the $k$-th standard basis vector of the lowest $K$-type of $D_{(-\lambda_2, -\lambda_1)}$.

What we shall need here is Proposition 8 of \emph{op.cit.}, which gives a formula for the image of $v_k$ under the canonical transformation from the Whittaker model to the Bessel model of $\Pi_\infty^W$. Recall that we have defined
\[
B_{v_k}(g, s) \coloneqq \int_{\RR^\times} \int_{\RR} W_{v_k}\left(\begin{smatrix} a \\ &a \\ &x&1\\ &&&1 \end{smatrix}w_2 g\right) |a|^{s-\tfrac{3}{2}}\chi_2(a)^{-1}\, \mathrm{d}x\, \mathrm{d}^\times a,\qquad w_2 = \begin{smatrix}1\\&&1 \\ &-1 \\&&&1\end{smatrix}.
\]

\begin{theorem}[Moriyama]
 For $\Re(t)$ sufficiently large, we have
 \[ \int_0^\infty B_{v_k}\left(\dfour{y}{y}{1}{1}, s\right) y^{t-\tfrac{3}{2}} \, \mathrm{d}^\times y
 \\=
 C \cdot
 \frac{(-1)^k L(\Pi_\infty, s) L(\Pi_\infty, t)}
 {\pi^t \Gamma(\tfrac{t  - s+ \lambda_2 + k + 1}{2})\Gamma(\tfrac{t + s + \lambda_1  - k}{2})},
 \]
 where $C$ is a constant depending on $\lambda_1$ and $\lambda_2$ (but not on $k, s, t$), and $L(\Pi_\infty, s) = \Gamma_\CC(s + \tfrac{\lambda_1 +\lambda_2 - 1}{2})\Gamma_\CC(s + \tfrac{\lambda_1 -\lambda_2 - 1}{2})$, $\Gamma_\CC(s) = 2(2\pi)^{-s} \Gamma(s)$.
\end{theorem}

\begin{proof}
 Moriyama states his formula slightly differently, in the form of an inverse Mellin transform. The result above follows by applying the ``forward'' Mellin transform to both sides of Moriyama's formula, in the same way as the proof of Proposition 5.12 of \cite{lemma17} (which is essentially the same argument as ours but with a different normalisation of the Bessel function). Moriyama also has a factor $\binom{d}{k}$, which does not appear in our formulae, owing to a different convention for the standard basis of the lowest $K$-type.
\end{proof}

We shall apply this with $(\lambda_1, \lambda_2) = (r_1 + 3, -r_2-1)$, so that $d = r_1 + r_2 + 4$. The vector $v_k$ is then an eigenvector for the action of the diagonal torus in $K_{G, \infty}^\circ$, which is just $K_{H, \infty}^\circ$, with weight $(r_2 + 1 - k, -r_1-3 + k)$. We shall let $(k_1, k_2)$ be integers $\ge 1$ summing to $r_1 - r_2 + 2$, and take $k = r_1 + 3 - k_2 = r_2  + 1 + k_1$; thus $v_k$ has $K_H^\circ$-type $(-k_1, -k_2)$, meaning it can pair non-trivially against a pair of holomorphic modular forms of weight $(k_1, k_2)$. Finally, we shall take $s = s_1 - s_2 + \tfrac{1}{2}$ and $t = s_1 + s_2 - \tfrac{1}{2}$. Note that the right-hand side of Moriyama's formula is now
\[  (-1)^{k_1 + r_2 + 1} C \cdot
\frac{L(\Pi_\infty, s_1 - s_2 + \tfrac{1}{2}) L(\Pi_\infty, s_1 + s_2 - \tfrac{1}{2})}
{\pi^{s_1 + s_2 + \tfrac12} \Gamma(s_1 + \tfrac{k_1}{2})\Gamma(s_2 + \tfrac{k_2}{2})}.
\]

On the other hand, if one chooses $\Phi_{i, \infty} = \Phi_{\infty}^{(k_i)}$, then the integrand in the local Bessel-model integral is right $K_{H, \infty}$-invariant; so by the Iwasawa decomposition, $H = N_H T_H K_H$ where $T_H$ is the diagonal torus, we conclude that the left-hand side of Moriyama's formula is exactly the integral $Z_\infty\left(v_k,  \Phi_{\infty}^{(k_1)},  \Phi_{\infty}^{(k_2)}, s_1, s_2\right)$ up to a sign and the factor $f^{\Phi_1}(1) f^{\Phi_2}(1)$, which corresponds precisely to the Gamma-factors in the denominator above.

If we instead consider the $\GSp_4 \times \GL_2$ integral, then the Whittaker function associated to the holomorphic vector $\lambda \in \sigma_\infty$ coincides with the Whittaker function of $\Phi_{\infty}^{(\ell)}$ at $s_2 = \tfrac{\ell}{2}$. So we see that in this case the local integral at $\infty$ is again equal to the expected value, namely $L(\Pi_\infty \times \sigma_\infty, s)$.

\section{Integral formulae for L-functions II: global theory}
\label{sect:globalintegrals}

 We now let $F$ be an arbitrary number field, with ad\`ele ring $\AA$. Let $\psi_F = \prod_v \psi_{F_v}$ be a non-trivial additive character of $\AA/F$, which we regard as a character of $N(\AA)$ as before.

 \subsection{Globally generic $\pi$}

 Let $\pi$ be a cuspidal automorphic represenation of $G(\AA)$. For each $\varphi\in\pi$ we define its \emph{Whittaker transform} with respect to $\psi$:
\[
W_\varphi(g) = \int_{N_G(F)\backslash N_G(\AA)} \varphi(ng) \psi(n^{-1}) \mathrm{d}n.
\]
The representation $\pi$ is {\em globally generic} if some $W_\varphi(g)$ is non-zero.
If $\pi$ is globally generic, it can thus be modeled as a space of functions $W: G(\AA) \to \CC$ satisfying $W(ng) = \psi(n) W(g)$
for all $n \in N_G(\AA)$.

We fix an identification $\pi = \sideset{}{'}{\bigotimes_v} \pi_v$ of $\pi$ with a restricted tensor product of irreducible local representations. If $\pi$ is globally generic then each $\pi_v$ is clearly also generic, that is, $\pi_v$ can also be modelled in a space of functions $W: G(F_v)\to \CC$ satisfying $W(ng) = \psi(n) W(g)$ for all $n\in N_G(F_v)$.

\begin{remark}
 Note that our definition of ``globally generic'' is strictly stronger than requiring that each $\pi_v$ be generic, and depends on the realisation of $\pi$ as a space of functions on $G(F) \backslash G(\AA)$ rather than on the abstract isomorphism class of $\pi$.
\end{remark}

Assuming $\pi$ is globally generic, we fix Whittaker models $\pi_v\stackrel{\sim}{\to}\mathcal{W}_v$, $\phi_v\mapsto W_{\phi_v}$ for all places $v$. These choices can and shall be made in such a way that:
\begin{itemize}
 \item If $\varphi = \otimes_v \varphi_v$, then we have $W_\varphi(g) = \prod_v W_{\phi_v}(g_v)$.
 \item If $v$ is a finite place such that $\pi_v$ and $\psi_{F_v}$ are unramified, and $\phi^0_v \in \pi_v$ is the basis of the unramified vectors that was implicitly fixed in the definition of the restricted tensor product $\pi = \sideset{}{'}{\bigotimes_v} \pi_v$, then $W_{\phi^0_v}(1) = 1$.
\end{itemize}

\subsection{Global Eisenstein series}\label{Eis-GL2}

For a Schwartz function $\Phi: \AA^2 \to \CC$, and $\chi$ a unitary Gr\"ossencharacter, we define a global Siegel section by
\[ f^{\Phi}(g; \chi, s) \coloneqq \|\det g\|^s \int_{\AA^\times} \Phi_v((0, a)g) \chi(a) \|a\|^{2s}\, \mathrm{d}^\times a, \]
We may then form the series
\[
E^\Phi(g; \chi, s) : = \sum_{\gamma \in B(F) \backslash \GL_2(F)} f^\Phi(\gamma g; \chi, s).
\]
(cf.~\cite[\S 19]{jacquet72}). The sum converges absolutely and uniformly on any compact subset of $\{ s: \Re(s) > 1\}$, and defines a function on the quotient $\GL_2(F) \backslash \GL_2(\AA)$ that transforms under the center by $\chi^{-1}$. It has a meromorphic continuation in $s$.

\begin{remark}
 Note that $g \cdot E^{\Phi}(-; \chi, s) = \|\det g\|^s E^{g \cdot \Phi}(-; \chi, s)$.
\end{remark}

We define the Whittaker transform of $E^\Phi(g; \chi, s)$ with respect to $\bar\psi_F$ to be
\[
W^\Phi(g;\chi,s): = \int_{N_{\GL_2}(F)\backslash N_{\GL_2}(\AA)} E^\Phi(ng; \chi, s) \psi_F(x) \mathrm{d}n,  \ \ n= \begin{smatrix} 1 & x \\ 0 & 1 \end{smatrix}.
\]
If $\Re(s)>1$ this unfolds to
\[
W^\Phi(g;\chi,s) = \int_ {N_{\GL_2}(\AA)} f^\Phi(\eta ng; \chi, s) \psi_F(x) \mathrm{d}n, \ \ \eta = \begin{smatrix} 0 & 1 \\ -1 & 0 \end{smatrix}.
\]

If $\Phi(g)= \prod_v\Phi_{v}(g_v)$ with $\Phi_{v}$ a Schwartz function on $F_v^2$ and $\Phi_{v}$ the characteristic function $\Phi_{v}^0$ of $\mathcal{O}_{F_v}^2$ for almost all finite $v$,
then we have factorisations
\[
f^\Phi(g;\chi,s) = \prod_v f^{\Phi_v}(g_v;\chi_v,s),\qquad
W^\Phi(g;\chi,s) = \prod_v W^{\Phi_v}(g_v;\chi_v,s)
\]
where the local integrals $f^{\Phi_v}$ and $W^{\Phi_v}$ are as defined above. Furthermore, the local Whittaker transforms $W^{\Phi_v}(g;\chi_v,s)$ converge for all values of $s$ and are holomorphic as functions of $s$, and (for a given $g$ and $\Phi$) all but finitely many of these are 1.

\subsection{The global integrals}
\label{global-zeta-int}

Let $\chi_1, \chi_2$ be two unitary Gr\"ossen\-characters of $\AA^\times$ such that $\chi_1 \chi_2 = \chi_\pi$. For $\varphi \in \pi$ and $\Phi_1, \Phi_2 \in \mathcal{S}(\AA^2)$, define
\[
Z(\varphi, \Phi_1, \Phi_2, s_1, s_2) \coloneqq
\int_{Z_G(\AA)H(F) \backslash H(\AA)} \varphi(\iota(h)) E^{\Phi_1}(h_1; \chi_1, s_1) E^{\Phi_2}(h_2; \chi_2, s_2)\, \mathrm{d}h,
\]
where $H = \GL_2 \times_{\GL_1} \GL_2$ and the $E^{\Phi_i}(h_i; \chi_i, s_i)$ are Eisenstein series as in Section \ref{Eis-GL2}. This integral converges absolutely and is holomorphic at all values of $s_1$ and $s_2$ such that neither Eisenstein series has a pole. This is because the restriction of the cuspform $\varphi$ to $H(\AA)$ is rapidly decreasing.

We also consider a second integral associated to $\pi$ and an auxiliary cuspidal automorphic representation $\sigma$ of $\GL_2(\AA)$. We now let $\chi = \chi_\pi \chi_\sigma$, and for $\varphi \in \pi$ and $\lambda \in \sigma$, we set
\[
Z(\varphi, \lambda, \Phi, s) \coloneqq
\int_{Z_G(\AA)H(F) \backslash H(\AA)} \varphi(\iota(h))  E^{\Phi}(h_1; \chi, s)\lambda(h_2)\, \mathrm{d}h.
\]
As before, this is holomorphic away from the poles of the Eisenstein series $E^{\Phi}$.

We now suppose that the test data $(\varphi, \Phi_1, \Phi_2)$ for the first integral, and $(\varphi, \lambda, \Phi)$ for the second, are products of local data $\varphi = \bigotimes_v \varphi_v$ etc; and for each place $v$ of $F$, we write $Z_v(\dots)$ for the local integrals of the previous section at the place $v$. As a straightforward consequence of Proposition \ref{local-zeta-prop} and Langlands' general results on the convergence of automorphic $L$-functions, we have the following Euler product factorisation:

\begin{proposition}
 \label{global-zeta-prop}
 (i) Let $s_2$ be fixed. There exists a positive real number $R'$, possibly depending on $\pi$ and $s_2$, such that for $\Re(s_1) > R'$, the product $\prod_v  Z_v(\varphi_v, \Phi_{1, v}, \Phi_{2, v}, s_1, s_2)$ converges absolutely. For such $s_1$ we have
 \[
 \begin{split}
 Z(\varphi, \Phi_1, \Phi_2, s_1, s_2)  & = \prod_v  Z_v(\varphi_v, \Phi_{1, v}, \Phi_{2, v}, s_1, s_2) \\
 & = L^S(\pi_v, s_1 + s_2 - \tfrac12) L^S(\pi_v \otimes \chi_2^{-1}, s_1 - s_2 + \tfrac12) \prod_{v\in S} Z_v(\varphi_v, \Phi_{1, v}, \Phi_{2, v}, s_1, s_2),
 \end{split}
 \]
 where $S$ is any finite set of places of $F$ including all the archimedean places and all finite places where $\pi$ or the $\chi_i$ are ramified, and is such that $\Phi_{i, v} = \Phi_{i, v}^0$, $\varphi_v = \varphi^0_v$ for all $v\notin S$.

 (ii) There exists a positive real number $R'$, possibly depending on $\pi$ and $\sigma$, such that for $\Re(s) > R'$ the product $\prod_v  Z_v(\varphi_v, \lambda_v, \Phi_v, s)$ converges absolutely, and for such $s$ we have
 \[
 \begin{split}
 Z(\varphi, \lambda, \Phi, s)  & = \prod_v  Z_v(\varphi_v,\lambda_v, \Phi_v, s) \\
 & = L^S(\pi_v \otimes \sigma_v, s) \prod_{v\in S} Z_v(\varphi_v, \lambda_v, \Phi_v, s),
 \end{split}
 \]
 for $S$ any sufficiently large finite set of primes as before.
\end{proposition}

\begin{proof}
 For part (ii), see \cite[Theorem 7]{novodvorsky79}, or the introduction of \cite{Soudry-GSpGL}. Part (i) follows similarly, taking an Eisenstein series in place of the cusp form $\lambda$.
\end{proof}

As $Z(\varphi, \Phi_1, \Phi_2, s_1, s_2)$, $L^S(\pi_v, s_1 + s_2 - \tfrac12)$ and $L^S(\pi_v \otimes \chi_2^{-1}, s_1 - s_2 + \tfrac12)$  have meromorphic continuations in $s_1$ (and are even meromorphic in both $s_1$ and $s_2$), for any fixed $s_2$ the equality
\[
Z(\varphi, \Phi_1, \Phi_2, s_1, s_2) = L^S(\pi_v, s_1 + s_2 - \tfrac12) L^S(\pi_v \otimes \chi_2^{-1}, s_1 - s_2 + \tfrac12) \prod_{v\in S} Z_v(\varphi_v, \Phi_{1, v}, \Phi_{2, v}, s_1, s_2)
\]
is an equality of meromorphic functions in $s_1$ (and similarly for the $\GSp_4 \times \GL_2$ integral).

\begin{remark}
 The proof of the preceding proposition also shows that if $\pi$ is not globally generic, then the integrals $Z(\varphi, \Phi_1, \Phi_2, s_1, s_2)$ and $Z(\varphi, \lambda, \Phi, s)$ are identically zero, for all choices of the test data.
\end{remark}

\section{Construction of the \texorpdfstring{$p$-adic $L$-function}{p-adic L-function}}
\label{sect:final}

 \subsection{Comparing algebraic and real-analytic Eisenstein series}

  We now assume the number field $F$ is $\QQ$. A routine unravelling of notations gives the following:

  \begin{proposition}
   Suppose $\Phi \in \cS(\AA^2, \CC)$ has the form $\Phi = \Phi_\f \cdot \Phi_\infty^{(k)}$, for some $k \ge 1$ and $\Phi_\f \in \cS(\Af^2, \CC)$, and $\widehat{\chi}$ is the adelic character associated to the Dirichlet character $\chi$ as in \S\ref{sect:dirichlet}.

   Then the Eisenstein series $E^{\Phi}(-; \widehat{\chi}, s)$ on $\GL_2(\QQ) \backslash \GL_2(\AA)$ of \S \ref{Eis-GL2} is related to the Eisenstein series $E^{k, \Phi_\f}(-; \chi, s)$ on $\GL_2(\Af) \times \cH$ of \S\ref{sect:nonholoEis} by
   \[ E^{\Phi}\left(g_\f \stbt{y}{x}{0}{1}; \widehat{\chi}, s\right) = y^{k/2} \|\det g_\f\|^{s} \cdot E^{k, \Phi_\f}(g_\f, x + iy; \chi, s)\]
   for $x + iy \in \cH$ and $g_\f \in \GL_2(\Af)$.\qed
  \end{proposition}

 \subsection{Periods and algebraicity}

  \begin{definition}
   Let $E$ be a subfield of $\overline{\QQ}$. An element $W \in \Ind_{N(\Af)}^{G(\Af)}(\psi)$ (i.e.~a $\psi$-Whittaker function on $G$) is said to be \emph{defined over $E$} if it takes values in $E(\mu_\infty)$ and satisfies
   \[ W(g)^\sigma = W\left( w(\kappa(\sigma)) g\right)\]
   for all $g \in G(\Af)$ and $\sigma \in \Gal(\overline{\QQ}/E)$, where $w(x) = \diag(x^3, x^2, x, 1)$ for $x \in \hat{\ZZ}^\times$, and $\kappa: \Gal(\overline{\QQ}/\QQ) \to \widehat{\ZZ}^\times$ is the adelic cyclotomic character.
  \end{definition}

  This is the global counterpart of the local definitions in \S\ref{sect:rationality}. We now take $E$ to be the field of definition of the arithmetic twist $\Pif'$, as in \S \ref{sect:coherentcoh} above. Given $\eta \in H^2(\Pif) \otimes \CC$, we say $\eta$ is \emph{Whittaker $E$-rational} if the Whittaker functions of the coordinate projections of the corresponding element of Harris' space $\mathscr{H}_{L_2}$ are the product of an $E$-defined Whittaker function on $G(\Af)$ and the \emph{standard Whittaker function} at $\infty$, which is the unique function that makes the constant $C$ in Moriyama's formula equal 1.

  It follows easily from the definitions that the space of Whittaker-$E$-rational classes is exactly $\Omega^W(\Pi) \cdot H^2(\Pif)$, for some nonzero constant $\Omega^W(\Pi) \in \CC^\times$, the \emph{Whittaker period} of $\Pi$, well-defined modulo $E^\times$.

  \begin{notation}
   For the remainder of this section, we resume the notations of \S\ref{sect:interpcupprod}, so $p > 2$ is an unramified prime, and $L$ the completion of $E$ at a prime above $p$ with respect to which $\Pi_p$ is Klingen ordinary.
  \end{notation}

 \subsection{The $\GSp_4$ $L$-function}

 \begin{proposition}
  \label{prop:Lfcn1}
  Let $R = \Lambda_L(\Zp^\times \times \Zp^\times)$. There is an element $\mathcal{L}_p(\Pi, \mathbf{j}_1, \mathbf{j}_2) \in R$ whose specialisation at a locally-algebraic point $x = (a_1 + \rho_1, a_2 + \rho_2)$ of $\Spec R$, such that $d = r_1 - r_2 \ge a_1 \ge a_2 \ge 0$ and $(-1)^{a_1 + a_2} \rho_1(-1) \rho_2(-1) = -1$, is given by
  \[
   \frac{R_p(\Pi, \rho_1, a_1) R_p(\Pi, \rho_2, a_2)
   \Lambda(\Pi \otimes \rho_1^{-1}, \tfrac{1-d}{2} + a_1) \Lambda(\Pi \otimes \rho_2^{-1}, \tfrac{1-d}{2} + a_2)}{\Omega^W(\Pi)},\]
  where
  \[ R_p(\Pi, a, \rho) =
  (-1)^a \cdot \begin{cases}
  \left(1 - \tfrac{p^{a+r_2 + 1}}{\alpha} \right) \left(1 - \tfrac{p^{a+r_2 + 1}}{\beta}\right) \left(1 - \frac{\gamma}{p^{a + r_2 + 2}} \right) \left(\frac{\delta}{p^{a + r_2 + 2}} \right) & \text{if $\rho$ is trivial},\\[2ex]
  G(\rho)^2 \left(\frac{p^{2a + 2r_2 + 2}}{\alpha\beta} \right)^{m} & \text{if $\rho$ has conductor $p^m > 1$},
  \end{cases}
  \]
  where $\alpha, \beta, \gamma, \delta$ are the roots of the Hecke polynomial such that $v_p(\alpha\beta) = r_2 + 1$, and $G(\rho)$ is the Gauss sum as in \S\ref{sect:dirichlet}.
 \end{proposition}

 \begin{proof}
  If we choose Schwartz functions $\Phi_1^{(p)}$ and $\Phi_2^{(p)}$ on $(\Af^p)^2$, then we can use these to build $p$-adic families $\mathcal{E}(\Phi_1^{(p)}, \Phi_2^{(p)})$ of Eisenstein series on $H$, as in \S\ref{sect:input1} above. If we also choose $\eta \in H^2(\Pif)$ which is invariant under $\Kl(p)$ and lands in the ordinary eigenspace, then we can form a $p$-adic analytic function interpolating the cup-products of $\eta$ with specialisations of the family $\mathcal{E}(\Phi_1^{(p)}, \Phi_2^{(p)})$.

  Specialising at $(a_1 + \rho_1, a_2 + \rho_2)$, the cup-product is given by a zeta-integral $Z(\dots)$, but with $\Pi$ replaced by the twisted representation $\Pi \otimes \rho_1^{-1}$. This factorises as a product of local integrals as above. For finite primes $\ell \ne p$ we may choose a finite linear combination of triples ($\phi_v, \Phi_{1, v}, \Phi_{2, v})$ such that the corresponding local zeta integral is the $L$-factor (and the rationality statements of \S \ref{sect:rationality} allow us to choose the data such that the resulting coherent cohomology classes are defined over $E$). For the place $\infty$, Moriyama's results show that the specific Archimedean data we have chosen give the Archimedean $\Gamma$-factor and the sign $(-1)^{a_1 - a_2}$. For the place $p$, we have chosen specific local data and evaluated the local integrals explicitly, which gives the local correction factors $R_p(\Pi, \rho_i, a_i)$.
 \end{proof}

 \begin{lemma}
  If $r_1 - r_2 \ge 1$, then for all finite-order characters $\rho$ of $\Zp^\times$, we have $\Lambda(\Pi \otimes \rho^{-1}, \tfrac{1+d}{2}) \ne 0$.
 \end{lemma}

 \begin{proof}
  If $r_1 - r_2 \ge 2$ this is obvious from the convergence of the Euler product. The case $r_1 - r_2 = 1$ follows from results of Jacquet--Shalika \cite{jacquetshalika76} on non-vanishing of $L$-functions for $\GL_n$ on the abcissa of convergence (the ``prime number theorem'' for $\GL_n$ $L$-functions), applied to the automorphic representation of $\GL_4$ lifted from $\Pi$.
 \end{proof}

 If $r_1 = r_2$ then we cannot directly establish the non-vanishing, so we must impose it as a hypothesis:

 \begin{hypothesis}
  \label{hyp:non-vanish}
  There exist finite-order characters $\rho_+, \rho_-$ of $\Zp^\times$, with $\rho_+(-1) = 1$ and $\rho_-(-1) = -1$, such that $\Lambda(\Pi \otimes \rho_+^{-1}, \tfrac{1}{2})$ and $\Lambda(\Pi \otimes \rho_-^{-1}, \tfrac{1}{2})$ are non-zero.
 \end{hypothesis}

 \begin{theorem}(Theorem A)
  Suppose that $r_1 - r_2 \ge 1$ or that Hypothesis \ref*{hyp:non-vanish} holds. Then there exist constants $\Omega^+_{\Pi}$ and $\Omega^-_\Pi$, uniquely determined modulo $E^\times$ and satisfying $\Omega^+_\Pi \Omega^-_\Pi = \Omega_\Pi^W \pmod{E^\times}$; and an element $\mathcal{L}(\Pi, \mathbf{j}) \in \Lambda_L(\Zp^\times)$, such that
  \[ \mathcal{L}(\Pi, a + \rho) = (-1)^a R_p(\Pi, \rho, a) \frac{\Lambda(\Pi \otimes \rho^{-1}, \tfrac{1-d}{2} + a)}{\Omega_{\Pi}^{\pm}} \]
  for all $(a, \rho)$ with $0 \le a \le d$, where the sign $\pm$ is given by $(-1)^a \rho(-1)$.
 \end{theorem}

 \begin{proof}
  Let $\rho_+$ and $\rho_-$ be any two Dirichlet characters of $p$-power conductor such that $\rho_{+}(-1) = (-1)^d$, $\rho^-(-1) = -(-1)^d$. From the preceding lemma, we see that $\Lambda(\Pi \otimes \rho_{+}^{-1}, \tfrac{1+d}{2})$ and $\Lambda(\Pi \otimes \rho_{-}^{-1}, \tfrac{1+d}{2})$ are both nonzero. If $d = 1$ then we assume also that both characters $\rho_{\pm}$ are ramified at $p$, which ensures that the local Euler factors $R(\rho_{\pm}, d)$ are non-zero.

  Then we may define
  \begin{align*}
  \Omega^{+}(\Pi) &= \Omega^W(\Pi) / [ R_p(\Pi, \rho_-, d)\Lambda(\Pi \otimes \rho_-^{-1}, \tfrac{1+d}{2}) ] \\
  \Omega^-(\Pi) &= \Omega^W(\Pi) / [ R_p(\Pi, \rho_+, d)\Lambda(\Pi \otimes \rho_+^{-1}, \tfrac{1+d}{2})].
  \end{align*}
  We can then define $\mathcal{L}(\Pi, \mathbf{j}) = \mathcal{L}(\Pi, \mathbf{j}, d + \rho^-)$ on the component of weight space where $(-1)^\mathbf{j} = +1$, and $\mathcal{L}(\Pi, \mathbf{j}, d + \rho^+)$ on the other component; and the interpolation formula follows easily from the corresponding property of the two-variable $L$-function $\mathcal{L}(\Pi, \mathbf{j}_1, \mathbf{j}_2)$.
 \end{proof}

 \subsection{The $\GSp_4 \times \GL_2$ $L$-function}

   A similar analysis for the $\GSp_4 \times \GL_2$ integral formula gives the following:

  \begin{theorem}(Theorem B)
   Let $R = \Lambda_L(\Zp^\times)$. There is an element $\mathcal{L}_p(\Pi \times \sigma, \mathbf{j}) \in R$ whose specialisation at a locally-algebraic point $x = a + \rho$ of $\Spec R$, with $0 \le a \le d' = r_1 - r_2 - \ell + 1$, is given by
   \[  R_p(\Pi \times \sigma, a, \rho) \frac{\Lambda^{S}(\Pi \otimes \sigma \otimes \rho^{-1}, \tfrac{1-d'}{2} + a)}{\Omega_\Pi^W},\]
   where $S$ is the set of finite places $\ell$ such that neither $\Pi_\ell$ nor $\sigma_\ell$ is principal series, and
   \[
    R_p(\Pi \times \sigma, a, \rho) = \frac{1}{\epsilon(\tau_p \otimes \sigma_p \otimes \widehat{\rho}_p^{-1}, \tfrac{1-d'}{2} + a)} L(\tau_p^\vee \otimes \sigma_p^\vee \otimes \widehat{\rho}_p, \tfrac{1+d'}{2} - a)^{-1} L(\tau_p \otimes \sigma_p \otimes \theta_p \widehat{\rho}_p^{-1}, \tfrac{1-d'}{2} + a)^{-1}
   \]
   where we have written $\Pi_p = \Ind_{P_{\Kl}}^{G}(\tau_p \boxtimes \theta_p)$ for some $\tau_p$ and $\theta_p$ as in \S \ref{sect:zeta-p}.\qed
  \end{theorem}

  (The factor $R_p(\Pi \times \sigma, a, \rho)$ can also be written explicitly in terms of Satake parameters and Gauss sums, as in Proposition \ref{prop:Lfcn1}; we leave the details to the reader.)

 \begin{remark}\label{rmk:intperiods}
  It would be interesting to attempt to define a canonical normalisation for the periods $\Omega^{\mathrm{W}}(\Pi)$ and $\Omega^{\pm}(\Pi)$ up to $p$-adic units, analogous to the results of \cite{vatsal99} in the $\GL_2$ case.

  For $\Omega^{\mathrm{W}}(\Pi)$, one possibility would be to use the $\GSp(4)$ new-vector theory of Okazaki \cite{okazaki}. This singles out a uniquely-determined level group $K \subset G(\Af)$ such that $\Pif^K$ is one-dimensional, and such that $\cW(\Pif)^K$ has a canonical normalised generator $w^{\mathrm{new}}$ satisfying $w^{\mathrm{new}}(1) = 1$. We then have two natural $\cO_{E, (v)}$-lattices in $\cW(\Pif)^K \otimes \CC$: one given by Whittaker functions whose values at 1 are integral, and one given by the image under the comparison map of the lattice in $H^2(\Pif)^K$ generated by the cohomology of the $\Zp$-model $X_G(K)$ (or the $K$-invariants of cohomology at a suitable higher level, since this $K$ may not be neat). Using the ratio of these two lattices, we obtain a canonical normalisation for $\Omega^{\mathrm{W}}(\Pi)$ up to multiplication by $\cO_{E, (v)}^\times$.

  However, since we do not know if the test data for the zeta-integrals $Z_\ell(\dots, s)$ at the bad primes $\ell$ can be chosen ``integrally'', even with this normalised period it is not clear if the $p$-adic $L$-functions of Proposition \ref{prop:Lfcn1} and Theorem B will lie in the Iwasawa algebra with $\cO$-coefficients.

  The situation for the signed periods $\Omega^{\pm}(\Pi)$ is still more mysterious, and there seems to be no natural way of defining a canonical normalisation for these periods directly on $\GSp_4$ (rather than via functorial lifting to $\GL_4$ as in \cite{DJR20}).
 \end{remark}


\newlength{\bibitemsep}
\setlength{\bibitemsep}{0.7ex plus 0.05ex minus 0.05ex}
\newlength{\bibparskip}
\setlength{\bibparskip}{0pt}
\let\oldthebibliography\thebibliography
\renewcommand\thebibliography[1]{%
 \oldthebibliography{#1}%
 \setlength{\parskip}{\bibparskip}%
 \setlength{\itemsep}{\bibitemsep}%
}

\providecommand{\bysame}{\leavevmode\hbox to3em{\hrulefill}\thinspace}
\providecommand{\MR}[1]{}
\renewcommand{\MR}[1]{%
 MR \href{http://www.ams.org/mathscinet-getitem?mr=#1}{#1}.
}
\providecommand{\href}[2]{#2}
\newcommand{\articlehref}[2]{\href{#1}{#2}}

\end{document}